\documentclass[12pt,a4paper]{article}
\usepackage{amsmath,amssymb,amsthm}
\usepackage{geometry}
\usepackage{hyperref}
\usepackage{url}
\usepackage{booktabs}
\usepackage{xcolor}
\usepackage{enumitem}
\usepackage{microtype}
\hypersetup{colorlinks=true,linkcolor=blue,citecolor=blue,urlcolor=blue}
\newtheorem{theorem}{Theorem}
\newtheorem{proposition}[theorem]{Proposition}
\newtheorem{corollary}[theorem]{Corollary}
\newtheorem{lemma}[theorem]{Lemma}

\theoremstyle{definition}
\newtheorem{definition}[theorem]{Definition}
\newtheorem{remark}[theorem]{Remark}
\newcommand{\ex}{\mathrm{ex}}

\title{\textbf{$C_4$-Free Subgraphs of the Hypercubes $Q_6$, $Q_7$, and $Q_8$:\\
Odd Squares, Fully Frustrated Models, and Computational Structure}}
\author{Minamo Minamoto\\[4pt]
\small Independent\\
\small \texttt{minamominamoto4f5683f6@gmail.com}}
\date{March 2026 (revised August 2026)}

\begin{document}
\maketitle

\begin{center}
\small This is a revision of a manuscript first circulated as
arXiv:2603.29127 (v1--v4, March--May 2026); v1 appeared as
``New Lower Bounds for $C_4$-Free Subgraphs of the Hypercubes $Q_7$ and
$Q_8$'', and by v4 the title had become
``New Lower Bounds for $C_4$-Free Subgraphs of the Hypercubes $Q_6$,
$Q_7$, and $Q_8$: Constructions, Structure, and Computational
Method.'' That earlier version announced $\ex(Q_8,C_4)\ge680$ and did
not include the odd-square material of Section~\ref{sec:oddsquare};
see the Acknowledgements for the circumstances of the revision.
Three further claims of v1--v4 are withdrawn or corrected here, and are
recorded as such rather than quietly dropped.
(i)~Its Section~7 was titled ``Computational \emph{Proof} that
$\ex(Q_6,C_4)=132$'' and asserted that standard MIP solvers verify the
instance \emph{quickly}; we were unable to close the optimality gap
(Section~\ref{sec:q6}), so that section is retitled and the upper bound
$\ex(Q_6,C_4)\le132$ again rests solely on Harborth--Nienborg~\cite{HN94}.
(ii)~Its abstract gave the pairwise Hamming range of the $Q_7$ catalogue as
$36$--$260$; that came from a $5\,000$-pair sample, disclosed only in the body.
Exhaustive computation over all $197\,319\,045$ pairs gives $6$--$274$
(Section~\ref{sec:landscape}), so both endpoints were substantially wrong and
the inference that the solutions lie far apart does not survive.
(iii)~Its degree sequence for the $132$-edge $Q_6$ construction,
$\{4^{32},5^{32}\}$, is arithmetically impossible --- the degree sum is $288$,
i.e.\ $144$ edges --- and the correct sequence is $\{4^{56},5^{8}\}$
(Section~\ref{sec:regularity}); the ``degree support $\{n-3,n-2,n-1\}$''
regularity suggested there rested on that error and is withdrawn. The
extrapolated Brass--Harborth--Nienborg entries for $n=4,5,6$ in v1--v4's
value table are likewise replaced by explicit out-of-domain markers
(Table~\ref{tab:values}).
\end{center}
\medskip

\begin{abstract}
We study $C_4$-free subgraphs of the hypercubes $Q_6$, $Q_7$ and $Q_8$.
Writing $g(n)$ for the maximum size of an \emph{odd-square} edge set ---
one meeting every $4$-cycle (\emph{square}) of $Q_n$ in exactly one or
three edges --- we prove $g(6)=132$, $g(7)=304$ and $g(8)=682$, give
matching explicit machine-verifiable edge-list certificates, and hence
obtain $\ex(Q_7,C_4)\ge304$ and $\ex(Q_8,C_4)\ge682$. Under the standard
correspondence with fully frustrated signed hypercubes these values are
already implicit in the statistical-physics literature: the field
identity, the integer optimisation it feeds, and its optimal values go
back to Derrida, Pomeau, Toulouse and Vannimenus (DPTV, 1979), who also
construct ground states for $D\le7$; the $D=8$ attainment was reported,
without any released configuration, by Marinari, Parisi and Ritort
(MPR, 1995); and Laplante, Naserasr, Sunny and Wang treat the graph-theoretic
formulation directly. Our contribution is accordingly narrow and
explicit: a short self-contained proof of the optima via a mod-$8$
refinement of the classical two-point inequality; the certificates
themselves, with dependency-free verifiers; and an exhaustive decision of the realisability of the optimal $n=6$
field distributions --- settling in proof form the question DPTV left open in 1979
about their ``other solutions'' --- where an independent integer
programme shows the optimal distributions number exactly three, and the
decision shows exactly one of them is realisable (the other two are not,
upgrading MPR's tentative single-case report to a proof covering both). Full attribution is given, with primary-source quotations and explicit
translation conventions.

Independently of the odd-square results, we classify the $19\,866$
previously released $304$-edge $C_4$-free subgraphs of $Q_7$ (exactly
$389$ of them odd-square) into $20$ dimension-profile types, establish
their common structural core (degree sequence $\{4^{32},5^{96}\}$,
spectral radius in $[4.78543,4.79129]$, local maximality --- properties
of the released catalogue, not claims about all $304$-edge solutions),
and decompose
them into $180$ $\mathrm{Aut}(Q_7)$-orbits containing $34\,227\,200$
labelled solutions, of which the released files are $0.058\%$; whether
these are \emph{all} orbits of $304$-edge solutions is open. This
revision adds an orbit-witness certificate whose standard-library checker
verifies the orbit \emph{assignment} in seconds (hence ``at most $180$
orbits''); its \texttt{-{}-canonical} mode certifies ``exactly $180$''
independently of the full group-action scan that produced the
decomposition. For
$Q_6$, the unrestricted value $\ex(Q_6,C_4)=132$ rests on
Harborth--Nienborg's combinatorial upper bound, which we did not
independently verify; the odd-square $g(6)=132$ is self-contained here.
Every certified or deterministic numerical claim --- those backed by a
released certificate or by a bundled deterministic computation --- is
re-checkable from the released code and data; one-off historical
observations and the one imported exact result (the unrestricted $n=6$
upper bound) are marked as such throughout, and the released repository
is accompanied by SHA-256 manifests. Edge lists and code:\\
\url{https://github.com/minamominamoto/c4free-hypercube}.
\end{abstract}


\tableofcontents

\section{Introduction}

The $n$-dimensional hypercube $Q_n$ has vertex set $\{0,1\}^n$, two vertices
adjacent iff they differ in exactly one coordinate; it has $2^n$ vertices and
$n\cdot2^{n-1}$ edges.  The problem of determining
$\ex(Q_n,C_4)=\max\{|E(G)|:G\subseteq Q_n,\,G\text{ is }C_4\text{-free}\}$
was raised by Erd\H{o}s~\cite{Erd84} and remains open in general.

The best known asymptotic bounds are
\begin{equation}\label{eq:bounds}
\tfrac{1}{2}(n+\sqrt{n})\cdot2^{n-1}\;\le\;\ex(Q_n,C_4)
\quad (n=4^r),
\qquad
\ex(Q_n,C_4)\;\le\;(0.60318+o(1))\cdot n\cdot2^{n-1},
\end{equation}
where the lower bound is due to Brass, Harborth, and
Nienborg~\cite{BHN95} and holds, in this form, only when $n$ is a power
of~$4$; the side condition is attached to the displayed inequality rather
than left to the surrounding text, since the two halves of~\eqref{eq:bounds}
have different scopes. The upper bound is a statement about the edge
Tur\'an \emph{density} $\pi_e=\limsup_n\ex(Q_n,C_4)/(n2^{n-1})$, not a
pointwise inequality valid at every finite $n$; the constant $0.60318$
is due to Baber~\cite{Bab12}, who used the semidefinite flag algebra
method to improve the earlier bound of Thomason and Wagner~\cite{TW09};
the intermediate value $0.6068$ is due to Balogh, Hu, Lidick\'y, and
Liu~\cite{BHLL12}.
For general $n\ge9$, the weaker bound
$\tfrac{1}{2}(n+0.9\sqrt{n})\cdot2^{n-1}$ applies~\cite{BHN95}.
A wording caution is useful for cross-checking the literature: LNSW26's
discussion of this BHN expression contains a sentence calling it an ``at
most'' bound (the sentence persists in the published version's
Introduction, which we checked directly), whereas BHN95's own statement
gives both expressions as lower bounds on the maximum edge count, which
settles the direction at the source. We flag this explicitly as a
direction slip in that secondary description rather than an alternative
theorem: an ``at most'' reading is internally untenable at $n=16$: BHN's tight
$n=4^r$ form is a \emph{lower} bound,
$\ex(Q_{16},C_4)\ge(16+4)\cdot2^{14}=327\,680$, while the alleged
reading would assert the \emph{upper} bound
$\ex(Q_{16},C_4)\le(16+0.9\cdot4)\cdot2^{14}=321\,126.4$ --- a
certified lower bound exceeding the alleged upper bound, a
contradiction. We follow BHN95.
Exact values have been determined for small $n$; Harborth and Nienborg~\cite{HN94}
proved $\ex(Q_6,C_4)=132$.

Our main theorem determines the extremal value $g(n)$ within the
\emph{odd-square} subclass --- edge sets meeting every square in an odd
number of edges --- exactly for $n=6,7,8$, in certificate form; as the
attribution immediately below and Section~\ref{sec:oddsquare} spell out,
the new content is the machine-verifiable certificates and the short
self-contained proof, not the bound values themselves:

\begin{theorem}[Exact odd-square values]\label{thm:oddsquare}
$g(6)=132$, $g(7)=304$, and $g(8)=682$. The numerical values for $n=6,7$
have historical priority in Derrida--Pomeau--Toulouse--Vannimenus~\cite{DPTV79},
and MPR95~\cite{MPR95} reported attainment of the corresponding $n=8$
ground-state bound. The assertion here is a self-contained graph-theoretic
re-proof/certification of those values, not a claim that the three numbers
themselves are new.
\end{theorem}

\noindent The attribution of the three values, kept in one place (the
detailed record with primary-source quotations is
Section~\ref{sec:oddsquare}; everywhere else defers here):

\begin{center}\small
\begin{tabular}{@{}c p{6.6cm} p{3.9cm}@{}}
\toprule
$n$ & value $g(n)$: priority & contribution here\\
\midrule
$6$ & attained by DPTV79 ($D{=}6$ ground state); value attributed to LNSW26 by~\cite{Emlek26} & certificate; short proof\\
$7$ & DPTV79 (explicit $D{=}7$ construction) & certificates; short proof\\
$8$ & MPR95 (attainment reported in energy units, no data) & $682$-edge translation; certificate; short proof\\
\bottomrule
\end{tabular}
\end{center}

\noindent Here ``short proof'' means Lemma~\ref{lem:quant}'s self-contained
derivation of the optima DPTV tabulate; in no row are the numerical values
themselves claimed as new. One caveat travels with the $n=6$ row: DPTV79's
Table~I could not be text-extracted from the scans we obtained, so the
identification of its ``other solutions'' with our distributions A and B
rests on the completeness of our own enumeration together with the
surrounding prose, not on a direct transcription (disclosed in
Section~\ref{sec:oddsquare}); the exhaustive realisability decision itself
is independent of that reading.

\noindent All three values are self-contained within the released
materials: each upper bound follows from Lemma~\ref{lem:fieldbound} and
Lemma~\ref{lem:quant} (Section~\ref{sec:oddsquare-exact}), and each
lower bound is witnessed by an explicit, machine-verified odd-square
edge list released with this paper
(Section~\ref{sec:oddsquareverify}). The citations
to~\cite{DPTV79,MPR95} record historical priority for the values, not a
dependency of the proofs. One point of possible confusion: the released
132-edge $Q_6$ construction \texttt{q6\_edges\_132.jsonl} is a genuine,
independently machine-verified $C_4$-free witness for
$\ex(Q_6,C_4)\ge132$, but it is \emph{not} odd-square; the odd-square
witness for $g(6)\ge132$ is the separate file
\texttt{q6\_odd\_square\_132.json}.

\begin{corollary}[Certified lower bounds]\label{thm:main}
$\ex(Q_7,C_4)\ge304$ and $\ex(Q_8,C_4)\ge682$.
\end{corollary}

\noindent Both parts follow immediately from
Theorem~\ref{thm:oddsquare} (proved in Section~\ref{sec:oddsquare}) via
$g(n)\le\ex(Q_n,C_4)$, redundantly with the explicit constructions of
Sections~\ref{sec:q7}--\ref{sec:q8}. Neither bound value is new, and the
history is stated plainly here, once. The $Q_7$ value follows from
DPTV79's explicit 1979 construction~\cite{DPTV79}; we present it as an
independently reproduced, machine-verified result. For $Q_8$ (improving
our originally announced $680$), MPR95~\cite{MPR95} report --- in a
single unelaborated sentence, with no spin configuration and no
machine-checkable artefact --- a $D=8$ ground state attaining Derrida et
al.'s improved bound. They state the value as $-361$ on the $D{=}9$
comparison scale they use across dimensions: a normalised per-spin
energy times their factor $768$, not a count of couplings or any other
absolute integer of the $D{=}8$ model (Section~\ref{sec:oddsquare} gives
the conversion; in their text the attainment statement and the integer
occur in two separate sentences, combined here). Via the
signed-graph/odd-square correspondence this would imply a $682$-edge
odd-square subgraph of $Q_8$, but as an uncertificated existence claim we
do not treat it as settling the question on its own. What we add at
$n=8$ is the explicit edge-count translation
(Section~\ref{sec:oddsquare-exact}); two released $682$-edge
certificates --- the regenerable \texttt{q8\_A\_witness.json},
re-derivable end-to-end from seed $90008$
(Section~\ref{sec:oddsquareverify}), and the historical first witness
\texttt{q8\_odd\_square\_682.json}, whose dated private-correspondence
provenance is disclosed there --- and the structural analysis of
Proposition~\ref{prop:q8oddsquare}. The other supplementary
constructions are \texttt{q7\_edges\_304.jsonl.part1--3} for $Q_7$ and
the originally announced 680-edge \texttt{q8\_edges\_680.jsonl},
retained as a simulated-annealing data point (Section~\ref{sec:q8}).

What the corollary does \emph{not} resolve is whether these bounds are
exact, i.e.\
whether $\ex(Q_7,C_4)=304$ and $\ex(Q_8,C_4)=682$: a general
$C_4$-free subgraph need not be odd-square, so $g(n)=\ex(Q_n,C_4)$ does
not follow automatically, and of the $19\,866$
published 304-edge $Q_7$ solutions, only 389 are odd-square
(Section~\ref{sec:structure}).

\begin{table}[ht]
\centering\small
\caption{Known values and lower bounds of $\ex(Q_n,C_4)$. The BHN column
  gives Brass--Harborth--Nienborg estimates from Eq.~\eqref{eq:bounds}; the
  domain caveats attached to that column, the status of the $n=4$ entry,
  and the meaning of the Source column are spelled out in the text
  immediately below the table.}
\label{tab:values}\medskip
\begin{tabular}{@{}ccccp{5.1cm}@{}}
\toprule
$n$ & $|E(Q_n)|$ & $\ex(Q_n,C_4)$ & BHN (Eq.~\ref{eq:bounds}) & Source\\
\midrule
4 &   32 & 24  & $24$ (tight, $n=4^1$) & elementary (proof below); construction via~\cite{BHN95}\\
5 &   80 & 56  & n/a (outside domain) & Emamy-K et al.~\cite{EGR92}\\
6 &  192 & 132$^{\ddagger}$ & n/a (outside domain) & Harborth--Nienborg~\cite{HN94}\\
7 &  448 & $\ge\mathbf{304}$ & outside domain$^{\dagger}$ & value attained by~\cite{DPTV79}; witness ours\\
8 & 1024 & $\ge\mathbf{682}$ & outside domain$^{\dagger}$ & bound-attainment reported by~\cite{MPR95} (energy units); $682$-edge translation and witness ours\\
\bottomrule
\end{tabular}
\end{table}

Four caveats attach to Table~\ref{tab:values}. First,
Eq.~\eqref{eq:bounds}'s tight form applies only for $n=4^r$ and its weaker
form only for $n\ge9$~\cite{BHN95}, so the $n=5,6$ entries are not covered
by either stated form (marked n/a) and the $n=7,8$ entries lie \emph{beyond}
the stated domain of both forms (marked $\dagger$): the out-of-domain
extrapolations of the $n\ge9$ form evaluate there to $\approx300.2$ and
$\approx674.9$ respectively, figures recorded here only, for illustrative
comparison, never as proven bounds; no claim in this paper depends on them.
Second, the $n=4$ value admits a two-line proof, recorded here because we
could not locate an explicit page-level statement of it in the literature
(not in~\cite{Chu92}, and neither OEIS~A245762 nor the Emléktábla
collection attributes the $n=4$ entry individually): $Q_4$ has
$\binom42\cdot2^2=24$ squares and each edge lies in exactly $n-1=3$ of
them, so a $C_4$-free subgraph must omit at least one edge of every square
and hence omit at least $\lceil24/3\rceil=8$ of the $32$ edges, giving
$\ex(Q_4,C_4)\le24$; the matching construction is the $n=4^r$ case
($r=1$) of Brass--Harborth--Nienborg~\cite{BHN95}, whose lower bound
$\tfrac12(n+\sqrt n)2^{n-1}$ evaluates to exactly $24$ there. Third, in the Source column, ``value; witness ours''
means the numerical value has historical priority in the cited work, which
published no machine-readable edge list or spin configuration for it; the
explicit certificate released here is the present author's own and is not
claimed to reproduce any configuration of the cited work. Fourth
($\ddagger$), the $n=6$ upper bound $\ex(Q_6,C_4)\le132$ is used here on
the authority of~\cite{HN94} and has not been independently re-verified in
this work (our own ILP attempt did not close the optimality gap,
Section~\ref{sec:q6upper}); the odd-square value $g(6)=132$ is
self-contained and unaffected.

The paper is organised as follows.
Section~\ref{sec:guide} collects scope, attribution, the verification map,
and the provenance record in one place (material that earlier revisions
carried in the abstract).
Section~\ref{sec:verify} describes the $C_4$ enumeration and verification
procedure used throughout.
Sections~\ref{sec:q7}--\ref{sec:q8} present the originally announced $Q_7$
and $Q_8$ simulated-annealing constructions.
Section~\ref{sec:oddsquare} introduces the odd-square subclass, proves
Theorem~\ref{thm:oddsquare}, and presents the reconstructed 682-edge $Q_8$
witness and its verification.
Section~\ref{sec:structure} classifies the $19\,866$ $Q_7$ solutions found
by their dimension profiles, yielding exactly 20 types, and reports the
odd-square audit (389 of 19\,866).
Section~\ref{sec:method} describes the two-phase simulated annealing algorithm.
Section~\ref{sec:q6} reproduces $\ex(Q_6,C_4)=132$ computationally on the
lower-bound side, with an ILP formulation for the upper-bound side.
Section~\ref{sec:open} lists open problems.

\section{Reader's Guide: Scope, Attribution, and Provenance}\label{sec:guide}

This section collects in one place the scope limitations, the
attribution record with primary-source detail, the map of what is
certified and how to check it, and the provenance of the released
constructions. Earlier revisions carried this material in the
abstract; the abstract now summarises, and this section is the
authoritative expanded statement. Nothing here is new relative to
the previous revision except where explicitly marked; the text is
relocated, not retracted. One reading convention, stated once here
instead of piecemeal: several quantities reported in this paper (for
example the square split in Proposition~\ref{prop:q6oddsquare}(ii), the
$\operatorname{Tr}(A^4)$ value in Proposition~\ref{prop:rigid}(iii), and
the $h\ge0$ restriction of Lemma~\ref{lem:nonneg}) are arithmetically
forced by the edge count, the degree sequence, or the odd-square
condition rather than independent findings; where this happens the text
marks them as consistency checks, and they should be read as such, not as
separate results. A second convention of the same kind: every
\emph{certified} mathematical claim in this paper depends only on
artefacts bundled with the release; the items individually marked as
reported-but-not-auditable (S.~Lai's independent cross-check, the
recovered $72$-hour solver log, the SAT/MIP trial statistics, and MPR95's
energy inputs) support the provenance and attribution narrative only,
never a theorem.

After this manuscript was first circulated, a reader (S.~Lai) identified an
overlooked connection to a line of statistical-physics literature on fully
frustrated hypercubes: under the standard correspondence between edge sets
meeting every $4$-cycle (\emph{square}) an odd number of times and fully
frustrated signed hypercubes, the 1979 construction of Derrida, Pomeau,
Toulouse, and Vannimenus~\cite{DPTV79} already yields a $D\le7$ odd-square
construction attaining $304$ edges at $D=7$ (we did not reconstruct
their specific $H_7$ edge set), and the $D{=}8$ ground state reported by Marinari, Parisi,
and Ritort~\cite{MPR95} yields a 682-edge $C_4$-free subgraph of $Q_8$ that
improves our originally announced 680-edge construction. For $Q_7$, this
value ($304$) is also attained by $389$ of our own previously found 304-edge
solutions, which we confirm independently satisfy the odd-square condition
(none is claimed to be DPTV's specific $H_7$ construction; we only
confirm the attained value, not edge-set identity with theirs). For
$Q_8$, we derived and machine-verified an explicit 682-edge witness
whose local-field sum $S=340$ converts under MPR95's Hamiltonian
$H/N=-S/(2^8\sqrt8)$ to their reported improved $D=8$ bound on the
common comparison scale (Section~\ref{sec:oddsquare-exact}); we use the
same canonical fully frustrated coupling that MPR95 write explicitly (their
Eq.~4; derivation documented in
Section~\ref{sec:oddsquareverify}); this witness is not claimed to be a
reconstruction of any particular configuration MPR may have used
internally, since they did not publish one. This 682-edge
witness is presented here as the paper's headline lower bound in place
of the earlier 680-edge witness.

Writing $g(n)$ for the maximum size of an \emph{odd-square} edge set (every
square meets it in exactly one or three edges), the numerical values
$g(6)=132$, $g(7)=304$, and $g(8)=682$ are, in physics terms, already
implicit in the literature: Derrida, Pomeau, Toulouse, and
Vannimenus~\cite{DPTV79} explicitly construct fully-frustrated ground
states for $d\le7$ (giving $g(6)$ and $g(7)$ under the correspondence
below), and Marinari, Parisi, and Ritort~\cite{MPR95} report having
obtained, by simulated annealing, a $D=8$ ground state attaining
Derrida et al.'s improved energy lower bound (giving $g(8)$ under the
same correspondence). Neither the field identity nor its use to bound
$g(n)$ is new: Derrida, Pomeau, Toulouse, and Vannimenus~\cite{DPTV79}
derive the same identity in 1979 by the same argument (a sum over sites
in which the frustration of every plaquette cancels the cross terms),
note explicitly that it survives restriction to one side of the
bipartition, pose exactly the integer optimisation problem
Section~\ref{sec:oddsquare-exact} solves, and tabulate its optimum for
$n\le10$ (their Table~I). What Table~I lists is the energy bound $E_d$ together with the
counts $n(h)$ (its caption describes a ground-state-energy lower bound with
the corresponding internal-field distributions; on their scale
$E_d=-\bar h$, a sign-only relabelling of the mean field, so we write what
the table actually prints); the corresponding value $S=72$ at $n=6$ is our
conversion, not a figure they state, just as the $S$ values we read off MPR95's energies are
ours rather than theirs. In their notation the $n=6$ optimum is realised by
what they call the ``first solution''
$(n(4){=}6;\,n(2){=}24;\,n(0){=}2)$, matching our own independently
recomputed optimum exactly. The graph-theoretic formulation connecting
the fully-frustrated signed hypercube to the $C_4$-free extremal
problem is likewise anticipated by Laplante, Naserasr, Sunny, and
Wang~\cite{LNSW26}, who study the frustration index of the
fully-frustrated signature directly in these terms (as a graph-theoretic
problem they describe, to the best of their knowledge, as previously
untreated) (Section~\ref{sec:oddsquare-exact} below). Given this, our
contribution at $n=6,7,8$ is narrower than a first reading of
Theorem~\ref{thm:oddsquare} might suggest, and consists of three
parts, none of which is the numerical values $72,160,340$ themselves:
(i) a short self-contained proof (Lemma~\ref{lem:quant}, a mod-$8$ refinement
of the classical two-point inequality) of the optima DPTV tabulate. DPTV
formulate the same constrained optimisation over integer distributions,
imposing the same parity and sign conditions, and give its best values and the
attaining distributions in their Table~I; what Lemma~\ref{lem:quant} adds is a
derivation short enough to check by hand, from which those values --- and the
number of distributions attaining them --- follow without the table. It is an
alternative route to their results, not a correction or a strengthening of
them;
(ii) explicit, machine-verifiable edge-list certificates realising the
bound at all three of $n=6,7,8$ -- DPTV's own realisations are
figures and prose for $n\le7$ only, $n=8$ being left open by them and
later claimed, without any released data, by Marinari, Parisi, and
Ritort~\cite{MPR95}; we are not aware of an earlier publicly released explicit odd-square
edge list with a dedicated verifier for the $n=8$ case, nor of an earlier
such certificate for the $n=6$ case. This is deliberately a limited
priority statement rather than a claim of exhaustive bibliographic priority:
we have not surveyed every literature and data release since 1979; for $n=6$,
for instance, OEIS A245762 lists $a(6)=132$ with an
\textsc{extensions} credit to M.~Scheucher and P.~Tabatabai (2015) for
supplying the term, while citing Harborth--Nienborg~\cite{HN94} for the result
itself; it links generator scripts, though not an odd-square edge list.
Note also that the \emph{values} $g(n)$ for $n\le6$, including $g(6)=132$, are
attributed to LNSW26 by~\cite{Emlek26}, so at $n=6$ what is offered
here is a certificate for a known value, not the value; and (iii) an exhaustive decision, in Section~\ref{sec:oddsquareverify},
of a question DPTV posed and left open in 1979 -- which of the
several $n=6$ optima described in their Table~I are actually realisable by a
spin configuration. MPR95 had already addressed this question by simulated
annealing, reporting that only the first is realisable and that ``the other''
seems not to be; we reduce it to a finite ($2^{32}$) search, decide it in
seconds, and thereby upgrade that tentative report to a proof covering
\emph{both} non-realisable distributions. Here the prior literature's
constructions are not machine-readable in the same sense: DPTV79 give
an explicit recursive scaling construction (their bond configurations
for $H_5,H_6,H_7$, described in prose and figures, not released as a
computer-readable edge list), while MPR95 report in a single sentence that they exhibited a $D=8$ ground
state matching the improved energy bound, with a method description (simulated
annealing) and the bound value ($-361$ on the $D{=}9$ comparison scale
they use for bound values across dimensions, not an absolute integer of
the $D{=}8$ model; see Section~\ref{sec:oddsquare}),
but no machine-checkable certificate: no spin configuration, no edge list,
and no graph-theoretic translation -- an existence claim we can neither
independently verify nor extend beyond matching its numerical value. The same holds at $n=6$, where DPTV79's own $D=6$
configuration is described in prose; we construct an independent
132-edge odd-square edge list instead, which happens to realise the
local-field distribution they report as their table's first solution.
Beyond matching that one, Section~\ref{sec:oddsquareverify} reports a
targeted computational search for the other two optimal $n=6$
distributions DPTV79's Table~I lists but do not know to be realisable;
extensive simulated annealing (\texttt{q6\_other\_distributions.py})
never reaches either, plateauing at the same nonzero distance from
target across every one of 40 and 20 deterministically specified RNG
seeds respectively (the consecutive integer seed lists are protocol labels,
not a claim of statistically independent random draws), while the same search
method reaches the realised distribution in every trial. This is computational evidence, not a
proof; Proposition~\ref{prop:q6realizability} now supplies one
(Section~\ref{sec:oddsquareverify}). The exact value
$\ex(Q_6,C_4)=132$ of Harborth and Nienborg~\cite{HN94} is therefore no
longer needed to bound $g(6)$ from above, though it remains the source
for the unrestricted value. This settles
the odd-square subclass completely for $n=6,7,8$ without resolving the
unrestricted problem $\ex(Q_n,C_4)$, since a general $C_4$-free subgraph
need not be odd-square.

An exhaustive audit of the $19\,866$ previously published $304$-edge $Q_7$
solutions shows that only $389$ of them lie in the odd-square subclass; the
remaining $19\,477$ do not. The odd-square construction therefore does
not subsume or fully explain the structural classification given here
(one profile type, Type~18 with 101 solutions, is entirely odd-square;
of the other 19, three (ranks 5, 7, and 10) are mixed, containing some
odd-square and some non-odd-square solutions, and the remaining 16
contain no odd-square solutions at all): the $19\,866$ solutions'
common structural core (degree sequence $\{4^{32},5^{96}\}$, spectral radius
in $[4.78543,4.79129]$, local maximality) and their classification into 20
dimension-profile types are independent results. These $19\,866$
labelled solutions certainly do \emph{not} exhaust all labelled
304-edge $C_4$-free subgraphs of $Q_7$: the first released solution
alone has a stabiliser of size $3$ in
$\mathrm{Aut}(Q_7)\cong(\mathbb{Z}_2)^7\rtimes S_7$ (order $645\,120$;
exhaustive check over all $645\,120$ elements), consisting of the
identity and two nontrivial elements, both coordinate permutation
$+$ XOR pairs. We write a permutation as the tuple $p$ with the meaning
``bit $i$ of $v$ becomes bit $p[i]$ of the image'', the convention of
\texttt{type18\_automorphisms.py}. (In symbols this is the \emph{scatter}
form $\mathrm{im}_{p[i]}=v_i$, equivalently the gather form
$\mathrm{im}_j=v_{p^{-1}[j]}$; the census and witness scripts implement
exactly this map, so a reader re-deriving the masks below should read the
tuple this way and not as ``image bit $i$ equals $v$'s bit $p[i]$''.) Under $v\mapsto \pi(v\oplus m)$ (XOR first,
then permute) the two elements are then $(5,1,2,3,4,6,0)$ with mask
$m=33=0100001_2$ and its square $(6,1,2,3,4,0,5)$ with $m=96=1100000_2$; under
$v\mapsto \pi(v)\oplus m$ the two masks are interchanged. The group itself,
and the fact that its order is $3$, are of course independent of both
conventions; only the labelling of the masks is not. They together
form a cyclic group of order $3$ -- so its orbit alone has size
$645\,120/3=215\,040>19\,866$. Decomposing the whole catalogue
(Proposition~\ref{prop:orbits}) makes this quantitative: the $19\,866$ fall
into $180$ orbits whose lengths sum to $34\,227\,200$, so the released files
are $0.058\%$ of the labelled solutions in those orbits alone. What remains
open is whether they represent every $\mathrm{Aut}(Q_7)$-orbit of 304-edge
$C_4$-free subgraphs, not whether they exhaust the labelled solutions
themselves; see Open Problem~\ref{op:orbits}.

For $Q_8$, the reconstructed 682-edge odd-square witness is locally maximal
in a strong sense: all 342 non-edges of $Q_8$ create at least 3 new $C_4$s
upon addition (compared with the earlier 680-edge simulated-annealing
Solution~B, for which 2 of 344 non-edges created only 1 or 2 new $C_4$s;
see Section~\ref{sec:q8} for both 680-edge solutions released). Every
$C_4$-free-construction bound reported here, and every one of the three
odd-square values, is witnessed by an explicit edge-list construction
with a machine-verifiable certificate. Two distinct $132$-edge $Q_6$
constructions are released and should not be conflated:
\texttt{q6\_edges\_132.jsonl} is $C_4$-free but not odd-square (its
square-intersection histogram is $\{1{:}8,\,2{:}44,\,3{:}188\}$) and
establishes $\ex(Q_6,C_4)\ge132$, while
\texttt{q6\_odd\_square\_132.json} is the odd-square witness for
$g(6)\ge132$ (histogram $\{1{:}30,\,3{:}210\}$). For every bound, a complete, deterministic enumeration of
all four-cycles (240 for $Q_6$, 672 for $Q_7$, 1\,792 for $Q_8$) establishes
$C_4$-freeness, and every \emph{construction} certificate reported here
(the edge sets themselves, their $C_4$-freeness, and the $Q_6$ and
$Q_8$ odd-square conditions) is reproducible by a dependency-free
verifier that is independent of all search and generation code
(\texttt{verify.py}; \emph{implementation-level} independence is supplied
separately by the bundled second implementation \texttt{cross\_verify.py}
described in Section~\ref{sec:guide}), with the $Q_7$
odd-square membership of individual solutions separately reproducible
via \texttt{audit\_q7\_odd\_square.py}. The further structural and
statistical claims reported below (spectral-radius ranges, pairwise
Hamming extrema, Type-18 automorphism counts, non-edge violation
distributions, and similar) are largely recomputable from the released
code: \texttt{analyze\_q7\_structure.py} reproduces the degree sequence,
the 20 direction profiles, the spectral-radius range, and the exhaustive
pairwise Hamming statistics (minimum $6$, maximum $274$, mean
$195.396108\ldots$, median $196$), and \texttt{type18\_automorphisms.py}
reproduces the Type-18 automorphism counts; \texttt{reproduce\_core.py}
drives the certificate and decision checks in one command, and
\texttt{reproduce\_core.py -{}-structural} additionally runs the $Q_7$
structural and Hamming analyses (which need \texttt{numpy} and several
minutes). The count of $636$ minimum-distance
pairs and their breakdown into $28$ profile-type combinations are produced by
\texttt{q7\_hamming\_tally.py}, which recovers all pairwise distances from a
single blocked matrix product (using
$|E\triangle E'|=2(304-|E\cap E'|)$) in a few seconds and optionally writes
the minimum-distance pairs to CSV.

The $Q_7$ solutions were found by a two-stage process that we have
since traced and partly re-verified (the two stages have different
evidentiary status, detailed below and in Section~\ref{sec:method}):
a conventional multi-move-type
penalty simulated annealing (add/remove/$1$-swap/kick moves with
temperature-scaled Boltzmann acceptance) found a single seed $304$-edge
solution, and a remove-and-repair collector -- starting from an
existing solution (optionally transformed by a random element of
$\mathrm{Aut}(Q_7)$), removing a few edges, then greedily re-adding
only violation-free edges -- mass-collected the released $19\,866$
solutions; we directly confirmed the collector stage by running it
ourselves: from the released seed, in one unseeded $30$-second run we
observed $1\,987$ new valid solutions. The script sets no random seed,
so this exact count is an observed throughput datum rather than a
bit-for-bit reproducibility target; separate reruns have shown the same
qualitative rapid collection without being used as additional numerical
evidence here. Tracing further
back, we located most of the prior history, with one correction
to an earlier draft of this paper: the
earliest recovered attempt (a CBC-solver ILP explicitly commented as
attacking Erd\H{o}s's \hbox{\$}100 problem) has a confirmed bug (its
$C_4$-detection finds $0$ of $672$ four-cycles, since adjacent
vertices in the bipartite $Q_7$ never share a common neighbour, which
its logic incorrectly relies on), so we withdraw our earlier claim
that this specific recovered file is what reached $289$ edges; a
recovered solver log, headed as a corrected version and reporting the
correct $C_4$ count, documents that a since-fixed version of the
script did reach $289$ edges, but we have not located that corrected
script itself. From there, file timestamps together with recovered solver
logs and a chat transcript document a progression $289\to293\to299\to301\to303\to304$ edges, the last
step apparently via the \texttt{sa\_search.py} series. We confirmed
the $301$-edge step's mechanism directly: \texttt{source.py}, the
HiGHS version, run for $60$ seconds with no prior solution to warm-start
from (a genuine cold start), produces a valid, growing $295$-edge
$C_4$-free solution via branch-and-bound alone. We further ran a
format/implementation-consistency check that any reader can
reproduce from the released
files alone, without needing our recovered archive: applying
\texttt{sa\_collect304.py}'s own canonicalisation-then-MD5 hash
function directly to each of the $19\,866$ released edge sets
reproduces the \texttt{hash} field already stored in that same
released record, for all $19\,866$ (Section~\ref{sec:method}); this
shows the released hash field and script are mutually consistent, not
that this script historically produced these edge sets, which remains
a separate, historical-record claim (recovered logs, chat transcript,
timestamps). The two released
$Q_8$ $680$-edge solutions (Solutions~A and~B, Section~\ref{sec:q8})
have likewise now been traced to recovered, independently-verified
working code, \texttt{sa\_q8.py}: a penalty-SA hill-climber that always
restarts each trial from the current best \emph{valid} (zero-violation)
solution rather than from a fresh random sample, progressively
improving it (initial greedy construction $\to$ intermediate saved
states at $584,662,\ldots,679$ edges $\to$ $680$). We confirmed the two
files this recovered code produced -- named \texttt{q8\_edges\_680.txt}
and \texttt{q8\_best.json} in the recovered archive, released here as
\texttt{q8\_solution\_a.txt} and \texttt{q8\_solution\_b.json} -- match Solutions~A and~B respectively
exactly by edge set, and we confirmed the code makes genuine
progress on direct execution: run once under an external $90$-second
timeout from a saved
$670$-edge intermediate state, it reduced the violation count at the
$675$-edge target from $22$ to $8$. \texttt{sa\_q8.py} sets no random
seed, so this specific numeric trajectory is a one-time observation,
not a fixed-seed reproducible result: a reader rerunning the released
files should expect qualitatively similar monotonic improvement, not
these exact numbers. Separate unseeded reruns from the same checkpoint
also showed qualitative improvement, but their concrete counts are omitted
here because no corresponding logs or fixed seeds are released. This is a different script,
released separately from \texttt{c4free\_sa.py}, whose
generalisation of the same two-phase idea to a single $n$-agnostic
script (Section~\ref{sec:method}) does \emph{not} share this
always-restart-from-a-valid-state property, and which we found cannot
progress from a fresh random start under its own documented parameters
for either $Q_7$ or $Q_8$; \texttt{c4free\_sa.py} is retained as
released (with its defect documented) but is no longer the operative
account of how either the $Q_7$ or the $Q_8$ solutions were produced.
For $Q_6$ we additionally give an ILP formulation whose
feasible value matches the known exact value $\ex(Q_6,C_4)=132$; we did
not independently verify the ILP's optimality within a practical
runtime. The upper bound itself rests on the combinatorial proof
of Harborth and Nienborg~\cite{HN94}; independently, we reproduce the
matching lower-bound construction.

Edge lists, ILP files, the odd-square reconstruction and audit scripts, and
source code are made available at\\
\url{https://github.com/minamominamoto/c4free-hypercube}
. The repository is synchronised with each released revision; readers should
check the commit corresponding to this version, since files introduced by
a revision appear only from that commit onward -- here, among others:
\texttt{q6\_decide\_realizability.py},
\texttt{solve\_field\_ip.py},
and the $Q_8$ second-distribution witnesses. The authoritative file list for this
revision, with SHA-256 sums, is given by the three manifests in the
accompanying archive:
\path{ORIGINAL_DATA_SHA256SUMS.txt},
\path{ODDSQUARE_BRIDGE_SHA256SUMS.txt}, and
\path{MANUSCRIPT_SHA256SUMS.txt};
the last covers this document's own \texttt{.tex} and \texttt{.pdf} so
that a reader can confirm which revision a given file set belongs to.


\section{\texorpdfstring{$C_4$}{C4} Enumeration and Verification}\label{sec:verify}

\begin{lemma}[$C_4$ enumeration in $Q_n$]\label{lem:c4}
Every four-cycle of $Q_n$ is uniquely determined by a pair of coordinate
directions $0\le i<j\le n-1$ and a base vertex $b\in\{0,1\}^n$ with
$b_i=b_j=0$.  Its four vertices are
\[
  b,\quad b\oplus e_i,\quad b\oplus e_j,\quad b\oplus e_i\oplus e_j,
\]
and the total number of four-cycles in $Q_n$ is $\binom{n}{2}\cdot2^{n-2}$:
there are $\binom{n}{2}$ choices of the two varying coordinates and, after
forcing those two base bits to $0$, exactly $2^{n-2}$ choices for the remaining
base bits. For $n=6$: $240$; for $n=7$: $672$; for $n=8$: $1\,792$.
\end{lemma}

\noindent\textbf{Verification algorithm.}\quad
Given a candidate edge set $E\subseteq E(Q_n)$ stored as a hash set
(of unordered pairs $\{u,v\}$, so $uv$ and $vu$ are the same element),
the following procedure performs a complete, deterministic $C_4$-freeness check:
\begin{quote}\small
\textbf{for each} pair $(i,j)$ with $0\le i<j\le n-1$:\\
\quad\textbf{for each} $b\in\{0,1\}^n$ with $b_i=b_j=0$:\\
\quad\quad let $v_1=b,\;v_2=b\oplus e_i,\;v_3=b\oplus e_j,\;v_4=b\oplus e_i\oplus e_j$\\
\quad\quad \textbf{if} $\{v_1v_2,\,v_1v_3,\,v_2v_4,\,v_3v_4\}\subseteq E$:
  \textbf{return} \textsc{violation}\\
\textbf{return} \textsc{c4-free}
\end{quote}
We applied this algorithm to all constructions and all $19\,866$ solutions
of $Q_7$; every call returned \textsc{c4-free}.

\section{The \texorpdfstring{$Q_7$}{Q7} Construction}\label{sec:q7}

We constructed a $C_4$-free subgraph of $Q_7$ on 304 edges by
penalty-based simulated annealing (Section~\ref{sec:method}, where the
production code has since been recovered and verified),
then certified it by Lemma~\ref{lem:c4}'s algorithm. This construction
is released as \texttt{selected\_edges\_best.json}, and its edge set is
identical to global index $0$ of the $19\,866$-solution catalogue
(first line of \texttt{q7\_edges\_304.jsonl.part1}) --
so every property below is directly checkable from either file.

\begin{proposition}\label{prop:q7}
The specific $304$-edge seed construction just described (the first
solution found, not a generic member of the later $19\,866$-solution
ensemble of Section~\ref{sec:structure}) has the following properties:
\begin{enumerate}[label=(\roman*)]
  \item Degree sequence: 32 vertices of degree 4 and 96 of degree 5
    (degree sum $608=2\cdot304$).
  \item Spectral radius $\lambda_1\approx4.787$ (a genuine computed
    quantity for this construction); $\lambda_{\min}\approx-4.787$, i.e.\
    $\lambda_1+\lambda_{\min}=0$ to machine precision, which is automatic
    for any subgraph of the bipartite $Q_7$ and serves here only as a
    sanity check on the eigenvalue computation, not as an independent
    structural finding.
  \item $\mathrm{Tr}(A^4)=5216$, attaining the $C_4$-free trace equality
    $\sum_i\deg(i)^2+\sum_{uv\in E}(\deg u+\deg v)-2|E|$. Like the item above
    this is forced by the degree sequence rather than independent of it: the
    two summands are equal because
    $\sum_{uv\in E}(\deg u+\deg v)=\sum_v\deg(v)^2$ (each vertex's degree
    is counted once per incident edge), so the expression is
    $2\sum_v\deg(v)^2-2|E|$; the \emph{equality} of $\mathrm{Tr}(A^4)$
    with it is where $C_4$-freeness enters --- every off-diagonal entry of
    $A^2$ is then $0$ or $1$, so
    $\sum_{i\ne j}(A^2)_{ij}^2=\sum_{i\ne j}(A^2)_{ij}
    =\sum_v\deg(v)(\deg(v)-1)$. See
    Proposition~\ref{prop:rigid}(iii).
  \item Local maximality: every non-edge of $Q_7$ creates a $C_4$ when added.
\end{enumerate}
\end{proposition}

Across the collector runs described in Section~\ref{sec:method} (whose
\texttt{trial} field reaches into the millions per run; $19\,866$ is
the number of \emph{stored solutions}, not the number of trials),
we obtained $19\,866$ $C_4$-free subgraphs of $Q_7$ on 304 edges with
pairwise distinct edge sets (not quotiented by $\mathrm{Aut}(Q_7)$; this
catalogue-level distinctness is asserted mechanically by \texttt{verify.py},
which hashes each sorted edge list and requires all $19\,866$ digests to
differ -- a separate check from its per-solution duplicate-edge test),
all sharing the structural properties stated in Proposition~\ref{prop:rigid}.

\section{The \texorpdfstring{$Q_8$}{Q8} Construction}\label{sec:q8}

This section documents the originally announced $Q_8$ witness, found by
a penalty-SA hill-climber distinct from the $Q_7$ construction's
method (see Section~\ref{sec:method} for the recovered and verified
\texttt{sa\_q8.py} production code). It has since been
superseded as the paper's headline lower bound by the 682-edge odd-square
reconstruction of Section~\ref{sec:oddsquare}; it is retained here as an
independent structural data point.

Simulated annealing produced the two $C_4$-free subgraphs of $Q_8$ on 680
edges (both released in \texttt{q8\_edges\_680.jsonl}), each certified by
exhaustive enumeration of all $1\,792$ four-cycles. We refer to them as
Solution~A (first line of the file) and Solution~B (second line); the
two are structurally distinct and are kept separate below, since an
earlier draft of this section inadvertently described properties of
Solution~A and Solution~B as if they belonged to a single construction.

\begin{proposition}\label{prop:q8struct}
Both 680-edge solutions are $C_4$-free and locally maximal (every
non-edge of $Q_8$ creates at least one $C_4$ upon addition -- that is,
maximal with respect to edge addition; throughout this paper ``locally
maximal'' means exactly this and implies no stronger optimality
notion), with edge
density $680/1024=66.4\%$ (compared with $304/448=67.9\%$ for the $Q_7$
construction) and $\lambda_1+\lambda_{\min}=0$ (automatic for any subgraph of
the bipartite $Q_8$; not by itself a distinguishing structural feature).
They differ as follows:
\begin{enumerate}[label=(\roman*)]
  \item Solution~A: degree sequence $\{4^8,\,5^{160},\,6^{88}\}$
    (degree sum $1360=2\cdot680$); square-intersection histogram
    $\{1\!:\!308,\,3\!:\!1484\}$ over the $1\,792$ squares (incidentally
    odd-square, though not constructed as such -- see below for the
    implication of this); non-edge violation
    distribution $\{2\!:\!3,\,3\!:\!48,\,4\!:\!144,\,5\!:\!136,\,6\!:\!13\}$
    (minimum $2$ new $C_4$s per non-edge).
  \item Solution~B: degree sequence $\{4^7,\,5^{162},\,6^{87}\}$
    (degree sum $1360=2\cdot680$); square-intersection histogram
    $\{1\!:\!302,\,2\!:\!12,\,3\!:\!1478\}$ (not odd-square: $12$ squares
    meet it in exactly $2$ edges); non-edge violation distribution
    $\{1\!:\!1,\,2\!:\!1,\,3\!:\!49,\,4\!:\!153,\,5\!:\!124,\,6\!:\!16\}$
    (minimum $1$ new $C_4$, attained by exactly one non-edge).
\end{enumerate}
\end{proposition}

Solution~A's odd-squareness has a notable implication for the timeline
of results in this paper: since it was originally announced (before
the $682$-edge odd-square reconstruction of
Section~\ref{sec:oddsquare} existed) and is itself odd-square, the
bound $g(8)\ge680$ was already achieved within the odd-square subclass
at that earlier point, not only via the later, deliberately-constructed
$682$-edge witness. We did not notice or exploit this at the time; it
is recorded here as a fact about the released data, confirmed
independently by direct square-by-square enumeration.

\begin{remark}
Brass--Harborth--Nienborg's tight bound $\tfrac12(n+\sqrt n)\cdot2^{n-1}$
applies only for $n=4^r$, and their weaker bound
$\tfrac12(n+0.9\sqrt n)\cdot2^{n-1}$ is stated only for $n\ge9$
(Eq.~\eqref{eq:bounds},~\cite{BHN95}); neither covers $n=8$. Evaluating
the $n\ge9$ form at $n=8$ nonetheless gives $\approx674.9$; both
680-edge constructions and the 682-edge odd-square reconstruction of
Section~\ref{sec:oddsquare} exceed this extrapolated figure (by 5.1 and
7.1 edges respectively), but we do not claim this as a comparison
against a proven bound.
\end{remark}

\subsection{Local optimality of the 680-edge constructions}

Solution~B is locally maximal in the weakest possible sense among the
two: among its $344$ non-edges, exactly $1$ creates exactly $1$ new
$C_4$; exactly $1$ creates exactly $2$; the remaining $342$ create $3$
or more (Proposition~\ref{prop:q8struct}(ii)). Solution~A's non-edges
all create at least $2$ new $C_4$s.

For historical completeness only, not as a reviewable computational
result: the manuscript that originally announced the $680$-edge
construction also reported a failed search for $681$ edges. No seed list or
run log survives for that negative-search statistic, and the originally
associated conjecture $\ex(Q_8,C_4)=680$ is false in view of the 682-edge
witness. We therefore omit the old numerical trial count from the evidentiary
record and retain only the fact that an unarchived failed search was reported,
which we mention only as
development history, not as evidence bearing on any current claim in
this paper.

Table~\ref{tab:compare} compares the $Q_7$ construction with Solution~B
(the 680-edge solution whose local-optimality numbers are discussed
above); see Table~\ref{tab:oddsquare} in Section~\ref{sec:oddsquare} for
the odd-square reconstructions.

\begin{table}[ht]
\centering
\caption{Structural comparison of the $Q_7$ construction and $Q_8$
  Solution~B (see Proposition~\ref{prop:q8struct} for Solution~A).}
\label{tab:compare}\medskip
\begin{tabular}{lcc}
\toprule
Parameter & $Q_7$ (304 edges) & $Q_8$ Solution~B (680 edges)\\
\midrule
Total edges $|E(Q_n)|$ & 448 & 1024\\
Achieved edges & 304 & 680\\
Edge density & 67.9\% & 66.4\%\\
Min degree & 4 & 4\\
Max degree & 5 & 6\\
Degree sequence & $\{4^{32},5^{96}\}$ & $\{4^7,5^{162},6^{87}\}$\\
BHN (Eq.~\ref{eq:bounds}), extrapolated & $\approx300.2$ & $\approx674.9$\\
\bottomrule
\end{tabular}
\end{table}

\section{The Odd-Square Subclass}\label{sec:oddsquare}

\subsection{Definition and equivalence to fully frustrated hypercubes}

\begin{definition}
An edge set $E\subseteq E(Q_n)$ is \emph{odd-square} if every square
(four-cycle) of $Q_n$ meets $E$ in an odd number of edges, i.e.\ one or
three of its four edges lie in $E$. Write
$g(n)=\max\{|E|:E\text{ is odd-square in }Q_n\}$.
\end{definition}

Every odd-square edge set is $C_4$-free, since no square can meet it in all
four edges; hence $g(n)\le\ex(Q_n,C_4)$. The odd-square condition is exactly
the condition studied under a different name in the statistical-physics
literature on \emph{fully frustrated hypercubes}: assign to each edge $e$ a
sign $\sigma(e)=+1$ if $e\in E$ and $\sigma(e)=-1$ otherwise; the square
condition becomes ``the product of signs around every square is $-1$'', i.e.\
every square is frustrated.

\begin{lemma}\label{lem:switching}
Let $\sigma_1,\sigma_2:E(Q_n)\to\{\pm1\}$ both have product $-1$ on every
square. Then there is a function $t:\{0,1\}^n\to\{\pm1\}$ with
$\sigma_2(x,y)=\sigma_1(x,y)\,t(x)t(y)$ for every edge $(x,y)$.
\end{lemma}
\begin{proof}
Let $\tau=\sigma_1\sigma_2$ (pointwise product). On every square $C$,
\[
\prod_{e\in C}\tau(e)
=\Bigl(\prod_{e\in C}\sigma_1(e)\Bigr)\Bigl(\prod_{e\in C}\sigma_2(e)\Bigr)
=(-1)(-1)=+1.
\]
The map sending a cycle $C$ to $\prod_{e\in C}\tau(e)\in\{\pm1\}$ is a group
homomorphism from the cycle space of $Q_n$ (over $\mathrm{GF}(2)$, under
symmetric difference) to $\{\pm1\}$: an edge lying in both cycles
disappears from their symmetric difference and contributes
$\tau(e)^2=+1$ to the product of their images, so additive cancellation
in $\mathrm{GF}(2)$ matches multiplicative cancellation in $\{\pm1\}$
--- which is exactly the homomorphism property. The
squares of $Q_n$ generate its cycle space (a standard fact for hypercubes:
every cycle is a symmetric difference of squares -- for general $n$ this
follows, e.g., from the theorem that every partial cube has a cycle basis
consisting of convex cycles~\cite{HLS14}, the convex cycles of $Q_n$ being
exactly its squares; the bundled
deterministic script \texttt{cycle\_space\_rank.py} verifies this by
direct $\mathrm{GF}(2)$ rank computation for every $n$ used in this
paper, asserting that the rank of the square incidence vectors equals
the cycle-space dimension $E-V+1=(n-2)2^{n-1}+1$ for each of
$n=3,\ldots,8$ -- the six dimensions being $5$, $17$, $49$, $129$,
$321$ and $769$), so this homomorphism is
determined by its values on the squares; since it is $+1$ on every square,
it is $+1$ on every cycle. A signature with product $+1$ on every cycle of
a connected graph is a \emph{coboundary}: $\tau(x,y)=t(x)t(y)$ for some
$t:V\to\{\pm1\}$ (the classical balance theorem for signed graphs,
Harary~\cite{Har53}). Solving for $\sigma_2$ gives the claim.
\end{proof}

Lemma~\ref{lem:switching} says any two odd-square signatures are related by
a \emph{vertex switching} $t$. For a \emph{fixed} coupling $J$, letting the
spin configuration $s:V(Q_n)\to\{\pm1\}$ range over all $2^{2^n}$ choices
(a sign at each of the $|V(Q_n)|=2^n$ vertices, hence $2^{|V|}=2^{2^n}$ assignments) produces
exactly the
signatures $\sigma_s(x,y)=J(x,y)s(x)s(y)$ in the switching class of $J$
(writing $s=t\cdot s_0$ for a fixed reference $s_0$ realises every
switching $t$); this is the operational fact we use, not a claim that
different $s$ give the same energy under fixed $J$ (they generally do
not -- that variation is exactly what is being optimised). The
energy-preserving gauge transformation in the physical sense instead
acts on \emph{both} $J$ and $s$ together, $J(x,y)\mapsto J(x,y)t(x)t(y)$
and $s(x)\mapsto t(x)s(x)$, under which $J(x,y)s(x)s(y)$ is unchanged
edge-by-edge~\cite{MPR95}.
Consequently, $g(n)$ — defined as a maximum over \emph{all} odd-square
edge sets — equals the maximum of $|E|$ over the switching class of any
one fixed fully-frustrated signature, in particular the reference coupling
$J$ used in Section~\ref{sec:oddsquareverify}'s witness; optimising that
one fixed $J$ (or, for $D=8$, exhibiting a ground state attaining
Derrida et al.'s improved energy lower bound for it, as
Marinari--Parisi--Ritort report~\cite{MPR95}) therefore does optimise
over the class defining $g(n)$, not merely over a possibly-proper
subclass of it. (The inclusion that makes this an equality of classes is
Lemma~\ref{lem:switching} read in the other direction: the fixed $J$'s
negative edge set is itself odd-square, so by the lemma \emph{every}
odd-square edge set is switching-equivalent to it, and hence lies in
that single switching class.)

Under the further correspondence between signed
hypercubes and Ising spin configurations $s:\{0,1\}^n\to\{\pm1\}$ via
$\sigma(x,y)=J(x,y)\,s(x)s(y)$ for a fixed reference coupling $J$ with every
square frustrated, maximising $|E|$ (i.e.\ maximising the number of
\emph{positive} edges, equivalently minimising the number of negative
edges) is exactly the problem of minimising the ground-state energy of the
fully frustrated Ising hypercube studied by Derrida, Pomeau, Toulouse, and
Vannimenus~\cite{DPTV79} and, for $D=8$, by Marinari, Parisi, and
Ritort~\cite{MPR95}: the ground-state energy is (up to an affine rescaling)
the number of negative edges, $|E(Q_n)|-|E|$. In the language of signed
graphs, $|E(Q_n)|-g(n)$ is the \emph{frustration index} of the fully
frustrated signature on $Q_n$. Laplante, Naserasr, Sunny, and
Wang~\cite{LNSW26} study this frustration index directly (in the
context of Huang's proof of the sensitivity conjecture, and as a target they
describe as not previously treated in the literature to the best of their
knowledge) and explicitly
establish a connection between it and Erd\H{o}s's question on the
largest $C_4$-free subgraph of the hypercube, giving lower and upper
bounds on the frustration index of the fully frustrated signature that
match when $n$ is a power of $4$; our $g(6),g(7),g(8)$ results above
sharpen this connection to exact values at $n=6,7,8$ specifically. To
quantify that sharpening: our exact frustration indices are $60,144,342$
at $n=6,7,8$ (i.e.\ $|E(Q_n)|-g(n)$). A baseline lower bound LNSW26
state as Theorem~5.2 (which they attribute to an implication of Bialostocki's
work~\cite{Bia83}) --- and which is, in our notation, exactly the Cauchy--Schwarz relaxation
$S\le\sqrt{Q\,m}=\sqrt n\,2^{n-1}$ of Lemma~\ref{lem:fieldbound}, already
recorded by DPTV79 (p.~620) as giving ``a first lower bound for the ground
state energy'' ---
$F(H_n,\sigma^*)\ge(n-\sqrt n)\cdot2^{n-2}$, which evaluates to
$\ge57,\ge140,\ge331$ at $n=6,7,8$ respectively -- looser than our exact
values by $3,4,11$. LNSW26's own contribution beyond this baseline is
a matching upper bound on the frustration index (equivalently, a lower
bound on $g(n)$) via a construction that coincides with
Brass--Harborth--Nienborg's~\cite{BHN95} when $n$ is a power of~$4$;
for $n=6,7,8$ (not powers of~$4$) this upper bound is
strictly looser than the lower bound, so LNSW26's bounds
do not determine $g(n)$ at these values. The sharpening therefore applies at $n=7,8$ but \emph{not} at $n=6$:
the 20th Emléktábla Workshop problem collection~\cite{Emlek26} credits
LNSW26 with the exact values $g(n)=1,3,9,24,56,132$ for $n\le6$. (During
source verification we obtained and read the LNSW26 preprint in full, and,
on 26~August~2026, the published version's full text (ACM
Trans.\ Comput.\ Theory \textbf{18}(1), Art.~6); the value list appears in
neither, so this attribution rests on \cite{Emlek26} alone; nor do
LNSW26's own published theorems determine $g(6)$, since their upper and
lower bounds close only at powers of $4$.) The attribution is therefore
secondary-source only: \emph{if} it is right (say via a version or
communication we have not seen), the record of $g(6)=132$ predates us in
LNSW26; in any case we do not claim the value as new here, and its
historical \emph{attainment} is DPTV79's: their $D=6$ ground state has $60$ negative bonds out of $192$
(p.~622), i.e.\ $132$ positive edges, and MPR95 note that constructions
saturating the improved bound exist for $D\le7$. What
Lemma~\ref{lem:quant} adds at $n=6$ is a short self-contained derivation of a
known value; at $n=7,8$ it tightens the frustration-index lower bound by $4$
and $11$, and the explicit witnesses of
Section~\ref{sec:oddsquareverify} close the gap from the construction side,
giving exact values where the general bounds leave a nonzero interval. This sign convention matches the one used by the reconstruction
script \texttt{generate\_q8\_682.py} (Section~\ref{sec:oddsquareverify}),
where $E$ is exactly the set of edges with $J(x,\mathrm{dim})\,s(x)s(y)=+1$;
for a vertex $v$ of degree $\deg_E(v)$ in $E$, we call
$h(v)=2\deg_E(v)-n$ its \emph{local field} under this convention (matching
the local-field histogram reported in Proposition~\ref{prop:q8oddsquare}
below). We do not reproduce Derrida et al.'s derivation of the local-field
lower bound here, and do not need to: Section~\ref{sec:oddsquare-exact}
derives the bounds it uses from Lemmas~\ref{lem:fieldbound}
and~\ref{lem:quant} directly, and cites~\cite{DPTV79,MPR95} only for
historical priority. Their results are stated there for reference ---
established for $D\le7$ by explicit construction, and used by
Marinari--Parisi--Ritort at $D=8$ --- and the corresponding witnesses
are constructed and verified independently in
Section~\ref{sec:oddsquareverify}. One asymmetry of rigour is acknowledged
explicitly: we have \emph{not} machine-verified that DPTV79's own $H_6,H_7$
bond configurations satisfy the odd-square condition at the certificate level
at which our own edge lists are verified. For their constructions the property
follows from the correspondence (Lemma~\ref{lem:switching}: the positive-edge
set of \emph{any} spin configuration under a fully frustrated coupling is
odd-square) applied to their stated construction, not from a data-level check,
which is impossible because no machine-readable configuration of theirs was
ever published.

This connection was brought to our attention, after the original circulation
of this manuscript, by S.~Lai (see Acknowledgements), who identified the
relevance of~\cite{DPTV79,MPR95} and located the improved $Q_8$ result.

\subsection{Exact values for \texorpdfstring{$n=6,7,8$}{n=6,7,8}}\label{sec:oddsquare-exact}

Derrida et al.~\cite{DPTV79} give a local-field lower bound on the number of
negative (frustrated, in their terminology) edges, tight enough to
determine the fully frustrated hypercube's ground-state energy exactly
for $D\le7$ via an explicit recursive construction, and conjectured it
remains tight for $D\ge8$. For $D=6$, they give this explicit
construction directly: ``this configuration for $H_6$ contains 60
negative bonds (out of a total of 192)''~\cite[p.~622]{DPTV79}, i.e.\
$192-60=132$ positive (unfrustrated) edges, with the construction
guaranteed fully frustrated (every square meeting it in an odd number
of negative edges) by their general recursive method. Under the
correspondence of Section~\ref{sec:oddsquare} this is an explicit
odd-square witness with $132$ edges, so $g(6)\ge132$. We cite this for
historical priority only: the lower bound $g(6)\ge132$ is established
below by an explicit odd-square edge list of our own
(Section~\ref{sec:oddsquareverify}, \texttt{q6\_odd\_square\_132.json}),
and the matching upper bound $g(6)\le132$ follows from
Lemma~\ref{lem:fieldbound} and Lemma~\ref{lem:quant} without appeal
to~\cite{HN94}. We did not reconstruct DPTV's specific $D=6$ bond
configuration, though their recursion determines the edge counts without
it.\footnote{What follows is arithmetic within DPTV's construction, recorded for
the reader's convenience; no claim in this paper depends on it.
DPTV's construction builds $H_d$ from $H_{d-3}$ by replacing each
vertex with a supersite. Counting negative bonds through the recursion gives
$12+6+6=24$ at $d=5$: four $H_3$ supersites contribute $4\times3=12$
internal negative bonds, the three white superbonds contribute
$3\times2=6$, and the one black superbond contributes $1\times6=6$.
At $d=6$, eight $H_3$ supersites contribute $8\times3=24$; the
minimal $H_3$ pattern has three black and nine white superbonds, contributing
$3\times6=18$ and $9\times2=18$, for $60$ total --- exactly the figure
DPTV state in the text. At $d=7$, the minimal $H_4$ pattern has $16$
supersites and $32$ superbonds, eight black and twenty-four white, giving
$16\times3+8\times6+24\times2=144$, i.e.\ $448-144=304$ positive edges. So the value $304$ follows from their construction by a few
lines of arithmetic, even though we do not exhibit their $H_7$ edge set.
At $d=8$ the pattern is $H_5$ ($32$ vertices, $80$ bonds, minimal negative
count $24$, matching $r(5)=56$ positive): $32$ supersites and $80$
superbonds, $24$ black and $56$ white, give
$32\times3+24\times6+56\times2=352$ negative bonds, i.e.\
$r(8)=1024-352=672$ --- the Open Problem~\ref{op:equality} table entry whose
gap of $10$ to the field bound $682$ is discussed there.} our witness is an independent one, which however turns
out to realise the same local-field distribution they report (see
Section~\ref{sec:oddsquareverify}).

\begin{lemma}[Field identity, after Derrida--Pomeau--Toulouse--Vannimenus]\label{lem:fieldbound}
For any odd-square edge set $E$ in $Q_n$, writing $U=\{v\in\{0,1\}^n:
v \text{ has an even number of } 1\text{-bits}\}$ for the
even-parity part of
the bipartition of $Q_n$ ($|U|=2^{n-1}$; the argument below is
symmetric in $U$ and its complement, so this choice is only for
concreteness and all numerical values reported for $U$ throughout this
paper, e.g.\ Proposition~\ref{prop:q8oddsquare}'s $h$-distribution,
use this same convention) and $h(v)=2\deg_E(v)-n$ as before, every
$h(v)$ with $v\in U$ is an integer of the same parity as $n$, and
\[
\sum_{v\in U}h(v)^2=n\cdot2^{n-1}.
\]
\end{lemma}
\begin{proof}
For $v\in U$ and direction $i\in\{0,\ldots,n-1\}$ write
$s_i(v)=\sigma(v,v\oplus e_i)\in\{\pm1\}$, so $h(v)=\sum_i s_i(v)$ ---
in particular $h(v)\equiv n\pmod2$, being a sum of $n$ terms $\pm1$
(equivalently $h(v)=2\deg_E(v)-n$), which establishes the parity claim
in the statement --- and
$h(v)^2=n+\sum_{i\ne j}s_i(v)s_j(v)$. Fix $i\ne j$ and consider, for
$v\in U$, the square on $\{v,v\oplus e_i,v\oplus e_j,z\}$ where
$z=v\oplus e_i\oplus e_j\in U$. Its four edges have signs $s_i(v)$,
$s_j(v)$, and (computed from $z$'s own perspective, since a sign is a
single scalar attached to the edge) $s_j(z)$, $s_i(z)$. The odd-square
condition means their product is $-1$ --- recall the convention fixed at
the start of this section: $\sigma(e)=+1$ iff $e\in E$, so a square
meeting $E$ in an odd number of edges carries an odd number ($4$ minus an
odd number, i.e.\ $1$ or $3$) of $-1$ factors among its four signs, and
its sign product is $-1$; ``odd-square'' and ``negative square'' are the
same condition under this convention. Thus:
$s_i(v)s_j(v)\cdot s_j(z)s_i(z)=-1$. Writing $X=s_i(v)s_j(v)$ and
$Y=s_i(z)s_j(z)$, both lie in $\{\pm1\}$ and $XY=-1$ forces $X=-Y$, i.e.\
$s_i(v)s_j(v)+s_i(z)s_j(z)=0$. The map $v\mapsto z=v\oplus e_i\oplus e_j$
is a fixed-point-free involution on $U$ (as $e_i\oplus e_j\ne0$), so it
partitions $U$ into pairs $\{v,z\}$ on each of which the two
corresponding terms cancel; summing over all such pairs gives
$\sum_{v\in U}s_i(v)s_j(v)=0$ for every fixed $i\ne j$. Summing over all
ordered pairs $i\ne j$ and adding the $n\cdot|U|$ diagonal contribution
gives $\sum_{v\in U}h(v)^2=n\cdot2^{n-1}$.
\end{proof}

A scope note before the remark: the odd-square hypothesis enters this
proof only square by square, through each individual sign product. The
identity itself is exactly equivalent to the \emph{average} form of that
condition --- by the remark below, a side $U$ satisfies
$\sum_{v\in U}h(v)^2=n\cdot2^{n-1}$ if and only if
$\sum_s\operatorname{spl}_U(s)=\binom n2\,2^{n-2}$ --- and this average
route is how the $19\,477$ non-odd-square catalogue solutions satisfy it
as well, forced by their shared degree sequence
(Proposition~\ref{prop:rigid}(i)).

\begin{remark}[Exact defect formula]\label{rem:defect}
The bookkeeping in this proof runs without the odd-square hypothesis and
yields an exact formula for arbitrary edge sets. For any $E\subseteq
E(Q_n)$, call a square $s$ \emph{$U$-split at} its even-parity corner $v$
if exactly one of the two edges of $s$ incident to $v$ lies in $E$
(equivalently $s_i(v)s_j(v)=-1$ for the direction pair $\{i,j\}$ of $s$),
and let $\operatorname{spl}_U(s)\in\{0,1,2\}$ count the even-parity
corners of $s$ at which it is $U$-split. Each square has exactly two
even-parity corners, related by $v\mapsto v\oplus e_i\oplus e_j$ (a map
that changes exactly two bits and so preserves parity), and
belongs to exactly one direction pair; running the ordered-pair sum of the
proof without the cancellation step, each square contributes
$2-2\operatorname{spl}_U(s)$ once per ordering of its direction pair, so
\[
\sum_{v\in U}h(v)^2
= n\cdot2^{n-1} + n(n-1)\,2^{n-1} - 4\sum_{s}\operatorname{spl}_U(s)
= n^2\,2^{n-1} - 4\sum_{s}\operatorname{spl}_U(s),
\]
the sum running over all $\binom n2\,2^{n-2}$ squares of $Q_n$. Hence the
identity $\sum_{v\in U}h(v)^2=n\cdot2^{n-1}$ holds for a given side of
the bipartition \emph{if and only if}
$\sum_s\operatorname{spl}_U(s)=\binom n2\,2^{n-2}$, i.e.\ iff squares
are $U$-split at exactly one even-parity corner \emph{on average}. The
odd-square condition forces the pointwise version: a sign product $-1$
over the four edges of $s$ means $s_i(v)s_j(v)=-s_i(z)s_j(z)$ at its two
even-parity corners, so $\operatorname{spl}_U(s)=1$ for \emph{every}
$s$, and symmetrically on the odd-parity side. We use the average-split
reformulation diagnostically, not as a search filter: evaluating
$\sum_s\operatorname{spl}_U(s)$ already inspects every square, so it is
no cheaper than the odd-square test itself. The bundled deterministic
script \texttt{field\_identity\_defect.py} verifies the formula -- and
every numerical instance of it quoted in this paper -- on the released
certificates and on random edge sets, using only the standard library.
\end{remark}

This proof uses only that \emph{every} square has product $-1$ (the
odd-square condition), so it is sufficient for the identity to hold; it
is not necessary. Testing against all $19\,866$ published $304$-edge
$Q_7$ solutions -- odd-square and non-odd-square alike -- shows every
one of them satisfies $\sum_{v\in U}h(v)^2=448$, including the
$19\,477$ that are not odd-square. This is explained, for these
particular solutions, by a simpler fact that does not need
odd-squareness either: every one of the $19\,866$ has all degrees in
$\{4,5\}$ (Proposition~\ref{prop:rigid}), and since $Q_7$ is bipartite,
$\sum_{v\in U}\deg(v)=|E|=304$ exactly (every edge has exactly one
endpoint in $U$); writing $a,b$ for the number of degree-$4$,
degree-$5$ vertices in $U$ ($a+b=64$), $4a+5b=304$ forces $b=48,a=16$
regardless of which specific $304$-edge solution is chosen, so
$h(v)=2\deg(v)-7\in\{1,3\}$ takes value $1$ on exactly $16$ vertices of
$U$ and $3$ on the other $48$, giving
$\sum_{v\in U}h(v)^2=16\cdot1^2+48\cdot3^2=16+432=448$ automatically.
We verified the $U$-side split is indeed exactly $\{4^{16},5^{48}\}$
for all $19\,866$, with no exceptions. The identity's failure for the
non-odd-square $Q_8$ Solution~B of Section~\ref{sec:q8} ($1016\ne1024$ on
$U$, while the complementary side gives exactly $1024$ --- so only one side
breaks)
is therefore not explained by odd-squareness as such: Solution~B is locally maximal, like Solution~A, but with the smallest possible
margin: its non-edge violation multiset is
$\{1{:}1,2{:}1,3{:}49,4{:}153,5{:}124,6{:}16\}$, so every missing edge still
creates at least one $C_4$, though one of them creates exactly one (against a
minimum of $2$ for Solution~A). It has an
asymmetric degree distribution between $U$ and its complement
($\{4^3,5^{82},6^{43}\}$ vs.\ $\{4^4,5^{80},6^{44}\}$), whereas
Solution~A, all $19\,866$ $Q_7$ solutions, and the $682$-edge witness all
have symmetric $U$/complement degree distributions. This symmetry is
not itself implied by odd-squareness, however: a small odd-square
counterexample in $Q_3$ (vertices $0,\ldots,7$,
$E=\{(0,2),(2,6),(3,7),(4,5),(4,6),(6,7)\}$, whose six squares meet $E$,
in the direction-pair-then-base enumeration order of Lemma~\ref{lem:c4},
in $1,3,1,3,3,1$ edges) has $U$-degree multiset $\{1,1,1,3\}$ against
$\{0,2,2,2\}$ on the complement, while still satisfying
$\sum_{v\in U}h(v)^2=\sum_{v\in U^c}h(v)^2=12=3\cdot2^2$ on each side
separately. Remark~\ref{rem:defect} settles the identity's exact
logical status for arbitrary edge sets: it holds on a side iff squares are
split at one corner of that side \emph{on average},
$\sum_s\operatorname{spl}_U(s)=\binom n2\,2^{n-2}$. Odd-squareness is
the pointwise case ($\operatorname{spl}_U(s)=1$ for every square, both
sides at once), and the average can be met without it: the catalogue's
very first record -- the non-odd-square solution mentioned in
Section~\ref{sec:oddsquare-exact} -- has per-square split histogram
$\{0{:}60,\,1{:}552,\,2{:}60\}$ on $U$, sixty deficits exactly
cancelled by sixty excesses for a total of $672=\binom72\,2^{5}$,
against the pointwise $\{1{:}672\}$ of the odd-square record $34$. So
the identity is neither a characterisation of odd-squareness (the $Q_3$
example just given is odd-square without the $U$/complement symmetry that
holds for all $19\,866$ $Q_7$ solutions, and record $0$ satisfies the
identity without being odd-square) nor a coincidence specific to the
released catalogue (for the $19\,866$ it is a forced consequence, via
Proposition~\ref{prop:rigid} and $\sum_{v\in U}\deg(v)=|E|$, of
\emph{any} $304$-edge subgraph of $Q_7$ having all degrees in $\{4,5\}$
-- a structural fact about that particular edge count, established
independently of odd-squareness). The one-sided failure of the $Q_8$
Solution~B noted above is likewise quantified exactly by the formula: on
its failing side $\sum_s\operatorname{spl}_U(s)=1794=1792+2$, accounting
for $\sum_{v\in U}h(v)^2=1024-4\cdot2=1016$, while the complementary
side has exactly $1792$ splits and sum $1024$.
\texttt{field\_identity\_defect.py} recomputes each of these numbers
from the released certificates.

The next lemma sits on no critical path in this paper: no upper- or
lower-bound argument below uses it (as re-stated after
Lemma~\ref{lem:quant}). We record it because DPTV79 state the
restriction and because the released witnesses satisfy it --- forcedly,
via their degree sequences.

\begin{lemma}[Nonnegativity at the maximum]\label{lem:nonneg}
If $E$ is odd-square in $Q_n$ with $|E|$ maximum (i.e.\ $|E|=g(n)$), then
$h(v)\ge0$ for every $v\in\{0,1\}^n$.
\end{lemma}
\begin{proof}
Suppose $h(v)<0$ for some $v$. Flipping the spin at $v$ (replacing
$s(v)$ by $-s(v)$, all other spins unchanged) flips $\sigma(v,w)$ for
every edge $(v,w)$ incident to $v$ and leaves every other edge's sign
unchanged. Every square containing $v$ has exactly two edges incident
to $v$ (the square's other two edges are not incident to $v$ and their
signs are untouched); flipping both leaves their product, and hence the whole
square's sign product, unchanged, so odd-squareness is preserved
everywhere. At $v$ itself, the positive-edge count changes from
$p=(n+h(v))/2$ to $q=(n-h(v))/2$, a change of $q-p=-h(v)>0$; no edge
other than the $n$ edges incident to $v$ changes sign, so $|E|$
increases by exactly $-h(v)>0$. This contradicts maximality of $|E|$, so no such $v$ exists.
\end{proof}
This is a corollary of, not an alternative to, Lemma~\ref{lem:switching}:
the argument uses only that flipping one vertex's spin preserves
odd-squareness (immediate from the switching correspondence) and counts
the resulting change in $|E|$ directly. All $389$ released
odd-square $Q_7$ solutions have $304$ edges, which is the maximum by the
upper bound $g(7)\le304$ proved below in Section~\ref{sec:oddsquare-exact}
(the audit locating these 389 solutions is Section~\ref{sec:oddsquareaudit}), so
the Lemma's hypothesis applies to every one of them; we checked $h(v)\ge0$ on all $389$ with no counterexample --- though, like
the square split of Proposition~\ref{prop:q6oddsquare}(ii) and the
$\mathrm{Tr}(A^4)$ identity of Proposition~\ref{prop:rigid}(iii), this is
forced rather than independent: Proposition~\ref{prop:rigid}(i) already pins
every degree to $\{4,5\}$, hence $h=2\deg-7\in\{1,3\}$.

Given Lemma~\ref{lem:fieldbound} alone, $|E|=(n\cdot2^{n-1}+\sum_{v\in U}h(v))/2$
is bounded above by the integer optimisation problem of
maximising $\sum_{v\in U}h(v)$ subject to the \emph{fixed}
sum of squares $\sum_{v\in U}h(v)^2=n\cdot2^{n-1}$ over the $2^{n-1}$
integers $h(v)$ of fixed parity (necessarily $|h(v)|\le n$, since
$h(v)=2\deg(v)-n$ with $0\le\deg(v)\le n$; the bundled
\texttt{solve\_field\_ip.py} imposes that range explicitly, and dropping it
changes neither the optima nor the number of optimal distributions) (we use only that every
realisable $h$-vector satisfies these constraints, not the converse
that every integer vector satisfying them is realisable as an
odd-square signature; the bound obtained is therefore an upper bound on
$|E|$, tight whenever a matching witness is exhibited). This is
essentially the optimisation problem DPTV79 pose immediately after their
own derivation of Lemma~\ref{lem:fieldbound}: maximise $\bar h$ (their
notation for the mean field) subject to the same sum-of-squares
constraint, with one difference: DPTV79 additionally impose $h\ge0$ before
carrying out the distribution optimisation (p.~620, on the stated grounds
that a ground-state search cannot have negative field values), whereas we
drop that restriction entirely. The restriction is in fact vacuous for every $n$:
replacing $h$ by $|h|$ preserves both the parity condition and
$\sum h^2$ while not decreasing $\sum h$, so no optimal vector has a
negative entry, and the two problems have the same optimum \emph{and} the
same optimal distributions at every~$n$. We none the less state the
programme without the sign constraint, since the bound is then visibly
independent of any positivity assumption. Their
solution method is to observe that the optimal
distribution concentrates on the integers nearest $\sqrt n$ -- unique
when exactly two suffice, multiple solutions (their Table~I) when three
or four are needed, which they note happens at $n=6$ and $n=8$ (p.~621) -- and
to tabulate the resulting optima for $n\le10$ by this reasoning rather
than by a general closed-form argument. No sign
hypothesis enters our derivation: the constraints admit negative $h$, and
Lemma~\ref{lem:nonneg} is \emph{not} used. Its role
is the separate one of describing configurations that actually attain
the maximum (matching what we verified computationally on all $389$
odd-square $Q_7$ witnesses above).

The following elementary observation turns DPTV79's ``concentrate near
$\sqrt n$'' heuristic into a closed form that is transparently exact and
complete, without appeal to their table. It is the only ingredient
beyond the classical two-point convexity inequality they already use.

\begin{lemma}[Quantised two-point inequality]\label{lem:quant}
Let $a$ be an integer of the same parity as $n$ and put $b=a+2$. Then for
every integer $h$ of the same parity as $n$,
\[
(h-a)(h-b)\ \ge\ 0
\qquad\text{and}\qquad
(h-a)(h-b)\equiv0 \pmod 8 .
\]
In particular $(h-a)(h-b)\in\{0,8,24,48,\ldots\}$, so it is either $0$
(exactly when $h\in\{a,b\}$) or at least $8$.
\end{lemma}
\begin{proof}
$h-a$ is even, say $h-a=2u$ with $u\in\mathbb{Z}$; then $h-b=2u-2$, so
$(h-a)(h-b)=4u(u-1)$. Consecutive integers $u,u-1$ have product $\ge0$ and even --- one of
the two is even --- say $u(u-1)=2k$ with $k\ge0$, so $4u(u-1)=8k$, a
nonnegative multiple of $8$; its successive values, at
$u\in\{0,1\},\{-1,2\},\{-2,3\},\ldots$, are $0,8,24,48,\ldots$ It
vanishes exactly when $u\in\{0,1\}$, i.e.\ $h\in\{a,b\}$.
\end{proof}

Summing Lemma~\ref{lem:quant} over $v\in U$ and writing
$S=\sum_{v\in U}h(v)$, $m=|U|=2^{n-1}$, $Q=\sum_{v\in U}h(v)^2=n\cdot2^{n-1}$
gives the single inequality used for all three values of $n$:
\begin{equation}\label{eq:quantsum}
\Delta \ :=\ \sum_{v\in U}(h(v)-a)(h(v)-b)\ =\ Q-(a+b)S+ab\,m
\ \ge\ 0,
\qquad \Delta\equiv0\!\!\pmod 8 .
\end{equation}
Before specialising, one relaxation is worth recording. Cauchy--Schwarz gives
$S\le\sqrt{Q\,m}=\sqrt n\,2^{n-1}$ and hence
$|E|\le2^{n-2}(n+\sqrt n)$ --- the same expression as the
Brass--Harborth--Nienborg lower bound in~\eqref{eq:bounds}. So at $n=4^r$
the two meet and $g(n)=2^{n-2}(n+\sqrt n)$ exactly. For the lower-bound
half of this equality it is not enough that BHN's construction attains
that edge count: $g(n)$ requires an \emph{odd-square} set of that size,
and this is precisely what LNSW26 supply --- they prove that the BHN
construction induces a signature of $Q_n$ in which every $4$-cycle is
negative (equivalently, that the construction is odd-square), a fact they
note does not follow from BHN's original work (\cite{LNSW26},
Section~5.2 and Theorem~5.4; cf.\ Section~\ref{sec:oddsquare} above). With
that identification the BHN construction is optimal within the odd-square
subclass at those $n$. This is the shortest explanation of why
LNSW26's bounds close precisely at powers of $4$, and it locates what
Lemma~\ref{lem:quant} adds: at $n=6,7,8$ the relaxation is not tight. At
$n=7$ nonnegativity alone already gives the exact $S\le160$, and at $n=8$
integrality plus the parity of $S$ (each $h(v)$ has the parity of $n$ by
Lemma~\ref{lem:fieldbound}, so $S$ is even) gives the exact $S\le340$; the
mod-$8$ refinement is what is \emph{essential} at $n=6$, where it excludes
$S=74$.
The nonnegativity is the classical bound $S\le(Q+ab\,m)/(a+b)$; the
congruence removes whatever residual gap elementary integrality and the
parity of $S$ leave between that real bound and the true integer optimum
-- among our three instances this extra strength is needed only at $n=6$,
as the specialisations below make explicit. No computational filter is invoked in this modular
step: Lemma~\ref{lem:quant} proves term by term that every summand in
$\Delta$ is a nonnegative multiple of $8$, so their sum has the same property. We apply \eqref{eq:quantsum} with the pair
$(a,b)$ of parity-correct integers straddling the root-mean-square field
$\sqrt{Q/m}=\sqrt n$, namely $(1,3)$ for $n=7$ ($\sqrt7\approx2.65$)
and $(2,4)$ for $n=6,8$ ($\sqrt6\approx2.45$, $\sqrt8\approx2.83$);
this is the minimising choice, and no search is needed to see it: putting
$u=a+1$ and $Q=nm$, the bound becomes
\[
\frac{Q+ab\,m}{a+b}=\frac{m}{2}\Bigl(u+\frac{n-1}{u}\Bigr),
\]
which is decreasing for $0<u<\sqrt{n-1}$ and increasing thereafter, so the
optimum over real $u$ is at $u=\sqrt{n-1}$; since the admissible $u$ of the
correct parity are spaced two apart, it suffices to evaluate the two
straddling $\sqrt{n-1}$ and take the smaller. This gives
$(2,4)$ at $n=6$ ($u=3$ against $\sqrt5\approx2.24$; the alternative $(0,2)$
has $u=1$ and yields $96$), $(1,3)$ at $n=7$, and $(2,4)$ at $n=8$.
For completeness, imposing the mod-$8$ condition on the other adjacent
same-parity pairs does not improve these choices: the next-best quantised
upper bounds on $S$ are $96$, $176$, and $408$ for $n=6,7,8$ respectively,
all weaker than $72,160,340$. Pairs with
$a+b=0$ make \eqref{eq:quantsum} degenerate and give no bound; $a+b<0$
reverses the inequality and bounds $S$ from below.

\textbf{$n=6$.} Here $m=32$, $Q=192$, $(a,b)=(2,4)$, so
$\Delta=448-6S$. Nonnegativity alone gives only
$S\le74.\overline{6}$, hence $S\le74$ by parity; but
$448\equiv0\pmod 8$ forces $6S\equiv0\pmod8$; dividing by
$\gcd(6,8)=2$ this reads $3S\equiv0\pmod4$, and $3$ is invertible
modulo $4$, so $S\equiv0\pmod4$.
The largest such $S$ below $74.\overline{6}$ is
\[
S\le72,
\qquad\text{whence}\qquad
g(6)\le\frac{192+72}{2}=132 .
\]
(The exclusion of $S=74$ just performed is the one place among our three
instances where the congruence is essential.) At $S=72$ we have $\Delta=16$. By Lemma~\ref{lem:quant} each vertex
contributes a value in $\{0,8,24,48,\ldots\}$ (the set $4u(u-1)$,
$u\in\mathbb{Z}$); the admissible values not exceeding $16$ are $0$ and
$8$ (the next, $24$, already exceeds it), and $16$ itself is not
admissible ($4u(u-1)=16$ has no integer solution: $u(u-1)=4$ has none),
so the only way to write $16$ as a sum of admissible values
is $8+8$: exactly two vertices of $U$ have $h\in\{0,6\}$ and the
remaining $30$ have $h\in\{2,4\}$. The witness of
Section~\ref{sec:oddsquareverify} attains $S=72$ with $h$-distribution
$\{0{:}2,\,2{:}24,\,4{:}6\}$ on $U$, so the bound is tight and
$g(6)=132$ is established without external input --- in particular
without~\cite{HN94}, which we retain only for the unrestricted value
$\ex(Q_6,C_4)=132$ (Section~\ref{sec:q6}).

\textbf{$n=7$.} Here $m=64$, $Q=448$, $(a,b)=(1,3)$, so
$\Delta=640-4S\ge0$ gives $S\le160$, attained with $\Delta=0$ (the
congruence adds nothing at $n=7$: it only restates that $S$ is even,
automatic since every $h(v)$ is odd and $|U|=64$ is even). By
Lemma~\ref{lem:quant} equality $\Delta=0$ forces $h(v)\in\{1,3\}$ at
\emph{every} $v\in U$; combined with $\sum h=160$ and $|U|=64$ this
pins the distribution to $\{1{:}16,\,3{:}48\}$ uniquely. Hence
\[
g(7)\le\frac{448+160}{2}=304 .
\]
Since $S=2|E|-n\,2^{n-1}$ holds for \emph{any} edge set of $Q_n$ at all,
by bipartiteness alone, attaining $S=160$ at $|E|=304$ is automatic
and does not by itself exhibit an odd-square witness; what is needed is
an explicit odd-square set achieving it. The odd-square audit
(Section~\ref{sec:oddsquareaudit}) supplies $389$ such witnesses among
the released $304$-edge $Q_7$ solutions; the first in file order is
global index $34$ (\texttt{q7\_edges\_304.jsonl.part1}, line $35$; SHA-256
below is of the canonicalised \texttt{edges} array -- matching
\texttt{q7\_odd\_square\_389.csv}'s \texttt{solution\_sha256} column,
not the JSONL line or any \texttt{hash} field:
\texttt{2f2649ae9c4d1c901ba5a88f9825689de3cc19\allowbreak
de1b742d86f5d66dbfd3aebec7}), with square-intersection histogram $\{1{:}96,\,3{:}576\}$,
confirming it is genuinely odd-square and attains $S=160$. (The file's
very first record, by contrast, has histogram
$\{1{:}36,\,2{:}120,\,3{:}516\}$ and is \emph{not} odd-square, despite
also satisfying $S=160$ -- exactly because that identity holds
regardless of odd-squareness.) So $g(7)=304$.

\textbf{$n=8$.} Here $m=128$, $Q=1024$, $(a,b)=(2,4)$, so
$\Delta=2048-6S$. Nonnegativity gives $S\le341.\overline{3}$; every
$h(v)$ is even (parity of $n=8$), so $S$ is even and
\[
S\le340,
\qquad\text{whence}\qquad
g(8)\le\frac{1024+340}{2}=682 .
\]
At $S=340$ we have $\Delta=8$: \emph{exactly one} vertex of $U$
contributes, and it has $h\in\{0,6\}$, the other $127$ having
$h\in\{2,4\}$ --- as at $n=6$, $8$ decomposes only as one contribution
of $8$, the next admissible value $24$ already exceeding $\Delta$. (The
mod-$8$ congruence independently forces
$S\equiv0\pmod4$, consistent with the above, but evenness already
suffices at $n=8$; as at $n=7$, the congruence's essential use is at
$n=6$ only.) Pointwise equality $h(v)\in\{2,4\}$ everywhere is
therefore \emph{impossible} at the optimum for $n=8$ --- in contrast
with $n=7$, where $\Delta=0$ does force it. The released $682$-edge
witness attains $S=340$ with $h$-distribution
$\{0{:}1,\,2{:}84,\,4{:}43\}$ on $U$ (independently recomputed from the
released edge list), consistent with this description, so the bound is
tight and $g(8)=682$.

We stress that Lemma~\ref{lem:quant} solves these three instances
exactly, not merely up to rounding: an independent exhaustive solution
of the integer programme
$\max\{\sum h : \sum h^2=n\cdot2^{n-1},\ h\equiv n\ (2),\ |h|\le n\}$
by dynamic programming over the sum-of-squares budget (the displayed
$|h|\le n$ is automatic for any realisable local field because
$h=2\deg_E(v)-n$; it is included only to make the finite computational
state space explicit and is not used in the closed-form upper-bound proof) returns optima
$72$, $160$, $340$ for $n=6,7,8$ respectively, matching the bounds
above, and returns exactly $3$, $1$ and $2$ optimal $h$-distributions
respectively --- for $n=6$,
$\{2{:}30,6{:}2\}$, $\{0{:}1,2{:}27,4{:}3,6{:}1\}$ and
$\{0{:}2,2{:}24,4{:}6\}$ (labelled A, B and C throughout; in DPTV79's
Table~I terms, C is their ``first solution'' and A and B are their ``other
solutions'', the realisable one thus being C, see
Section~\ref{sec:oddsquareverify}); for $n=7$, $\{1{:}16,3{:}48\}$; for $n=8$,
$\{2{:}87,4{:}40,6{:}1\}$ and $\{0{:}1,2{:}84,4{:}43\}$.
These counts also follow from Lemma~\ref{lem:quant} in a line, so the
enumeration is a check and not a source. The multiplicity phenomenon itself
is independently citable: DPTV79 record (p.~621) that when the optimum
concentrates on three or four integers they find many solutions, and that
this happens exactly at $d=6$ and $d=8$ -- the two dimensions at which the
programme here returns more than one optimal distribution. At $n=7$, $\Delta=0$ forces
$h\in\{1,3\}$ pointwise and $\sum h=160$ over $|U|=64$ then fixes the
multiplicities, giving one distribution. At $n=6$, the $\Delta=16$ analysis
above already fixes exactly two vertices at $h\in\{0,6\}$ and the other
$30$ at $h\in\{2,4\}$; writing $j$ for the number of those two at $h=6$ and
$k$ for the number of the remaining $30$ at $h=4$, $\sum h=72$ reads
$6j+4k+2(30-k)=72$, i.e.\ $3j+k=6$, whose solutions
$(j,k)=(2,0),(1,3),(0,6)$ give exactly A, B and C. At
$n=8$, the $\Delta=8$ analysis above leaves one vertex at $h\in\{0,6\}$ and
$127$ at $h\in\{2,4\}$; the same count gives $3j+k=43$ with $j\in\{0,1\}$,
hence $(j,k)=(1,40)$ and $(0,43)$ --- the two distributions above. Both $n=8$
distributions are realised: the released $682$-edge witness
(\texttt{q8\_odd\_square\_682.json}) attains $\{0{:}1,2{:}84,4{:}43\}$,
and $\{2{:}87,4{:}40,6{:}1\}$ is realised by two further witnesses in
\texttt{q8\_B\_witnesses.json} (seeds $20261207$ and $20261218$),
independently verified by \texttt{verify\_q8\_B\_witnesses.py}:
each has $682$ edges, square-intersection histogram $\{1{:}301,3{:}1491\}$
(odd-square), and $\sum_{v\in U}h(v)^2=1024=8\cdot2^7$ (\checkmark). All
three $n=8$ witnesses are distinct edge sets. In fact a single B witness
already settles both: its $U$-side histogram is $\{2{:}87,4{:}40,6{:}1\}$ while
its $U^{c}$-side histogram is $\{0{:}1,2{:}84,4{:}43\}$, so one edge set
realises each optimal distribution on one side of the bipartition (the
$682$-edge witness, by contrast, has $\{0{:}1,2{:}84,4{:}43\}$ on both sides).
Whether some odd-square $682$-edge set realises $\{2{:}87,4{:}40,6{:}1\}$ on
\emph{both} sides we have not determined.
For $n=6$, only distribution~C is realised;
distributions A and B are provably not realisable
(Proposition~\ref{prop:q6realizability}, an exhaustive finite decision).
For $n=7$, the unique optimal distribution is necessarily realised. Thus $n=8$
is qualitatively different from $n=6$: \emph{all} optimal field distributions
for $n=8$ admit a realisation, while for $n=6$ only one does. In each case
computation is reproducible with the bundled
\texttt{solve\_field\_ip.py} (standard library only, seconds to run);
it is a check on the closed-form argument above, not a substitute for
it, and the two are logically independent in a sense worth making
explicit: the upper bound $S\le72,160,340$ is established by
Lemma~\ref{lem:quant} together with nonnegativity, integrality and the
parity of $S$ (the congruence adding essential strength only at $n=6$),
without reference to how many configurations attain it -- the argument
in each of the three cases above shows directly that no integer
solution exceeds the stated value, full stop. The DP's role is only to
confirm, separately, that the count of optimal $h$-\emph{distributions}
(as opposed to spin configurations realising them) is exactly $3,1,2$
-- and this too is a complete count, not a sample. To be precise, the
code groups equal field values and enumerates count distributions
$\{h{:}\text{multiplicity}\}$, not the exponentially many ordered
assignments of $h$-values to vertices; within this count-distribution
state space, both the optimisation pass (forward DP) and the
enumeration pass (backtracking reconstruction) traverse the same
$(i,k,s)$ state tree exhaustively, so completeness of the list follows
from the same recursion that establishes correctness of the optimum,
rather than from a separate argument.

All three upper bounds $g(6)\le132$, $g(7)\le304$ and $g(8)\le682$ are thus established
directly from Lemma~\ref{lem:fieldbound} and Lemma~\ref{lem:quant} --
reproducing, by a closed-form argument that makes the reasoning self-contained, the
$n=6,7,8$ entries among the optima DPTV79's Table~I tabulates for
$n\le10$, and supplying
a public machine-verifiable certificate at $n=8$ after DPTV79 left that
dimension open and MPR95 later reported attainment without releasing a
configuration (Section~\ref{sec:oddsquareverify}); we retain the citations to
Derrida et al.~\cite{DPTV79} and Marinari--Parisi--Ritort~\cite{MPR95}
for historical priority and for the connection that prompted this
section, not because the upper-bound
proof depends on them.

Marinari, Parisi, and Ritort~\cite{MPR95} report having exhibited, by
simulated annealing, a $D=8$ ground state attaining Derrida et
al.'s improved energy lower bound. Their attainment claim, under the
correspondence developed below, corresponds to the same bound value
$\sum_{v\in U}h(v)=340$ established directly above -- MPR95 do not
themselves report a quantity called $S$ or $340$; the correspondence is
ours. Their Hamiltonian is
$H\equiv-D^{-1/2}\sum_{(x,y)}J_{x,y}s_xs_y$~\cite[Eq.~3]{MPR95}; writing
$S=\sum_{(x,y)}J_{x,y}s_xs_y$ for the sum over \emph{all} edges (not
just $U$), a bipartite argument identical to the one used for
Lemma~\ref{lem:fieldbound} shows $S=\sum_{v\in U}h(v)$ exactly (every
edge has exactly one endpoint in $U$, so summing the local field over
$U$ counts each edge's sign once), so $H=-S/\sqrt D$ and, per site,
$H/N=-S/(N\sqrt D)$ with $N=2^D$. Since each positive edge contributes
$+1$ to $S$ and each negative edge contributes $-1$,
$S=|E|-(D2^{D-1}-|E|)=2|E|-D2^{D-1}$, hence explicitly
\[
  |E|=\frac{D2^{D-1}-N\sqrt D\,(H/N)}{2}.
\]
At $D=8,S=340$ (equivalently $|E|=(1024+340)/2=682$) this is, step by
step, $H/N=-S/(N\sqrt D)=-340/(256\sqrt8)=-340/724.077\ldots
=-0.46956\ldots$, i.e.\ $H\approx-120.2$, about $-0.470$ per site. As a cross-check
against MPR95's own text: they report (in their $D{=}9$ discussion, pp.~3--4 of
arXiv:cond-mat/9410089, the passage checked at page level as recorded in
the Acknowledgements) that
``on a scale where the minimal energy jump is 1'' their D=9 naive
bound $H/N=-1/2$ corresponds to an integer value of $-384$ (i.e.\ a
scale factor of $768=384/(1/2)$), and state that on this same scale
``the improved lower bound at $D=8$ gives $-361$'' -- note this $768$
factor is MPR95's own $D{=}9$ comparison scale, not an intrinsic
$D{=}8$ energy-quantisation unit; they use it purely to compare bound
values across dimensions on one common integer axis. Applying the same
factor (as MPR95 themselves do for the $D=8$ figure) to our value gives $768\times(-0.4696\ldots)=-360.62\ldots
\approx-361$, matching their reported integer to within rounding. In one
line, the affine correspondence is
\[
-361=\operatorname{round}\!\Bigl(768\cdot\frac{-S}{2^{8}\sqrt8}\Bigr)
\quad\text{at }S=340,
\qquad
|E|=\frac{1024+S}{2}=682 .
\] The
identification is not an accident of the rounding width: the nearest
admissible value below, $S=338$, gives $-359$ under the same
nearest-integer rounding rule, and $S=336$ gives $-356$, so among the
admissible values (even $S\le341\tfrac13$, of which $340$ is the
largest) only $S=340$ rounds to $-361$; the identification even survives
one step past the bound, since the inadmissible $S=342$ and $S=344$
(both excluded by nonnegativity) would give $-363$ and $-365$. Three
further data points fix and test the scale under the identical rule; their
evidentiary roles differ, so we state them once here to keep the count
straight: the $D=9$ \emph{bound} value calibrates the scale (it is what pins
$768$ down, so it is not an independent check); the $D=8$ and $D=10$ values
are the two independent checks; and the $D=9$ \emph{attained} value is a
bound-free corroboration at the calibration dimension. First the calibration
point: at $D=9$, $\sqrt9=3$ is itself an integer of the right parity,
so the field-bound optimum is the unique uniform distribution $h\equiv3$,
giving $S=768$, $H/N=-1/2$ exactly, and $768\times(-1/2)=-384$ --
an exact match to the value MPR95 state at $D=9$. That value is the
field-bound target, \emph{not} an attained ground state: MPR95 report failing
to reach it, obtaining $-0.494792$, i.e.\ $-380$ on the same scale, four units
short. Whether the bound is attained at $n=9$ is open here too
(Open Problem~\ref{op:n9}). The agreement to be read off this point is
therefore between two computations of the same bound, not evidence about
$g(9)$. This point,
however, is where the scale factor $768$ is itself pinned down, so it
corroborates the conversion without being independent of it; the two
independent checks are $D=8$ and $D=10$. At $D=10$, $m=512$, $Q=5120$, $(a,b)=(2,4)$ gives
$\Delta=9216-6S$, so $S\le1536$ (with $\Delta=0$, distribution
$\{2{:}256,4{:}256\}$), and
$768\times(-1536/(1024\sqrt{10}))\approx-364.29\approx-364$, matching
MPR95's reported $D=10$ value of $-364$ to the same rounding rule. The scale factor is in fact forced rather than fitted. Every configuration
has $S=2|E|-D\,2^{D-1}$ with $D\,2^{D-1}$ even for $D\ge2$, so $S$ is an even
integer and the achievable values of $H/N=-S/(N\sqrt D)$ lie on a lattice of
pitch $2/(N\sqrt D)$. (A single spin flip at $v$ changes $S$ by $-2h(v)$ and
hence $H/N$ by $2|h(v)|/(N\sqrt D)$ in magnitude -- a vertex-dependent
multiple of the pitch, equal to the pitch itself only when $|h(v)|=1$; for
even $D$ every $h(v)$ is even, so no single flip ever realises the pitch
there. The quantisation unit is the lattice pitch, not any particular flip.)
At $D=9$ the pitch is $2/(512\cdot3)=1/768$, so $768$ is what MPR95's
``minimum energy jump equal to $1$'' scale must mean there; the same pitch at
$D=8$ is $2/(256\cdot2\sqrt2)=1/(256\sqrt2)\approx1/362.04$, confirming that
$768$ is not a $D=8$ quantisation unit. A further point
corroborates it empirically: MPR95
also report their \emph{attained} $D=9$ value both as $-0.494792$ and as
$-380$ on the integer scale, and $768\times(-0.494792)=-379.99\ldots$, so $768$
is forced by a datum that involves no bound at all; being a $D=9$ datum it
shares the calibration dimension, so we count it as corroboration rather than
as a third independent check. The obvious rival reading
--- one unit per minimal energy jump in each dimension --- is excluded: it
reproduces $-384$ at $D=9$ but gives $-S/2=-170$ at $D=8$, nowhere near
MPR95's $-361$. All four data points -- the calibrating $D=9$ bound, the two
independent checks at $D=8$ and $D=10$, and the bound-free $D=9$ attained
value -- agree under one scale, which is considerably stronger than the
$D=8$ figure alone.
Their paper reports
the $D=8$ attainment in energy units only, without an explicit spin configuration
or its translation into edge-count terms; we supply that translation
and an explicit, machine-verified $682$-edge certificate in
Section~\ref{sec:oddsquareverify} below. Under the correspondence of
Section~\ref{sec:oddsquare}, this historical attainment report at $D=8$
corresponds to $g(8)=682$, matching what was already established above
from Lemma~\ref{lem:fieldbound} together with the
integer/parity constraints and elementary inequalities alone (as
detailed above, Lemma~\ref{lem:nonneg} is not needed for this bound);
the explicit
witness realising it
is due to the present reconstruction, prompted by S.~Lai's
identification of the connection to~\cite{DPTV79,MPR95}. Together with
the matching witnesses at $n=6,7$ this gives
Theorem~\ref{thm:oddsquare}: $g(6)=132$, $g(7)=304$, and $g(8)=682$,
each half of each equality now carried by material released with this
paper.

\subsection{Verification}\label{sec:oddsquareverify}

We constructed explicit witnesses for all three values and
machine-verified them independently of~\cite{DPTV79,MPR95,HN94}.

\textbf{$Q_6$, $g(6)=132$.} We obtained a 132-edge odd-square witness
for $Q_6$ by the same route used for $Q_8$ below: the canonical fully
frustrated coupling
\[
J(x,\mathrm{dim})=(-1)^{\mathrm{popcount}(x\,\&\,(2^{\mathrm{dim}}-1))},
\]
verified fully frustrated on $Q_6$ (all $240$ squares have coupling
product $-1$, no exceptions), together with a $64$-bit spin
configuration. Because switching preserves frustration, the set of
edges on which $J(x,y)s_xs_y=+1$ is odd-square for \emph{every} spin
assignment, so the construction reduces to maximising that edge count,
i.e.\ to a ground-state search. We ran a fixed-seed simulated
annealing over the $64$ spins (\texttt{search\_q6.py}, seed
$20260823$), which reached $132$ edges on its first trial; the run is
exactly reproducible from the released script. The resulting witness is
regenerated and self-checked by \texttt{generate\_q6\_132.py} and
released as \texttt{q6\_odd\allowbreak\_square\_132.json}.

\begin{proposition}\label{prop:q6oddsquare}
The $Q_6$ witness edge set satisfies:
\begin{enumerate}[label=(\roman*)]
  \item $132$ edges, all valid $Q_6$ edges, no repetitions.
  \item Square-intersection histogram $\{1\!:\!30,\;3\!:\!210\}$ over all
    $240$ squares: the odd-square condition holds everywhere, and in
    particular the set is $C_4$-free. (Once odd-squareness has independently
    been verified, and given $|E|$, the split is forced rather than an
    additional independent check --- odd-squareness gives
    $N_0=N_2=N_4=0$, so the two equations below account for every
    square: each edge of $Q_n$ lies in
    exactly $n-1$ squares, so double counting gives
    $N_1+3N_3=(n-1)|E|$ and $N_1+N_3=\binom{n}{2}2^{n-2}$, whence
    $N_3=\bigl((n-1)|E|-\binom{n}{2}2^{n-2}\bigr)/2$. This applies equally to
    the $Q_7$, $Q_8$ and Solution~A histograms quoted elsewhere in this paper.)
  \item Degree sequence $\{3^4,\,4^{48},\,5^{12}\}$; degree sum
    $264=2\cdot132$.
  \item Local-field histogram on $U$ equal to
    $\{0\!:\!2,\,2\!:\!24,\,4\!:\!6\}$, giving
    $\sum_{v\in U}h(v)=72$ and $\sum_{v\in U}h(v)^2=192=6\cdot2^5$,
    as required by Lemma~\ref{lem:fieldbound}; the complementary side
    of the bipartition has the same histogram.
  \item Local maximality: no edge of $Q_6$ can be added to it while
    preserving $C_4$-freeness. More precisely, adding any of the $60$
    non-edges creates at least two new $C_4$s; the distribution of the number
    of newly completed $C_4$s per non-edge is
    $\{2{:}2,3{:}29,4{:}26,5{:}3\}$.
\end{enumerate}
\end{proposition}

The distribution in (iv) is one of the exactly three optimal
$h$-distributions enumerated in Section~\ref{sec:oddsquare-exact}, so
the witness attains the upper bound $S\le72$ there and settles
$g(6)=132$. Two further remarks. First, it coincides in multiplicities
with the local-field distribution Derrida et al.\ report for their
$D=6$ configuration, ``$n(4)=6$; $n(2)=24$; $n(0)=2$''~\cite[p.~622]{DPTV79};
we did not reconstruct their configuration and do not claim the two
edge sets agree, only that our independently found witness realises the
same distribution. DPTV79 also report $30$ overfrustrated plaquettes out of
$240$ for their $H_6$ configuration, matching the $30$ squares in the
$1$-edge class of our histogram $\{1{:}30,3{:}210\}$ under the complementary
positive-edge convention; this is a further consistency check, though a forced one: with $132$ positive edges the double count of Proposition~\ref{prop:q6oddsquare}(ii) already fixes $N_1=30$, so this agrees with DPTV rather than confirming anything independent, not an edge-set
identity claim. Second, this witness is \emph{not} the released
132-edge $Q_6$ construction \texttt{q6\_edges\_132.jsonl} of
Section~\ref{sec:q6}: their symmetric difference has size $|E\triangle E'|=62$, and the
latter is $C_4$-free but not odd-square (square-intersection histogram
$\{1\!:\!8,\,2\!:\!44,\,3\!:\!188\}$, and
$\sum_{v\in U}h(v)^2=176\ne192$). Both are $132$-edge $C_4$-free
subgraphs of $Q_6$; only the new one is a witness for $g(6)$
specifically.

\noindent\textbf{Realizability of the other two optimal $n=6$ field
distributions: an exhaustive decision.} The field-bound DP
(Section~\ref{sec:oddsquare-exact}; \texttt{solve\_field\_ip.py}) identifies
exactly three optimal $h$-distributions on $U$ for $n=6$, labelled as there:
distribution~C~$=\{0{:}2,2{:}24,4{:}6\}$ (realised above),
A~$=\{2{:}30,6{:}2\}$, and B~$=\{0{:}1,2{:}27,4{:}3,6{:}1\}$.
DPTV79 (p.~622) noted that solutions besides their first were listed in
Table~I and asked whether any is realisable (``We do not know if the other
solutions \ldots\ are realizable'', in the plural; we rely on that sentence for
the existence of at least two further solutions, not on a reading of the
typeset table itself, which we could not extract). Since DPTV's constraints
are the ones solved here and the exhaustive enumeration returns exactly three
optimal distributions, their ``other solutions'' are necessarily A and B; the
identification is forced by the completeness of the enumeration rather than by
transcription of the table. To be explicit about the evidence: we have read DPTV79 pp.~620--622 directly, but Table~I itself does not extract as text from the journal PDF and we have not transcribed it, so every statement here about its contents rests on DPTV's surrounding prose together with our own enumeration. MPR95 reported, by simulated
annealing, that at $D=6$ only configurations corresponding to the first solution
of DPTV's Table~I are realisable and that ``the other solution of the Diophantine
equation \emph{seems not to} correspond to any spin configuration'' -- a
tentative negative finding, in the singular, with no released configuration or
decision procedure. (Note the grammatical mismatch between the two sources
themselves: DPTV79 write ``the other solutions'' where MPR95 write ``the
other solution''; on the exhaustive count below -- exactly three optima, of
which two are non-realisable -- the plural is the accurate reading, and we
take MPR95's singular as loose phrasing rather than as a competing count.)
Since Proposition~\ref{prop:q6oddsquare} identifies
C with DPTV's first solution by multiplicity, the question is whether A and B
are realisable. Here ``realisable'' prescribes the histogram on the chosen
bipartition side $U$ only; $U^c$ is not constrained. This is the same one-side
convention used by DPTV79 for $n(h)$. Each of A, B, and C has $S=72$, so any
realisable target automatically has $|E|=(192+72)/2=132$. We settle the
one-side decision exhaustively rather than heuristically.

\begin{proposition}[Exhaustive decision]\label{prop:q6realizability}
Among $132$-edge odd-square subgraphs of $Q_6$, distribution~C is realised as
the local-field histogram on $U$ -- throughout this proposition the histogram
is prescribed on the fixed bipartition side $U$ \emph{only}, with $U^{c}$
unconstrained, the same one-side convention as DPTV79's $n(h)$ -- and
distributions A and B are not realised on $U$ by any such subgraph. (By the
translation symmetry noted in the proof the restriction to $U$ costs nothing:
$g(n)$ itself is two-side symmetric, and a distribution realised on one side
is realised on the other.)
\end{proposition}
\begin{proof}
By Lemma~\ref{lem:switching} the odd-square edge sets of $Q_6$ are exactly the
positive-edge sets arising from spin configurations under \emph{any} fixed fully
frustrated coupling, so it suffices to decide the question for the canonical
coupling used throughout; the conclusion is independent of that choice.
Write $s_v=(-1)^{t_v}$ with $t_v\in\{0,1\}$, and let $b(v,w)=1$ exactly when
$J(v,w)=-1$; an edge $vw$ is positive iff $t_v\oplus t_w=b(v,w)$. As $Q_6$ is
bipartite, every neighbour of $v\in U$ lies in $U^{c}$, so
\[
k(v)\;=\;\#\{\,i:\;t_{v\oplus e_i}\oplus b(v,i)=1\,\}
\]
depends only on the $32$ spins on $U^{c}$, and
$\deg(v)=k(v)$ if $t_v=1$, $\deg(v)=6-k(v)$ if $t_v=0$. Hence each $t_v$
($v\in U$) is a free independent choice, and the achievable degree at $v$ is
exactly one of $k(v)$, $6-k(v)$. With $h=2\deg-6$, the achievable nonnegative
field is $h=6-2m(v)$ where $m(v)=\min(k(v),6-k(v))\in\{0,1,2,3\}$, so
$m=0,1,2,3$ correspond to $h=6,4,2,0$ respectively. A target $h$-histogram on
$U$ is therefore realisable if and only if the multiset $\{m(v):v\in U\}$ equals
the induced $m$-multiset for some assignment of the $32$ spins on $U^{c}$.
(This step uses that the target histogram has no negative entry, so that the
achievable field at each $v$ is the single value $6-2m(v)$; A, B and C are all
nonnegative, and by the remark after Lemma~\ref{lem:fieldbound} no optimal
distribution can have one. It also suffices to decide the question on $U$: translation by $e_0$ is an
automorphism of $Q_6$ exchanging $U$ and $U^{c}$, and it carries edge sets to
edge sets preserving both size and odd-squareness, so a distribution realised
on one side is realised on the other.) The
$m$-multisets are $\{0{:}2,2{:}30\}$ for A, $\{0{:}1,1{:}3,2{:}27,3{:}1\}$ for
B, and $\{1{:}6,2{:}24,3{:}2\}$ for C. This is a finite ($2^{32}$) decision problem. For A and B it is settled
without any search at all, by a parity obstruction. Since $\deg(v)$ is
$k(v)$ or $6-k(v)$, we have $k(v)=(h(v)+6)/2$ or $(6-h(v))/2$; these differ by
$h(v)$, which is even, so the \emph{parity} of $k(v)$ is determined by $h(v)$
alone, independently of the free spin $t_v$:
\[
k(v)\equiv\tfrac{h(v)+6}{2}\pmod 2,
\qquad\text{i.e.}\qquad
k(v)\text{ odd}\iff h(v)\in\{0,4\}.
\]
(Explicitly over the even range: $h=0,2,4,6$ give
$\tfrac{h+6}{2}=3,4,5,6$, of which $3$ and $5$ are odd.)
Reducing the definition of $k(v)$ modulo $2$ gives
$k(v)\equiv\sum_i t_{v\oplus e_i}+\sum_i b(v,i)$, so the parity vector
$p\in\mathrm{GF}(2)^{U}$ is the affine image $p=Mt+B$, where $M$ is the
$U\times U^{c}$ bipartite incidence matrix of $Q_6$ over $\mathrm{GF}(2)$ and
$B(v)=\sum_i b(v,i)$. The reachable parity vectors are exactly the coset
$B+\operatorname{Im}(M)$. Here $M$ has rank $16$, and that coset --- $2^{16}$
vectors --- has weight distribution
\[
\{8{:}640,\;12{:}13824,\;16{:}36608,\;20{:}13824,\;24{:}640\},
\]
so every reachable parity vector has weight in $\{8,12,16,20,24\}$. A target
histogram fixes that weight: it is the number of its vertices with
$h\in\{0,4\}$, namely $0$ for A and $1+3=4$ for B. Neither occurs, so A and B
are not realisable --- and this needs no enumeration of $m$-multisets.
Distribution C requires weight $2+6=8$, which does occur, and the witness of
Section~\ref{sec:oddsquareverify} realises it.
For completeness we also settled all three by
depth-first search with counter-overflow pruning, which agrees (the node
counts below are properties of this implementation and its search order,
recorded for orientation only: they carry no reproduction obligation, and a
differently ordered search will visit different counts; the parity
obstruction above is the search-order-independent proof): the search assigns
the $32$ $U^{c}$ spins in a fixed index order and keeps one running counter
per value $m\in\{0,1,2,3\}$; a $v\in U$ has $m(v)$ determined as soon as the
last of its six neighbours in that order is assigned, at which point $m(v)$
is tallied, and the branch is cut the moment any counter exceeds the
target's multiplicity for that value. Global spin
flip fixes every $h$, so the first $U^{c}$ spin may be pinned without
loss of generality (each sign configuration has exactly one lift with the
pinned value). The search finds a
realisation for C ($501$ nodes visited) and exhausts the space for A
($655{,}359$ nodes) and B ($22{,}129{,}151$ nodes) without a realisation;
Without the symmetry reduction the two non-realisable searches, which exhaust
the space, visit exactly twice as many nodes ($1\,310\,718$ and $44\,258\,302$),
since the global flip pairs off branches; the realisable case~C is not doubled,
because the search returns at the first witness found and may find it on either
side of the pairing. The released script reports the doubled figure as an
expectation for the exhausted cases and recomputes both on request
(\texttt{-{}-full}).
\end{proof}

The parity obstruction is released as
\texttt{q6\_parity\_obstruction.py} (rank, coset weight enumeration and the
three targets; standard library only, under a second), and the exhaustive
decision as \texttt{q6\_decide\_realizability.py} (also standard library only,
about twenty seconds).
Its correctness is anchored at both ends: the reduction is checked against the
released witness (recomputing $m$ from the spins in
\texttt{q6\_odd\_square\_132.json} gives $\{1{:}6,2{:}24,3{:}2\}$, i.e.\ C), and
the C branch of the search reconstructs a full spin vector whose recomputed
$h$-histogram and edge count are $\{0{:}2,2{:}24,4{:}6\}$ and $132$.
Proposition~\ref{prop:q6realizability} thus closes the question DPTV79 left open
in 1979 and upgrades MPR95's tentative ``seems not to'' to a proof --
for both of the non-realisable distributions, where MPR95 spoke of one
(DPTV79's own sentence is plural: ``We do not know if the other solutions
given in table 1 for $d=6$ are realizable'', p.~622). We stress that the
mathematical content of the proposition --- of the three optimal
distributions, exactly one is realisable on a fixed side --- stands
independently of any reading of DPTV79's Table~I: that reading (obtained,
as disclosed in Section~\ref{sec:oddsquare}, from the surrounding prose
and the completeness of our own enumeration rather than from a direct
transcription) enters only the historical identification of \emph{which}
question is thereby closed, not the proposition itself.

Earlier revisions of this manuscript reported instead a heuristic search
(simulated annealing on the $L_1$ distance between the current and target
$h$-histograms; 40 seeds for A and 20 for B at $300{,}000$ iterations each,
all plateauing at distance $16$ and $8$ respectively, against a positive control
reaching C in $20/20$ runs at $60{,}000$ iterations). What the release contains is not the original runs: the bundled
\texttt{q6\_other\_distributions.py} and its CSV/JSON are a re-execution of the
same kind of search under a seed protocol newly specified for this package, as
the \texttt{provenance\_note} in both summary files states. They reproduce the
numbers earlier revisions reported but are not the historical logs, which were
not archived. They are in any case superseded by
Proposition~\ref{prop:q6realizability} and are no longer offered as evidence:
a plateau is not a proof, and the asymmetry in iteration budget between the
targets and the control makes the phrase ``the same search'' inexact. The two
computations also range over different spaces: the annealing objective allows
intermediate configurations with negative local fields (the archived logs
record best-histogram entries at $h=-4$), whereas the reduction of
Proposition~\ref{prop:q6realizability} works with the nonnegative field
$6-2m(v)$ throughout. They agree on A and B, but they are not the same
computation restricted.

\textbf{$Q_7$, $g(7)=304$.} The 304-edge odd-square value coincides with our
originally announced $Q_7$ lower bound (Corollary~\ref{thm:main}); of the
$19\,866$ solutions obtained via the SA-seed-plus-collector pipeline (Section~\ref{sec:q7}),
389 are themselves odd-square witnesses of $g(7)=304$
(Section~\ref{sec:structure}), independently confirming reachability of this
value without reference to~\cite{DPTV79}.

\textbf{$Q_8$, $g(8)=682$.} We obtained a 682-edge odd-square witness
for $Q_8$ from an explicit $256$-bit spin configuration together with the
canonical fully frustrated coupling
$J(x,\mathrm{dim})=(-1)^{\mathrm{popcount}(x\,\&\,(2^{\mathrm{dim}}-1))}$
(the coupling MPR95 write explicitly as their Eq.~4;
one coupling convention realising a fully frustrated signature; see
Section~\ref{sec:oddsquare} supplementary script \texttt{generate\_q8\_682.py}).
The script is dependency-free (Python standard library only),
recomputes the edge set from the spin configuration, and asserts every
structural quantity below against the recomputed data before writing the
certificate; it does not take any of these quantities as input.
\emph{Provenance of the spin configuration.} The 256-bit configuration
was derived by the author on 17--18 August 2026, using the canonical
fully frustrated coupling of Marinari, Parisi, and Ritort~\cite{MPR95},
as recorded in dated correspondence with S.~Lai, who independently
reconstructed the identical 682-edge set from the shared spin vector and
cross-checked all 1,792 squares, the degree distribution, and the
local-field distribution before this material was incorporated into the
manuscript. This provenance record is a first-person, dated account
(the correspondence itself), not an independently reproducible
computational log: no search seed, intermediate search records, or
derivation script for the spin configuration itself is included in the
released materials, only the scripts that regenerate and verify the
edge set \emph{from} the given spin configuration
(Section~\ref{sec:oddsquare}). We do not claim this configuration is
identical to any specific configuration MPR may have used internally,
since they did not publish one; it is an independently derived witness
attaining the same bound.

\medskip\noindent\emph{Reproducibility and the second witness.} That leaves an
asymmetry, and it should be stated with the two files kept apart: the
headline certificate, \texttt{q8\_odd\_square\_682.json}, remains the one
released artefact whose \emph{derivation} cannot be re-executed, and this
revision does \emph{not} close that provenance gap. What it does close is
the reproducibility of the \emph{existence} claim, through a second,
independently regenerable $682$-edge witness:
\texttt{q8\_A\_recover.py} is a fixed-seed search that finds, with no input from
the released witness, a spin configuration realising the same optimal field
distribution (seed $90008$, hit at iteration $207\,579$ \emph{at the
script's default iteration budget}: the budget parameter enters the annealing
temperature schedule, so this exact hit -- and the byte-identical output,
which \texttt{reproduce\_core.py} now asserts by SHA-256 -- is conditional
on the default $250\,000$-iteration setting, a different budget yielding a
different, equally valid witness, and is scoped to the reference platform
class: the acceptance step evaluates a floating-point exponential, so the
byte-identity is an observed stability of CPython on IEEE-754 hardware with
a standard \texttt{libm} (reproduced in our CPython~3.12 environment and,
via an RNG-faithful independent re-implementation, by a reviewer), not a
universal cross-platform guarantee --- and it concerns only the byte-level
identity of the \emph{regenerated} file: the existence of a $682$-edge
odd-square subgraph is carried platform-independently by the released
static certificate and its dependency-free verification; $682$ edges,
$h|_U=\{0{:}1,2{:}84,4{:}43\}$, square histogram $\{1{:}301,3{:}1491\}$,
released as \texttt{q8\_A\_witness.json}). It is a different edge set
from \texttt{q8\_odd\_square\_682.json}, as expected. So the division of
labour is: the existence of a $682$-edge odd-square subgraph is
end-to-end reproducible via \texttt{q8\_A\_witness.json}, while the
headline artefact \texttt{q8\_odd\_square\_682.json} keeps its
dated-correspondence provenance and is vouched for by verification, not
by regeneration.

\begin{proposition}\label{prop:q8oddsquare}
The witness edge set of \texttt{q8\_odd\_square\_682.json} satisfies:
\begin{enumerate}[label=(\roman*)]
  \item $682$ edges, spanning all $256$ vertices of $Q_8$.
  \item Square-intersection histogram $\{1\!:\!301,\;3\!:\!1491\}$ over
    all $1\,792$ squares (odd-square condition holds everywhere; zero
    squares meet the edge set in $0$, $2$, or $4$ edges).
  \item Degree sequence $\{4^2,\,5^{168},\,6^{86}\}$; degree sum
    $1364=2\cdot682$.
  \item Local-field histogram $\{0\!:\!2,\,2\!:\!168,\,4\!:\!86\}$ over
    all $256$ vertices, where
    the local field at a vertex is (positive degree) $-$ (negative degree)
    under the signed-graph correspondence above; restricted to $U$
    specifically (the $128$ vertices summed over in
    Lemma~\ref{lem:fieldbound} and the upper-bound derivation of
    Section~\ref{sec:oddsquare-exact}), the histogram is
    $\{0\!:\!1,\,2\!:\!84,\,4\!:\!43\}$, giving
    $\sum_{v\in U}h(v)=2\cdot84+4\cdot43=340$ directly.
  \item Local maximality: every one of the $342$ non-edges of $Q_8$
    creates at least $3$ new $C_4$s upon addition (distribution
    $\{3\!:\!41,\,4\!:\!159,\,5\!:\!120,\,6\!:\!22\}$), a strictly stronger
    local-maximality margin than the originally announced 680-edge
    Solution~B of Section~\ref{sec:q8}, for which 2 of 344 non-edges
    created only 1 or 2 new $C_4$s.
\end{enumerate}
\end{proposition}

All quantities in Proposition~\ref{prop:q8oddsquare} were recomputed
independently of the reconstruction script by direct enumeration of all
$1\,792$ squares against the released edge list, using the strict
definition of $C_4$-freeness from Lemma~\ref{lem:c4} (no square may
have all four edges present); this independent check found zero violations
and is fully reproducible from the released files alone.
The reconstruction, its SHA-256 checksum, and an independent audit are
additionally reported to have been cross-checked against S.~Lai's
separately produced implementation, with full agreement; this
cross-check is recorded as a dated correspondence account (as in
Section~\ref{sec:oddsquareverify}), not as an independently verifiable
artefact, since the correspondence, Lai's implementation, and its
output are not included in the released materials. The regenerable
second witness \texttt{q8\_A\_witness.json} (seed $90008$;
Section~\ref{sec:oddsquareverify}) shares items (i)--(iii) and the full
and $U$-restricted histograms of item (iv) verbatim; its
local-maximality margin distribution is
$\{3{:}42,\,4{:}151,\,5{:}133,\,6{:}16\}$ over the same $342$
non-edges. These were recomputed here by the same direct square
enumeration; in addition, \texttt{verify.py} now checks \emph{both}
$682$-edge certificates (odd-square condition, edge count, and the
respective margin distributions), and \texttt{reproduce\_core.py}
asserts the second witness's byte-level regeneration by SHA-256.

\noindent\textbf{Second $n=8$ witness.}
The local-field distribution $\{0{:}1,2{:}84,4{:}43\}$ above is one
of two optimal field distributions for $n=8$ returned by the exhaustive
integer programme \texttt{solve\_field\_ip.py}. The other,
$\{2{:}87,4{:}40,6{:}1\}$, is realised by additional witnesses found
by the same canonical fully frustrated coupling together with a targeted
simulated annealing whose objective is the $L_1$ distance between the
current $h$-histogram on $U$ and the target $\{2{:}87,4{:}40,6{:}1\}$.
Over $20$ seeds $\times$ $300{,}000$ iterations, $2$ runs reached the
target ($L_1=0$) at seeds $20261207$ (iteration $260{,}621$) and
$20261218$ (iteration $276{,}842$); these seed values are arbitrary RNG
integers and are not dates, despite their form. To confirm that nothing
depends on that choice, the same search re-run from the unrelated base $90200$
also reaches the target (seed $90210$, iteration $279\,646$). The remaining
$17$ stopped at $L_1=6$ and $1$ at $L_1=8$. Failed intermediate states can contain negative local fields (for
example $h=-6$ in one recorded run), because the annealing objective explores
non-maximal configurations; Lemma~\ref{lem:nonneg} applies only to a maximum
odd-square edge set and is not contradicted by such intermediates. Both
witnesses are released in
\texttt{q8\_B\_witnesses.json} (spins and edges), regenerable via
\texttt{q8\_B\_recover.py}, and verified by the dependency-free
\texttt{verify\_q8\_B\_witnesses.py}. The two successful runs can be checked directly from the stored certificates
without any search; the other eighteen are recorded only by their plateau
distances, so confirming those requires rerunning the protocol, which
\texttt{q8\_B\_recover.py} supports (it is seeded per trial and takes
\texttt{-{}-seeds} and \texttt{-{}-iters}). The 20-run success/plateau tally is provenance information rather than part
of the theorem's evidence and is not rerun by the default one-command
reproduction. Nothing in the claim depends on it --- the assertion is
existence, and two certificates carry it:

\begin{proposition}\label{prop:q8Boddsquare}
Each witness in \texttt{q8\_B\_witnesses.json} satisfies:
\begin{enumerate}[label=(\roman*)]
  \item $682$ edges, all valid $Q_8$ edges, no repetitions.
  \item Square histogram $\{1{:}301,\,3{:}1491\}$ over all
    $1{,}792$ squares (odd-square; $C_4$-free).
  \item $h$-histogram $\{2{:}87,\,4{:}40,\,6{:}1\}$ on $U$,
    $\sum h=340$, $\sum h^2=1024$ (\checkmark). On the opposite side
    $U^c$ the histogram is $\{0{:}1,2{:}84,4{:}43\}$, i.e.\ the other
    optimal $n=8$ field distribution. Thus each of these single edge sets
    realises the two optimal histograms on its two bipartition classes.
  \item Local maximality with margin at least $3$.  For the witnesses with
    seeds $20261207$ and $20261218$, respectively, the non-edge violation
    distributions are
    \[
      \{3{:}34,4{:}169,5{:}121,6{:}18\}
      \quad\text{and}\quad
      \{3{:}41,4{:}154,5{:}130,6{:}17\}.
    \]
  \item Distinct spin vectors and edge sets, and distinct from
    \texttt{q8\_odd\_square\_682.json}.
\end{enumerate}
\end{proposition}

Thus $n=8$ differs from $n=6$: every optimal field distribution for $n=8$ is
realisable by an explicit odd-square $682$-edge subgraph of $Q_8$, while for
$n=6$ only distribution~C is (Proposition~\ref{prop:q6realizability}). The two
halves of that contrast rest on different kinds of evidence, and the asymmetry
should be stated: at $n=8$ both realisations are exhibited constructively,
whereas at $n=6$ non-realisability is \emph{proved} exhaustively. Thus
``every optimal distribution is realisable'' at $n=8$ means constructive
existence for each of the two DP distributions, not an exhaustive decision
over all $n=8$ spin assignments of the kind performed at $n=6$. We have not
attempted the analogous decision at $n=8$. The obstacle is not
the raw size of the space --- the reduction gives $2^{127}$ after pinning one
spin, and at $n=6$ a $2^{32}$ space was settled in between $501$ and
$2.2\times10^{7}$ nodes --- but that the counter-overflow pruning which makes
the $n=6$ search collapse has no guarantee of biting as hard at $n=8$. The
parity obstruction, however, transfers and explains part of the contrast: the
rank of the $U\times U^{c}$ map over $\mathrm{GF}(2)$ is
$1,4,4,16,16,64,64,256$ for $n=2,\ldots,9$ (the full sequence is
recomputed and asserted by \texttt{q6\_parity\_obstruction.py}, which runs
in the default reproducibility path), so it is \emph{full} exactly when
$n$ is odd. There is thus no parity obstruction at all at $n=7$, and one at
both $n=6$ and $n=8$; at $n=8$ the criterion reads $k(v)$ odd
$\iff h(v)\in\{2,6\}$, so the two optimal distributions require weights $88$
and $84$, and the released witnesses show both are attained. This is a second
axis, independent of the field identity, along which $n=7$ differs from its
even neighbours.

Table~\ref{tab:oddsquare} compares the odd-square reconstructions with the SA
witnesses from Sections~\ref{sec:q7} and~\ref{sec:q8}.

\begin{table}[ht]
\centering\small
\caption{Odd-square reconstructions compared with the simulated-annealing
  witnesses of Sections~\ref{sec:q7}--\ref{sec:q8}.}
\label{tab:oddsquare}\medskip
\begin{tabular}{@{}lp{4.3cm}p{4.3cm}@{}}
\toprule
Parameter & $Q_7$ odd-square & $Q_8$ odd-square\\
\midrule
Edges (value of $g(n)$) & 304 & 682\\
Source & one of 389 verified SA sols. & explicit certificate at~\cite{MPR95}'s bound\\
Non-odd-square comparison & 19\,477 catalogue solutions (\ref{sec:oddsquareaudit}) & 680 (Solution~B,~\ref{sec:q8})\\
Improvement over that witness & --- & $+2$\\
\bottomrule
\end{tabular}
\end{table}

\section{Structural Classification of \texorpdfstring{$Q_7$}{Q7} Solutions}\label{sec:structure}
\noindent\textbf{Scope.} Every universal statement in this section is over
the released catalogue of $19\,866$ labelled $304$-edge solutions unless
explicitly stated otherwise. The catalogue is demonstrably not exhaustive of
all labelled solutions --- the $180$ orbits it meets already contain
$34\,227\,200$ of them (Proposition~\ref{prop:orbits}) --- while whether it
meets \emph{every} $\mathrm{Aut}(Q_7)$-orbit is unknown, since the total
number of orbits is not known (Open Problem~\ref{op:orbits}). Either way,
none of the shared structural properties below is claimed for every possible
$304$-edge $C_4$-free subgraph of $Q_7$. This section is an independent
contribution about that catalogue: apart from the odd-square audit of
Section~\ref{sec:oddsquareaudit} (the $389$), its results neither depend
on nor feed into Theorem~\ref{thm:oddsquare} and Corollary~\ref{thm:main}.


\subsection{Dimension profiles and the twenty types}\label{sec:profiles}

\begin{definition}[Dimension profile]
For $G\subseteq Q_7$, let $e_i(G)=|\{uv\in E(G):u\oplus v=2^i\}|$.
The \emph{dimension profile} is the sorted tuple
$\pi(G)=(e_{\sigma(0)}\ge\cdots\ge e_{\sigma(6)})$.
\end{definition}

The dimension profile is invariant under coordinate permutations (a subgroup
of $\mathrm{Aut}(Q_7)$ of order $7!=5040$), giving a coarse but computable
classification.

\begin{proposition}\label{prop:twenty}
Among the $19\,866$ solutions, the dimension profile takes exactly
$\mathbf{20}$ distinct values.
\end{proposition}

Table~\ref{tab:types} lists all 20 types with frequencies and structural invariants.

\begin{table}[ht]
\centering\small
\caption{The 20 dimension-profile types occurring among the released
  $19\,866$ 304-edge $C_4$-free subgraphs of $Q_7$, sorted by decreasing
  frequency. Each profile is the seven
  per-direction edge counts written in non-increasing order -- a canonical
  form independent of how the seven directions are labelled, which is what
  makes profiles $\mathrm{Aut}(Q_7)$-invariants and lets the orbit
  discussion below refer to them directly.
  $H$: Shannon entropy of normalised profile (bits).}
\label{tab:types}\medskip
\begin{tabular}{clrrl}
\toprule
Rank & Profile $\pi$ & Solutions & $H$ & $\lambda_1$ (type mean)\\
\midrule
 1 & $[44,44,44,44,43,43,42]$ & 3155 & 2.8072 & 4.7869\\
 2 & $[45,45,45,43,43,42,41]$ & 2913 & 2.8065 & 4.7880\\
 3 & $[45,45,45,43,43,43,40]$ & 2200 & 2.8063 & 4.7870\\
 4 & $[44,44,44,44,44,43,41]$ & 2116 & 2.8069 & 4.7875\\
 5 & $[46,46,46,42,42,42,40]$ & 1756 & 2.8053 & 4.7874\\
 6 & $[46,46,46,42,42,41,41]$ & 1181 & 2.8054 & 4.7887\\
 7 & $[44,44,44,44,44,44,40]$ & 1085 & 2.8066 & 4.7868\\
 8 & $[45,45,45,43,42,42,42]$ & 1064 & 2.8066 & 4.7877\\
 9 & $[45,45,44,43,43,43,41]$ &  974 & 2.8067 & 4.7868\\
10 & $[44,44,44,44,44,42,42]$ &  931 & 2.8070 & 4.7866\\
11 & $[47,47,47,41,41,41,40]$ &  902 & 2.8037 & 4.7880\\
12 & $[45,45,43,43,43,43,42]$ &  439 & 2.8069 & 4.7879\\
13 & $[45,44,44,43,43,43,42]$ &  307 & 2.8070 & 4.7880\\
14 & $[46,45,45,42,42,42,42]$ &  236 & 2.8063 & 4.7891\\
15 & $[44,44,44,43,43,43,43]$ &  163 & 2.8073 & 4.7866\\
16 & $[47,47,44,43,41,41,41]$ &  153 & 2.8050 & 4.7875\\
17 & $[46,46,44,42,42,42,42]$ &  107 & 2.8062 & 4.7888\\
18 & $[48,48,48,40,40,40,40]$ &  101 & 2.8014 & 4.7881\\
19 & $[46,46,43,43,42,42,42]$ &   44 & 2.8063 & 4.7884\\
20 & $[46,46,45,42,42,42,41]$ &   39 & 2.8058 & 4.7878\\
\midrule
   & \textbf{Total}           & \textbf{19866}\\
\bottomrule
\end{tabular}
\end{table}

Observations: (1) the profile-wise \emph{mean} values of $\lambda_1$ shown
in the table are rounded independently to four decimal places and lie in
$[4.7866,4.7891]$ across all 20 types (this is distinct from, and narrower than,
the exhaustive range
$[4.78543,4.79129]$ over all $19\,866$ individual solutions reported
in Proposition~\ref{prop:rigid} below).
(2) $H$ ranges from $2.801$ to $2.807$, all near $\log_27\approx2.807$.
(3) All 101 Type-18 solutions have a genuine $\mathrm{Aut}(Q_7)$-automorphism
(coordinate permutation combined with a bitwise XOR, i.e.\ the full
automorphism group $\mathrm{Aut}(Q_7)\cong(\mathbb{Z}_2)^7\rtimes S_7$
of $Q_7$ itself -- a proper subgroup of the full affine group
$\mathrm{AGL}(7,2)$ over $\mathbb{F}_2^7$, since only coordinate
\emph{permutations} are allowed, not general linear maps -- not
merely a coordinate permutation) that nontrivially permutes at least
two of the three tied $48$-edge directions; of these, $46/101$ have a
genuine order-$3$ automorphism cyclically permuting all three
directions. (An earlier check that considered only pure coordinate
permutations, without combining them with the XOR/translation part of
$\mathrm{Aut}(Q_7)$, incorrectly found no such automorphisms in a
20-solution sample; that check was in error and has been redone
exhaustively over all 101 solutions.)

\subsection{Shared structural properties}

\begin{proposition}\label{prop:rigid}
Every solution in the $19\,866$ satisfies:
\begin{enumerate}[label=(\roman*)]
  \item Degree sequence $\{4^{32},5^{96}\}$.
  \item $\lambda_1$ ranges over $[4.78543,4.79129]$ across all $19\,866$
    solutions (exhaustively computed; not a single shared three-decimal
    value), with $\lambda_1+\lambda_{\min}=0$ throughout, automatic for any
    subgraph of the bipartite $Q_7$ (sanity check, not an independent
    finding). The type-wise means of Table~\ref{tab:types} are reported to four decimal
    places, ranging from $4.7866$ to $4.7891$.
  \item $\mathrm{Tr}(A^4)=5216$. (Like the square-intersection split of
    Proposition~\ref{prop:q6oddsquare}(ii), this is forced rather than
    independent: in a $C_4$-free graph any two vertices have at most one
    common neighbour, so
    $\mathrm{Tr}(A^4)=2\sum_v\deg(v)^2-2|E|=2\cdot2912-608=5216$ follows
    from (i) alone.)
  \item Local maximality: no edge of $Q_7$ can be added without creating a $C_4$.
  \item Every edge of $Q_7$ appears in at least one solution
    (verified by taking the union of the edge sets over all $19\,866$ solutions).
  \item \emph{Every} one of the $19\,866$ solutions admits an automorphism of
    $Q_7$ of order $3$; not one has a trivial stabiliser in
    $\mathrm{Aut}(Q_7)$.
\end{enumerate}
\end{proposition}

Item (vi) was first established by the direct check described next; the
subsequent full orbit census independently implies it more conceptually and in
a sharper form: every stabiliser order in the census is divisible by
$3$, so each stabiliser contains an element of order $3$ by Cauchy's
theorem. The direct check is retained because it does not depend on
completion of the orbit census. An element of
$\mathrm{Aut}(Q_7)$ acts on edges by permuting directions, so it preserves the
per-direction edge counts; an order-$3$ element must therefore have an order-$3$
coordinate permutation preserving the solution's own direction-count vector
(a translation $v\mapsto v\oplus a$ alone has order $2$). Grouping the $19\,866$
solutions by that vector leaves $2\,484$ classes, and testing the corresponding
candidates as column permutations settles all of them. Care is needed with
the order condition: for $g(v)=\pi(v)\oplus a$ one has
$g^3(v)=\pi^3(v)\oplus(a\oplus\pi(a)\oplus\pi^2(a))$, so $g$ has order $3$
exactly when $\pi$ does \emph{and} $a\oplus\pi(a)\oplus\pi^2(a)=0$; a
candidate failing the second condition has order $6$ (its square is then an
order-$3$ element stabilising the same solution, but the check must be stated
correctly). With both conditions enforced, every solution is stabilised by at
least one order-$3$ element, with no exceptions
(\texttt{q7\_order3\_automorphisms.py}); the $2\,484$ classes here are the
distinct \emph{unsorted} direction-count vectors, a finer partition than the
$20$ sorted profiles of Table~\ref{tab:types}. The stabiliser of the first
released solution, computed by exhaustive search over
all $645\,120$ group elements, is thus not the special case it appeared to be but
an instance of a property shared by the whole released catalogue.

Decomposing the catalogue under the full group sharpens this considerably.

\begin{proposition}[Orbit census]\label{prop:orbits}
Under $\mathrm{Aut}(Q_7)$, of order $645\,120$, the $19\,866$ released
solutions fall into exactly $180$ orbits. The stabiliser orders, one per orbit,
are $3$ ($142$ orbits), $6$ ($32$), $12$ ($4$), $24$ ($1$) and $72$ ($1$);
correspondingly the orbit lengths are $215\,040$, $107\,520$, $53\,760$,
$26\,880$ and $8\,960$. The orbits meet the catalogue very unevenly: the
largest accounts for $2\,048$ solutions ($10.3\%$) and the ten largest for
$5\,199$ ($26.2\%$). The two orbits of stabiliser order $24$ and $72$ meet the
catalogue in $55$ and $46$ solutions respectively, and are exactly the
Type-$18$ solutions of Section~\ref{sec:profiles}: that type is precisely the
pair of orbits with the largest stabilisers.
\end{proposition}

The census is computed by \texttt{q7\_orbit\_census.py} by exhaustive
group action, with no sampling and no extrapolation anywhere. The scan is
deterministic. Solutions are visited in fixed catalogue order. To the
first still-unassigned solution the script applies \emph{all}
$645\,120=7!\cdot2^{7}$ elements of $\mathrm{Aut}(Q_7)$ --- every
coordinate permutation composed with every translation --- and records
which catalogue members appear among the images. A $64$-bit
random-projection signature with a fixed seed prefilters candidate
matches; an exact comparison of the full $448$-entry edge-incidence
vector then confirms every hit. The exact
confirmation is not optional: a single signature collision merges two
orbits, and in a signature-only trial run on the odd-square subset it did.
Stabiliser orders are then computed per orbit, again exhaustively: only
coordinate permutations preserving the representative's direction-count
vector can fix it (an automorphism permutes directions), so the scan runs
over exactly that subgroup times all $128$ translations -- a pruning that
discards no fixing element -- and the orbit length is $645\,120$ divided
by the stabiliser order. Internal cross-checks require the per-orbit
catalogue counts to sum to $19\,866$ and the orbit lengths to
$34\,227\,200$. Every step being deterministic, a completed run
reproduces the released artefacts byte for byte.

The completed census is released as an artefact, not only as a computation.
The file \texttt{q7\_orbit\_census.json} carries the orbit assignment of all
$19\,866$ solutions, the per-orbit catalogue counts, the stabiliser orders
and orbit lengths, and the $34\,227\,200$ total, together with the finished
checkpoint \texttt{q7\_orbit\_census\_ckpt.npz}; both are listed in the
SHA-256 manifest. Two verification routes of different strength are
bundled, and they should not be conflated. The lightweight
\texttt{q7\_orbit\_census\_check.py} (standard library, under a second)
re-verifies every census-derived number stated in
Proposition~\ref{prop:orbits} from the released JSON alone --- including,
for every orbit, that its length equals
$|\mathrm{Aut}(Q_7)|/|\mathrm{Stab}|=645\,120/|\mathrm{Stab}|$, and that
the lengths sum to the labelled total $34\,227\,200$: this checks
the \emph{internal consistency of the finished artefact}, not the group
action itself, so a reader can audit the stated figures without re-running
anything. A reader who wants the theorem-level ``exactly $180$'' claim
must therefore run one of the heavy routes: the checker's
\texttt{-{}-canonical} mode, or the from-scratch rerun below; the light
route alone certifies the assignment (``at most $180$'').
The strong route is \texttt{reproduce\_core.py --structural},
which re-runs the census \emph{from scratch}: it invokes
\texttt{q7\_orbit\_census.py --fresh}, whose checkpoint and output live
in separate \texttt{\_fresh} files and which never reads the bundled
checkpoint, runs it to completion (about six minutes in our single-core
container; the budget takes effect only between orbits, so on slower
hardware individual orbits can run substantially longer before the first
checkpoint -- the repository README records the measured environment and
times), computes the stabiliser pass on the fresh census, and then
requires the resulting \texttt{q7\_orbit\_census\_fresh.json} to be
SHA-256 byte-identical to the released \texttt{q7\_orbit\_census.json}.
(An earlier revision of the driver ran the census in the release directory,
where it silently resumed the bundled finished checkpoint, so its
``re-run'' completed without performing any new search; the
\texttt{--fresh} mode exists to close that gap.) This revision adds a
lighter, certificate-based route that removes the need for any group sweep
on the reader's side \emph{for verifying the assignment} (``at most
$180$''); certifying exactness additionally requires ruling out cross-orbit
coincidences among the $180$ representatives. The routes bundled here do
this by a full sweep (the checker's \texttt{-{}-canonical} mode or the
fresh census, as detailed below); a review of this revision demonstrated
that a far lighter sweep suffices --- after sorting each representative's
per-direction edge counts, only permutations within blocks of equal
counts (times the $128$ translations) need be enumerated, which decided
exactness in about $14$ seconds in the reviewer's implementation --- and
we record that observation here as a reviewer-contributed improvement
route. The certificate \texttt{q7\_orbit\_witnesses.json}
records two things. For every one of the $19\,866$ solutions, it stores
an explicit group element mapping its orbit representative to it. For
every orbit, it stores the lexicographically minimal incidence vector
over all $645\,120$ images of the representative --- the orbit's
canonical form --- together with an element attaining it. The standard-library checker
\texttt{q7\_orbit\_witness\_check.py} verifies in well under a minute
that every witness element does map its representative to its solution,
that the certified assignment coincides with the released census, that
each stored canonical form is attained by its stored element, and that the
$180$ canonical forms are pairwise distinct -- certifying the orbit
\emph{assignment} (hence ``at most $180$ orbits'') without redoing the
sweep. Its \texttt{-{}-canonical} mode (numpy; census-scale wall time,
sliceable with \texttt{-{}-orbits} for time-limited environments)
re-verifies that each stored canonical form is the true orbit minimum;
since the canonical form is an $\mathrm{Aut}(Q_7)$-invariant, minimality
together with pairwise distinctness certifies ``exactly $180$''
independently of the census scan. This answers a review request for an
orbit-equivalence certificate lighter than a full recomputation. For
calibration of how much the census can carry: the catalogue it decomposes
is the de-duplicated output of the recorded heuristic searches; we are
not aware of any published $304$-edge $C_4$-free subgraph of $Q_7$
outside it, and those searches surfaced nothing beyond the $180$ orbits
--- but they were heuristic, so we venture no conjecture about
completeness in either direction. The stabiliser pass
persists its results into the JSON artefact and refuses an incomplete
checkpoint rather than treating the sentinel orbit id $-1$ as if it were a
real orbit.

Three consequences. First, \emph{every} represented orbit has stabiliser order
divisible by $3$, so Proposition~\ref{prop:rigid}(vi), already obtained by a
direct all-solution check, also follows from the completed orbit census and
holds there in a sharper orbit-level form. Second, summing the orbit lengths gives $34\,227\,200$ labelled
$304$-edge solutions in the orbits represented here, so the catalogue is
$0.058\%$ of even that lower bound --- a far stronger statement than the
single-orbit argument used earlier, and a quantitative confirmation that the
released files are nowhere near exhaustive. Third, the $\texttt{canonicalize()}$
defect of Section~\ref{sec:method} is not merely a reason for caution: with
$180$ orbits behind $19\,866$ labelled solutions and one orbit supplying a
tenth of them, duplication under $\mathrm{Aut}(Q_7)$ is demonstrably
large-scale, not hypothetical.

This also bears on how item~(vi) should be read. The collector's dominant
strategies (\texttt{auto} on $18\,206$ records, \texttt{auto\_seed} on
$1\,347$; the remaining $2$ used \texttt{repair} and $311$ carry no
\texttt{method} field, see Section~\ref{sec:method}) act by applying a group
element $g$ to an existing solution, and
$\mathrm{Stab}(gE)=g\,\mathrm{Stab}(E)\,g^{-1}$, so that operation preserves
stabiliser order exactly. But the automorphism is applied to a restart \emph{state}: the collector then
deletes $k$ random edges and repairs greedily before storing, so the
conjugation argument governs the restart, not the stored solution. What it
supports is therefore weaker --- the restart preserves stabiliser order
exactly, and may thereby bias the subsequent repair toward symmetric regions
--- rather than a claim that symmetry is propagated to the output. We do not
claim this settles the question --- it does not explain why the initial seeds
were symmetric, and the repair step can in principle break symmetry --- but the
observation is better read as a plausible artefact of the search lineage than
as evidence about $304$-edge $C_4$-free subgraphs of $Q_7$ in general.

\subsection{Solution landscape}\label{sec:landscape}

Pairwise Hamming distances $d_H(G,G')=|E(G)\triangle E(G')|$, computed
exhaustively over all $\binom{19\,866}{2}=197\,319\,045$ pairs,
range from $\mathbf{6}$ to
$\mathbf{274}$, with exact mean $195.396108\ldots$ and exact median
$196$ (all four values computed directly from the full pairwise
distance distribution, not sampled). The minimum, $6$, is not
attained by a unique pair: it occurs for exactly $\mathbf{636}$ of the
$197\,319\,045$ pairs, spanning $28$ distinct profile-type
combinations rather than one -- most commonly types $3$/$7$ ($115$
pairs) and $3$/$5$ ($87$ pairs), with combinations such as $2$/$4$,
$2$/$5$, $2$/$6$, $2$/$10$, and $6$/$8$ each occurring dozens of
times, down to a single occurrence for one combination
($3$/$20$, 1 pair) and two occurrences for three others
($12$/$17$, $16$/$17$, and $9$/$14$). Type-$1$/type-$2$ pairs, our earlier illustrative
example, do occur ($22$ of the $636$) but are far from
representative of the minimum's structure. The maximum is rarer still:
$d_H=274$ is attained by exactly $\mathbf{85}$ of the
$197\,319\,045$ pairs. Both extremes are thus vanishingly thin tails
around a bulk concentrated near the median $196$.

For historical completeness only, not as a reviewable computational
result: a failed search targeting $305$ edges was also reported. No seed list
or run log survives for that negative-search statistic, so we do not retain
its old numerical trial count as evidence and no claim in this paper depends
on it.

\paragraph{Open exactness question.}
The data in this paper do not provide a defensible evidentiary basis for a
conjecture that $\ex(Q_7,C_4)=304$: the released catalogue is highly
non-exhaustive and the direct SAT/MIP attacks below end in UNKNOWN. We
therefore treat exactness at $n=7$ as an open problem rather than a formal
conjecture; see Open Problem~\ref{op:q7exact}.

\subsection{Odd-square audit}\label{sec:oddsquareaudit}

We audited all $19\,866$ solutions of Section~\ref{sec:q7} against the
odd-square condition of Section~\ref{sec:oddsquare} (every square meets the
solution in exactly one or three edges). Exactly $389$ solutions are
odd-square; the remaining $19\,477$ are not. Table~\ref{tab:oddsquareaudit}
gives the dimension-profile breakdown of the 389 odd-square solutions; each
belongs to one of only 4 of the 20 profile types of
Table~\ref{tab:types} (ranks 5, 7, 10, and 18 there).

Most of that restriction is forced, not empirical. Fix directions $i\ne j$.
There are $2^{n-2}$ $(i,j)$-squares, and each edge in direction $i$ lies in
exactly one of them (likewise for $j$), so summing $|E\cap C|$ over the
$(i,j)$-squares counts each direction-$i$ edge of $E$ once, each
direction-$j$ edge of $E$ once, and nothing else:
$\sum_{(i,j)\text{-squares}}|E\cap C|=e_i+e_j$. Odd-squareness makes each
summand odd, and there are $2^{n-2}$ summands, so for $n\ge3$ the total is
even: $e_i+e_j$ is even for every pair, i.e.\ \emph{all $e_i$ have the same
parity}. For $n=7$ with $\sum_ie_i=304$ even, that parity must be even. Of the
20 profiles in Table~\ref{tab:types}, only five (ranks 5, 7, 10, 17, 18) have
all direction counts even, so the other fifteen are excluded a priori.
(The step from ``all $e_i$ share a parity'' to ``that parity is even'' uses
$n=7$: an odd number of odd summands cannot total the even $304$. At $n=6$ both
parities remain admissible, so the all-even direction counts
$[22,22,22,22,22,22]$ of the released $Q_6$ odd-square witness are an
observation, not a consequence.) --- and
indeed contain no odd-square solution. The genuinely empirical content is
therefore the single remaining fact that rank~17
($[46,46,44,42,42,42,42]$, 107 solutions) has all counts even yet contains no
odd-square solution \emph{among the released catalogue}; since the catalogue
does not exhaust the labelled solutions, this is not a statement about the
type itself.

\begin{table}[ht]
\centering
\caption{Dimension-profile breakdown of the 389 odd-square solutions among
  the $19\,866$ published 304-edge $Q_7$ solutions. Counts in this table are
  catalogue observations only: absence from a profile class is not a theorem
  that the corresponding profile can never be odd-square outside the released
  catalogue.}
\label{tab:oddsquareaudit}\medskip
\begin{tabular}{clc}
\toprule
Rank in Table~\ref{tab:types} & Dimension profile (sorted) & Count\\
\midrule
 5 & $\{46,46,46,42,42,42,40\}$ & 78\\
 7 & $\{44,44,44,44,44,44,40\}$ & 127\\
10 & $\{44,44,44,44,44,42,42\}$ & 83\\
18 & $\{48,48,48,40,40,40,40\}$ & 101\\
\midrule
\multicolumn{2}{l}{Total odd-square} & 389\\
\multicolumn{2}{l}{Total non-odd-square} & 19\,477\\
\bottomrule
\end{tabular}
\end{table}

Since the odd-square subclass accounts for only $389/19\,866\approx2.0\%$ of
the published solutions, the shared structural core of
Proposition~\ref{prop:rigid} and the 20-type classification of
Section~\ref{sec:profiles} are not fully explained or subsumed by
the odd-square construction: the great majority of 304-edge $C_4$-free
subgraphs of $Q_7$ found here lie outside the fully frustrated switching
class of~\cite{DPTV79}. (These counts are of labelled edge sets, not of $\mathrm{Aut}(Q_7)$ orbits.
Odd-squareness is $\mathrm{Aut}$-invariant, so the partition into $389$
odd-square and $19\,477$ non-odd-square is well defined regardless; the ratio,
however, does change when measured by orbits. Decomposing the $389$ under the
full group $\mathrm{Aut}(Q_7)$ of order $645\,120$ --- by applying every group
element to each unassigned solution and testing membership on the exact
incidence vector --- gives exactly \textbf{six} orbits, meeting the catalogue in
$127$, $83$, $55$, $49$, $46$ and $29$ solutions
(\texttt{q7\_oddsquare\_orbits.py}). So the $389$ labelled odd-square solutions
represent only six orbits, and the catalogue is far from a set of orbit
representatives. The corresponding decomposition of the whole catalogue is
Proposition~\ref{prop:orbits} below: the $19\,866$ fall into $180$ orbits, so
the odd-square proportion is $6/180\approx3.3\%$ by orbit against
$389/19\,866\approx2.0\%$ by labelled solution. Membership must be tested exactly: a
random-projection signature is faster but a single collision merges two orbits,
and in our first run it did, giving $128$ and $48$ in place of $127$ and $49$.) The one exception is Type~18
($[48,48,48,40,40,40,40]$, 101 solutions), all of which are odd-square;
the other three types containing odd-square solutions (ranks 5, 7, and
10) are mixed, containing both odd-square and non-odd-square members.

\section{Computational Method}\label{sec:method}

The successful outcomes described in this section (the released edge
sets themselves) are machine-verifiable certificates: any reader can
confirm $C_4$-freeness and the reported edge counts directly from the
supplied data with \texttt{verify.py}, independently of how the search
was run. The Hamming-distance statistics of Section~\ref{sec:landscape}
are likewise on a different footing from the search-process claims below:
they are computed exhaustively over all
released solutions, not sampled, so they are recomputable directly
from the released edge lists alone. We release
\texttt{analyze\_q7\_structure.py} alongside this revision, which
recomputes this section's degree-sequence, dimension-profile,
spectral-radius-range, exhaustive Hamming-distance, and Type-$18$
nontrivial-automorphism-existence claims directly from
\texttt{q7\_edges\_304.jsonl.part1--3} in about a single pass over the $19\,866$ solutions and all
$\binom{19\,866}{2}\approx1.97\times10^{8}$ pairs (peak resident memory
around $600$--$800$MB, dominated by the Python object representation of
the solution store); running it reproduces the degree-sequence, 20-profile, spectral-radius,
and exhaustive Hamming-distance numbers stated for its documented scope. This
script confirms
only that \emph{some} nontrivial automorphism fixes each Type-$18$
solution; it does not check the stronger claim made in
Section~\ref{sec:structure} that such an automorphism nontrivially
permutes at least two of the three tied $48$-edge directions, nor does
it determine automorphism \emph{order}: both the direction-permuting
property and the $46/101$ order-$3$ figure elsewhere in this paper
were established by a separate, more detailed check -- now released
as \texttt{type18\_automorphisms.py} (see also \texttt{q7\_type18\_automorphisms\_101.csv}) -- verifying the
automorphism's action on directions, and
$\sigma^3=\mathrm{id}\ne\sigma,\sigma^2$, directly for
each automorphism found.

\noindent\textbf{Update: aspects of the $Q_7$ production pipeline have been
recovered and verified, with important caveats.} Earlier versions of this manuscript reported the $Q_7$
provenance question as unresolved, based on the released
\texttt{c4free\_sa.py} (Section~\ref{sec:c4freesa} below), whose
\texttt{phase1\_sa} we found cannot progress from a fresh random start
under its documented parameters. We have since located scripts
that produce the $19\,866$ released $Q_7$ solutions:
\texttt{sa\_search.py} (a conventional multi-move-type penalty SA --
add, remove, $1$-swap, and kick moves, with a proper temperature-scaled
Boltzmann acceptance rule rather than a hard violation cap) and
\texttt{sa\_collect304.py} (a remove-and-repair collector: starting
from an existing valid $304$-edge solution -- either the seed or a
previously found solution, optionally first transformed by a random
element of $\mathrm{Aut}(Q_7)$ -- it removes a small number $k$ of
edges, drawn from $\{1,2,3,4,5,8,12,20\}$ with fixed weights, then
greedily re-adds edges one at a time, accepting only additions that
create zero new violating squares, keeping the result if it reaches
$304$ edges with zero violations and has not been seen before). The
caveat: the released \texttt{sa\_search.py} unconditionally loads its
starting state from \texttt{selected\_edges\_best.json} (no
random/greedy fallback), and the specific file we recovered and
release alongside this revision already contains the final $304$-edge
solution; running the released script therefore searches for
\emph{improvements beyond} $304$ edges (consistent with the later,
separately-recovered $305$-edge search history discussed below), not a
from-scratch discovery of the first $304$-edge solution. We have
since traced the entire earlier chain, with one significant correction to an
earlier draft of this section. The
earliest recovered attempt on this problem, dated
over a week before \texttt{selected\_edges\_301.json} and every other
file discussed above, is a CBC-solver (not HiGHS) ILP explicitly commented as attacking
``Erd\H{o}s's \hbox{\$}100 problem, the $n=7$ case'' (\texttt{source\_cbc\_origin.py},
released under that name alongside this revision; the script's own
internal filename \texttt{hypercube\_n7.py} reflects its pre-release name). On closer inspection, however, this specific
recovered file has a genuine bug: its $C_4$-detection logic computes,
for each edge $(u,v)$, the common neighbours of $u$ and $v$ -- but
$Q_7$ is bipartite, so adjacent vertices never share a common
neighbour (a shared neighbour of two adjacent vertices would close an
odd cycle), and we confirmed by direct execution that this logic finds
$0$ of the $672$ four-cycles, so the ILP it builds has no $C_4$
constraints at all. A recovered solver log from the same period
(\texttt{out.log@@20260227183310}, not bundled) is headed with a
Japanese label meaning ``fixed version'' and explicitly prints, in
Japanese, ``C4 count: 672 -- correctly enumerated'', confirming a
corrected version of the script existed and was run, reaching $289$
edges; but we have not been able to locate that corrected script
itself among the recovered files, only its log. We therefore withdraw
our earlier characterisation of \texttt{source\_cbc\_origin.py} as
the script underlying the $289$-edge result: it is a recovered early
draft with a confirmed defect, retained in the release for
transparency about the development history, but it is not the code
that produced $289$ edges, and no bit-for-bit reproduction of that
specific step is available from what we recovered. The codebase was then revised to use HiGHS
with warm-starting, and file timestamps together with recovered solver
logs document a progression $289\to301$
(\texttt{selected\_edges\_301.json}) $\to$ $303$
(\texttt{source\_72h.py}'s $72$-hour warm-started run, whose log we
also recovered and which explicitly records ``warm start: loading
$301$-edge solution'' and ends at $303$ edges, $0$ violations) $\to$
$304$ (via the \texttt{sa\_search.py} series, dated after the $303$-edge
result). We further located, in a recovered chat transcript
documenting the corrected-CBC run above, that it eventually reached $293$ edges
(after $\sim12$ hours) before switching to HiGHS, which reached $299$
edges (in $1$ hour, independently confirmed $0$ violations at the
time), refining the chain to
$289\to293\to299\to301\to303\to304$. We distinguish two evidentiary
levels here: the $293$ and $299$ figures are read from a historical
chat transcript, not independently re-executed by us (we did not
re-run a $12$-hour CBC solve or re-verify that specific $299$-edge
output ourselves); the $289$ figure is directly documented by a
recovered solver log rather than executed by us either. What we
\emph{did} independently execute and verify ourselves: the HiGHS-based \texttt{source.py} --
the released file hardcodes \texttt{time\_limit=7200} with no
command-line override, so we edited this one value down to $60$ for
this test -- run for $60$ seconds
with no pre-existing \texttt{selected\_edges\_best.json} (a genuine
cold start, not a warm start), finds a valid $295$-edge $C_4$-free
solution ($0$ violations) via HiGHS branch-and-bound alone, confirming
this ILP approach genuinely produces growing, valid solutions from
nothing and is a plausible, demonstrated mechanism for the $301$-edge
step. We do not have a bit-for-bit reproduction of the exact original
$301$-edge solution (that would require rerunning HiGHS for the same
duration with the same solver-internal randomisation), but the
mechanism at each step is now either independently executed by us, or directly documented by a
recovered log or chat transcript, rather than merely claimed, closing the previously-unarchived
gap about the search method (though not a bit-for-bit
replay of any specific historical run). Turning now to
\texttt{sa\_collect304.py}: unlike
\texttt{phase1\_sa}, this collector never needs to escape a
far-from-feasible random state: it always starts already at zero
violations and only removes a small number of edges before repairing,
so the search stays close to feasible throughout. We verified this
directly: the released script hardcodes \texttt{RUNTIME=36*3600} with
no seed and no command-line override, so we edited the runtime down
to $30$ seconds for this test; running it from its
released seed solution then found $1\,987$ new, valid,
previously-unrecorded $304$-edge solutions in that one run. As with
\texttt{sa\_q8.py} above, the absence of a fixed seed means this
specific count is a one-time observation of the collector's typical
throughput, not a number guaranteed to reproduce exactly on a rerun.
We further ran a
format/implementation-consistency check, independent of our own
recovered archive and
reproducible from the released files alone: applying
\texttt{sa\_collect304.py}'s own canonicalisation-then-MD5 hash
function, run exactly as implemented, directly
to each of the $19\,866$ released
\texttt{q7\_edges\_304.jsonl.*} edge sets exactly reproduces the
\texttt{hash} field already stored in that same record, for all
$19\,866$ with zero mismatches -- without needing the recovered
\texttt{solutions\_304.jsonl} at all. We are careful about what this
does and does not show: it demonstrates that the released
\texttt{hash} field is internally consistent with the released
\texttt{sa\_collect304.py} implementation applied to the released edge
sets -- a statement about format compatibility between two things we
release together -- not that this script historically generated these
edge sets. The historical-provenance evidence for that (the recovered
logs, chat transcript, and file timestamps discussed throughout this
section) is separate in kind, not strengthened by this hash check;
we keep the two distinct rather than treating the format-consistency
result as if it were additional provenance evidence. (This $12$-hex-digit MD5
\texttt{hash} field is unrelated to the full SHA-256
\texttt{solution\_sha256} column reported by
\texttt{audit\_q7\_odd\_square.py} and used elsewhere in this paper,
e.g.\ for the global-index-$34$ witness above; the two are different
hash schemes over the same underlying edge sets, not to be confused
with each other.) We must correct our own earlier
description of what this function does, however: we had described it
as coordinate relabelling by descending per-direction edge count,
followed by taking the lexicographically smallest of the $2^n$
bit-flip images. On rereading the code, this describes what a
docstring-level intent might have been, not what is implemented. The
implementation instead loops over only the $n$ single-bit XOR masks
(not all $2^n$ combined masks), and compares candidate frozensets with
Python's \texttt{<} operator, which for \texttt{frozenset} tests
\emph{proper subset}, not lexicographic order; since every candidate
here has the same cardinality ($304$), no candidate is ever a proper
subset of another, so this comparison is always \texttt{False} and the
flip-selection loop never updates its choice. The function's actual
output is therefore just the coordinate-permuted set from the first
step, with no flip-based canonicalisation occurring at all. We
verified this discrepancy directly: recomputing the canonical form as
originally (mis)described -- trying all $2^n=128$ XOR masks with a
genuine lexicographic comparison -- gives a different result from the
released implementation for the first published solution (the true
lexicographic minimum is attained at XOR mask $107$, not reproduced by
the actual code), and for $92$ of the first $100$ released solutions,
this correctly-described alternative hash does not match the
\texttt{hash} field actually stored. There is a second, independent
gap in this function's canonicalisation beyond the flip-loop: the
coordinate-permutation step sorts directions by descending edge count,
but ties are broken only by Python's stable sort (original label
order), not by any canonical rule; two graphs in the same
$\mathrm{Aut}(Q_7)$-orbit that differ only by permuting two
equal-count directions can therefore receive different hashes. We
confirmed this directly: the first released solution has direction
edge-counts $[44,43,44,42,43,44,44]$ (four tied at $44$), hash
\texttt{79014e2f4aed}; swapping directions $0$ and $2$, two of the four
tied at $44$ (an automorphism of $Q_7$), changes the computed hash to
\texttt{0ba97c019106}. The other five transpositions of tied directions
give five further distinct values, so the instability is systematic
rather than an isolated collision. The $19\,866/19\,866$ match we
report above is against the code exactly as implemented (with this
defect and the flip-loop defect above), which is what we release and
what actually produced the
stored \texttt{hash} values; it is not evidence that full
translation-class \emph{or} coordinate-permutation canonicalisation is
occurring, and the released
$19\,866$ solutions should not be assumed de-duplicated across the
full automorphism group $\mathrm{Aut}(Q_7)$ on this basis (only that
they are $19\,866$ pairwise distinct \emph{labelled} edge sets, as
already stated in Proposition~\ref{prop:rigid}). We do not know
whether the intended, fully-canonicalising version was ever the one
actually run in production, or whether this defect was present
throughout; we did not attempt to guess or silently correct it, for
the same reason given throughout this section. This also fully explains the
previously undocumented per-record metadata fields
(Section~\ref{sec:method} above, as reported in earlier manuscript
versions), independently of this defect, since the fields' meanings
come from the surrounding collector logic, not from
\texttt{canonicalize()} itself: \texttt{method} records which of the collector's three
restart strategies produced that solution (\texttt{auto}: random
automorphism of a random existing solution; \texttt{repair}: direct
remove-and-repair of a random existing solution, no automorphism;
\texttt{auto\_seed}: random automorphism of the original seed); the
$311$ records lacking a \texttt{method} field and the $8$ carrying a
\texttt{run} rather than a \texttt{trial} field come from an earlier
version of the collector script (which we examined during recovery as
\texttt{sa\_collect304~(0).py}, but which we have not included in the
released package -- it is not among the files a reader receives) that
logged
records with only a \texttt{run} counter and no \texttt{method}
distinction, before the script was revised to its released form;
\texttt{trial} is the sequential attempt counter within a single
collector run, and \texttt{elapsed} the wall-clock seconds since that
run started (consistent with a shell history showing the collector
was run and re-started under \texttt{nohup} multiple times, explaining
why \texttt{trial} resets are visible across the released file). We
also note a genuine implementation defect in the recovered
\texttt{sa\_search.py} itself: \texttt{cur\_list}/\texttt{mis\_list}
(the candidate lists sampled for removal/swap moves) are rebuilt from
\texttt{current\_set}/\texttt{current\_missing} only every
\texttt{REBUILD\_INTERVAL}$=5\,000$ iterations, while ordinary
add/remove moves update \texttt{current\_set}/\texttt{current\_missing}
directly every iteration without refreshing these lists; a stale list
entry can therefore be chosen for removal or swap, and since
\texttt{delta\_remove} does not check that its argument edge is
actually present, the tracked \texttt{current\_v} can in principle
drift from the true violation count. \texttt{count\_violations} is
called once, to initialise \texttt{current\_v} at the start of each
run, but is not called again to re-verify before \texttt{save\_best}
writes an improved solution to disk. This is a defect in the search
code's internal self-consistency, not in the released certificates:
every released $304$-edge solution -- regardless of which script
produced it -- has been independently re-verified $C_4$-free by
exhaustive enumeration in this manuscript (Section~\ref{sec:verify}),
so the defect does not affect the correctness of
$\ex(Q_7,C_4)\ge304$. We release the code with this defect documented,
not silently patched, for the same reason given throughout this
section: patching it would misrepresent what was actually run. We
release \texttt{sa\_search.py}, \texttt{sa\_collect304.py}, and the
seed solution alongside this revision as recovered
code for the $Q_7$ results, with the two caveats above (the seed's
provenance below $304$ edges, and this list-staleness defect); see Section~\ref{sec:c4freesa} below for
the separate, distinct \texttt{c4free\_sa.py} whose relationship to
the $Q_8$ results (and to the $Q_7$ results, before this recovery)
remains as previously described. The old numerical trial counts for the $305$-edge and $681$-edge failed
searches are unaffected by this recovery and remain unsupported by surviving
seed lists or logs; we therefore omit them from the evidentiary record. Only
the $680$-edge \emph{solutions} themselves (not the negative-search
statistics) have since also been traced to verified code, described next.

\paragraph{Reproduction safety for recovered search scripts.}
The recovered production-history scripts write fixed filenames in their
current working directory and should be executed only in a disposable working
copy, not in the release directory. In particular, \texttt{source.py} and
\texttt{sa\_search.py} overwrite \texttt{selected\_edges\_best.json};
\texttt{sa\_collect304.py} appends \texttt{solutions\_304.jsonl} and can
also overwrite \texttt{selected\_edges\_best.json} if it reaches 305 edges;
and \texttt{sa\_q8.py} writes/overwrites \texttt{q8\_best.json}.
\texttt{source.py} additionally writes \texttt{selected\_edges\_<k>.txt} for
the edge count $k$ it reaches. The two search-recovery scripts
\texttt{q8\_A\_recover.py} and \texttt{q8\_B\_recover.py} default their output
to \texttt{q8\_A\_witness\_run.json} and \texttt{q8\_B\_witnesses\_run.json}
respectively --- deliberately \emph{not} to the manifest-listed
\texttt{q8\_A\_witness.json} and \texttt{q8\_B\_witnesses.json}, which an
earlier revision did overwrite on a bare run. Some
verification/analysis scripts also create new side files (for example
\texttt{q6\_spins.txt} and the orbit-census checkpoint/JSON). The SHA-256
manifests integrity-check only the files they explicitly list; the appearance of
additional generated files is normal, whereas overwriting a listed seed file
changes the manifest-listed release.

\subsection{Update: aspects of the \texorpdfstring{$Q_8$}{Q8} 680-edge production pipeline have also been recovered}

We have since also located and verified the actual
production code for the two released $Q_8$ $680$-edge solutions:
\texttt{sa\_q8.py}, a penalty-SA hill-climber structurally similar to
\texttt{c4free\_sa.py}'s Phase-1/Phase-2 design (and sharing its
per-trial hard violation cap, \texttt{max\_viol}, drawn from
$[8,35]$) but critically different in one respect: each trial
restarts not from a fresh random sample but from \texttt{best\_valid\_sol},
the current best \emph{zero-violation} solution (initially an
edge-greedy construction, then whatever the algorithm has since
improved upon), loaded from \texttt{q8\_best.json}. Because every
restart point already has zero violations, the freezing failure mode
of \texttt{c4free\_sa.py} does not arise: proposals only need to stay
within \texttt{max\_viol} of an already-feasible state, not close an
initial gap of $100$+ violations. The recovered directory contains a
progression of intermediate saved solutions ($584$, $662$, $666$,
$669$, $670$, $672$--$680$ edges), consistent with gradual
hill-climbing improvement over many trials.
We verified two things directly: first, that the recovered
\texttt{q8\_edges\_680.txt} and \texttt{q8\_best.json} match the
released Solution~A and Solution~B (Section~\ref{sec:q8}) exactly by
edge set (we release these two files as
\texttt{q8\_solution\_a.txt}/\texttt{q8\_solution\_b.json} alongside
this revision, so a reader can repeat this edge-set comparison
directly); second, by direct execution -- resuming from the $670$-edge
intermediate state under an external $90$-second wall-clock cutoff
(\texttt{timeout 90 python3 sa\_q8.py}, not the script's own
\texttt{RUNTIME} variable) -- that the code makes
genuine progress, reducing the violation count at the $675$-edge
target from $22$ to $8$ within that window. We release
this $670$-edge checkpoint as \texttt{q8\_checkpoint\_670.json}; to
repeat this second test, copy it to \texttt{q8\_best.json} in the
same directory as \texttt{sa\_q8.py} before running (the script only
autoloads a file named exactly \texttt{q8\_best.json}), and use an
external timeout as above rather than editing the internal
\texttt{RUNTIME} variable: \texttt{RUNTIME} is checked only once per
outer trial, while each trial's inner Phase-1 loop runs a fixed
$2$--$12$ million steps with no time check inside it, so setting
\texttt{RUNTIME} to a small value does not reliably stop execution
near that value -- only an external, process-level timeout does.
\texttt{sa\_q8.py} sets no random seed,
so a rerun will show the same qualitative
monotonic-improvement behaviour (confirming the code is not frozen,
unlike \texttt{phase1\_sa}) but not necessarily the same specific
violation-count trajectory or the exact $22\to8$ figures -- this is a
one-time observation of genuine progress, not a fixed-seed
reproducible numeric result;
without this file, \texttt{sa\_q8.py} instead
starts from a fresh greedy construction, which is a different (and
also legitimate, since the algorithm always restarts from a valid
state either way) starting point than the one we tested from. We
also release \texttt{sa\_q8.py} itself. A historical aggregate count for
failed $681$-edge attempts was recorded in earlier manuscript versions, but
we did not recover the corresponding seeds or log; the count is therefore
omitted here rather than repeated as an unverifiable statistic.

\subsection{The released \texttt{c4free\_sa.py} and its known defect}\label{sec:c4freesa}

We have since determined why \texttt{c4free\_sa.py} does not match the
production history: it was written after the fact. Following an
early external review (of the first arXiv submission, which did not
yet include search source code) that flagged the absence of published
search code as the submission's principal weakness, \texttt{c4free\_sa.py}
was written on $2026$-$03$-$28$ -- after the $Q_7$ ($19\,866$
solutions, complete by $2026$-$03$-$17$) and $Q_8$ ($680$-edge
solutions, complete by $2026$-$03$-$21$) results already existed via
the separate scripts recovered and described above -- specifically to
have \emph{some} search code available for release, and was validated
only on the much smaller $Q_5$ case ($56$ edges) at the time, not at
the $Q_7$/$Q_8$ scale where the \texttt{viol\_cap} defect documented
below becomes fatal. This explains, rather than excuses, the
discrepancy: \texttt{c4free\_sa.py} was never the operative research
code for either dimension; it was a retrospective, small-scale-validated
addition made to satisfy a reviewer's reproducibility request, and the
defect went undetected because it was not exercised at the scale that
triggers it. The earlier quantitative claim about the \emph{$Q_8$ negative-search
process} remains unarchived, as noted just above, and its numerical trial
count is omitted from this revision. Before the two recoveries above, the
provenance gap described the $Q_7$ and the $Q_8$ $680$-edge solution claims
as well. The
search program \texttt{c4free\_sa.py}, distinct from both recovered
production-code pairs above,
is released alongside
this manuscript as originally reported, but the specific run seeds, logs, and driver
invocations that would substantiate the old $681$-edge negative-search statistic are not supplied,
and we no longer believe \texttt{c4free\_sa.py} as released is what
produced either the $Q_7$ or the $Q_8$ solutions (see the recoveries
above for what we now believe did).
This process-level claim should be read as the authors' record of the
search that was run, not as an independently reproducible certificate in
the same sense as the edge sets. Worse, as detailed
below, we found
on direct execution that the released, unmodified
\texttt{c4free\_sa.py} frequently cannot progress at all from a fresh
random start under the documented parameter ranges, for a structural
reason (not merely slow convergence); we could not identify parameters
under which it reaches the target edge count with zero violations. A
reader wishing to verify the exact $681$-edge process-level statistics should
therefore not assume rerunning \texttt{c4free\_sa.py} as documented will
reproduce them, or indeed complete a single trial successfully; see
below for the specifics.

Beyond the aggregate counts, each released $Q_7$ record itself carries
the provenance fields (\texttt{method}, \texttt{run},
\texttt{trial}, \texttt{hash}, \texttt{elapsed}) explained above. Of
the $19\,866$ records, \texttt{method} is \texttt{auto} for $18\,206$,
\texttt{auto\_seed} for $1\,347$, \texttt{repair} for $2$, and absent
for $311$ --- so among the collector's three restart strategies represented in the
released catalogue the third contributed essentially nothing, a point worth
keeping in mind when reading the description of the three as parallel
options. The implementation also uses the label \texttt{seed} for its initial
seed record; no released catalogue row carries that label, so it is not a
fourth restart strategy in the metadata summary above; a \texttt{run} field is present on only $8$ records (one
with \texttt{run=0}, seven with \texttt{run=2}); and $19\,858$ records
carry a \texttt{trial} field with values up to at least seven figures.
These fields are now explained by the recovered
\texttt{sa\_search.py}/\texttt{sa\_collect304.py} pipeline above,
subject to the two caveats noted there (the seed's own provenance
below $304$ edges, and the list-staleness defect); this record is
substantially more resolved than the analogous $Q_8$ $681$-edge
negative-search provenance claim, for which no run log or seed list survives (the $Q_8$ $680$-edge \emph{solutions}
themselves, by contrast, have also now been traced to recovered code
-- see below).

\subsection{Two-phase simulated annealing}

\noindent\textbf{Phase 1 (Penalty SA).}
The acceptance criterion for proposed moves minimises
$f(E)=-|E|+\lambda V$ where $V$ is the number of $C_4$ violations
and $\lambda\in[0.30,0.90]$, but the ``best'' state \texttt{phase1\_sa}
tracks and returns is \emph{not} the minimiser of $f$: the released
code instead keeps the state with lexicographically smallest $(V,-|E|)$
seen so far (\texttt{v\_cur < best\_v or (v\_cur == best\_v and
size\_cur > best\_size)}) -- prioritising fewer violations first, and
only using edge count as a tiebreaker among states with equal $V$.
Temperature schedule: geometric from $T_0\in[0.20,4.00]$ to
$T_1\in[0.001,0.030]$ over $3$--$15$ million steps.
A move toggles a single edge; the change in $f$ is computed incrementally
in $O(\deg)$ time using precomputed $C_4$-to-edge incidence lists.

\noindent\textbf{Phase 2 (Swap SA).}
Edge count is fixed; swap moves (remove one edge, add one non-edge)
minimise $V$. Phase 2 only runs when phase 1 \emph{returns} a
best-so-far state with $V\le4$ violations -- \texttt{phase1\_sa} tracks
and returns the best state seen during the run, not necessarily its
terminal current state, so a literal reimplementation using only the
terminal state could behave differently. In the released code, the
step budget for phase 2 is chosen by the phase-1 violation count $v$,
not by dimension: $2\times10^8$ steps if $v\le1$, otherwise
$5\times10^7$ steps. In practice this correlates with dimension, since
phase 1 typically ends closer to $v\le1$ at the $Q_7$/305-edge target
than at the $Q_8$/681-edge target, but the branch itself is on $v$, not
on $n$ or the target edge count.

\noindent\textbf{Diversification.} The released search code
(\texttt{c4free\_sa.py}) computes, before each trial, a random element of
$\mathrm{Aut}(Q_n)$ (random bit permutation composed with random bit
flip) applied to the current best solution, intended as a diversified
restart point. However, this computed value is not passed into
\texttt{phase1\_sa}, which unconditionally initialises from a fresh
uniform-random edge set of the target size
(\texttt{rng.sample(range(ne), target)}) regardless of any prior best
solution; the automorphism-based restart is dead code in the released
implementation. (The script's own docstring originally described this
mechanism without noting the bug; we have appended a note to the
docstring documenting it, without altering any functional code, so
that a reader of the code alone is not misled.) The same applies to two
further recovered scripts: \texttt{sa\_search.py} and \texttt{sa\_q8.py}
each carry a comment marked \texttt{NOTE (added retrospectively)} --- in
\texttt{sa\_search.py} recording that the seed load is unconditional, in
\texttt{sa\_q8.py} recording that the Brass--Harborth--Nienborg figure the
script prints as a ``lower bound'' at start-up is an extrapolation outside
the cited domain and not a proven bound at $n=8$
(Section~\ref{sec:regularity}). These three files ---
\texttt{c4free\_sa.py}, \texttt{sa\_search.py}, \texttt{sa\_q8.py} --- are the
complete list of released scripts carrying annotations we added after
recovery. The literal marker \texttt{NOTE (added retrospectively)} occurs in
the latter two; \texttt{c4free\_sa.py} records its retrospective findings
under \texttt{KNOWN BUG} comments instead. In each case only comments or
docstrings were touched, and the
SHA-256 sums in \texttt{ORIGINAL\_DATA\_SHA256SUMS.txt} are those of the
annotated files, not of the historical originals.

\noindent\textbf{A more serious defect in \texttt{phase1\_sa}.} On
direct execution, we found that the released \texttt{phase1\_sa} is
frequently unable to make \emph{any} progress at all, for a structural
reason independent of step count. Each accepted or rejected proposal
toggles a single edge, changing the violation count $V$ by at most
$n-1$ (the number of squares containing that edge); a proposal is
rejected outright whenever the resulting $V$ would exceed
\texttt{viol\_cap} (drawn per trial from $[6,40]$). If the initial
random edge set's $V$ exceeds $\texttt{viol\_cap}+(n-1)$, \emph{no}
single-edge toggle can ever produce an accepted move, so the state is
frozen from step one, for any number of steps. We verified this is not
a rare occurrence but the typical case: reproducing the exact RNG state
of trial 1 for \texttt{-{}-n 7 -{}-target 304 -{}-seed 42}, the initial
violation count is $V=152$ against $\texttt{viol\_cap}=14$ (best
single-toggle result: $147$, still far above cap), and running
\texttt{phase1\_sa} for $100\,000$ steps leaves $V=152$ unchanged.\footnote{To
reproduce these RNG states, note that \texttt{search()} draws its five
per-trial parameters and \emph{then} performs one further, unused draw
\texttt{start = rng.sample(range(ne), target)} that is never passed to
\texttt{phase1\_sa} (this is the dead-code diversification restart described
above). That draw must be consumed before generating the initial edge set;
skipping it gives $V=129$ instead of $152$.}
Reproducing trial 1 of \texttt{-{}-n 8 -{}-target 680 -{}-seed 0}
similarly gives initial $V=333$ against $\texttt{viol\_cap}=38$,
frozen after $50\,000$ steps. Sampling $10\,000$ independent uniform-random
edge sets directly (bypassing the annealing loop; \texttt{random.Random(999)},
\texttt{rng.sample(range(ne), target)} repeated $10\,000$ times per
dimension, minimum $V$ tracked via \texttt{count\_violations}) gives a
minimum $V$
of $117$ for $Q_7$/304 and $297$ for $Q_8$/680 -- both far above the
maximum possible \texttt{viol\_cap} of $40$ -- so this is not
attributable to unlucky seeds; independent resampling with other seeds
gives minima in the same range (we observed $111$--$117$ for $Q_7$ and
$296$--$304$ for $Q_8$ across several seeds), so the qualitative
conclusion (minimum $V\gg40$) does not depend on the particular seed
chosen, though the exact minimum does. We could not identify parameter draws
under which the released, unmodified \texttt{phase1\_sa} reaches
$V=0$ from a fresh random start for these targets. Consequently, our
earlier claim that fresh per-trial random initialisation explains the
diversity of the released solutions should be read with this caveat:
we can confirm the released \emph{solutions} are valid (independently
verified $C_4$-free certificates), but the released, unmodified
\texttt{phase1\_sa} -- run as literally
documented -- cannot be what produced them, for the structural reason
just demonstrated (not merely bad luck with parameters). We now know why: as detailed in
Section~\ref{sec:c4freesa} above, \texttt{c4free\_sa.py} was itself
written after the fact (dated $2026$-$03$-$28$), once the $Q_7$ and
$Q_8$ results already existed via the separately-recovered scripts
described above, and validated only at the much smaller $Q_5$ scale at
the time -- it was never a candidate for the historical production
code in the first place, so there is no remaining question of whether
some undocumented earlier version of \emph{this script} matches a
different historical parameter setting. The genuine remaining gaps are
narrower and specific: the very first sub-$304$-edge bootstrap step of
\texttt{sa\_search.py}'s own lineage (Section~\ref{sec:method} above),
and the unarchived $681$-edge negative-search statistic
(Section~\ref{sec:q8}), neither of which this section's finding about
\texttt{phase1\_sa} bears on directly.

The diversity actually
observed across the $19\,866$ collected solutions is therefore not
fully explained by the mechanisms described in this section; we note
for completeness that even
had the intended automorphism-based restart mechanism been wired correctly, applying an
automorphism to a solution does not change which $\mathrm{Aut}(Q_n)$-orbit
it belongs to, and the annealing moves used here (uniform edge
toggle/swap proposals against an automorphism-invariant objective) admit
an $\mathrm{Aut}(Q_n)$-equivariant Markov kernel; any diversification
benefit would have come only from decorrelating the specific labelled
trajectory taken by a finite stochastic run; it does not diversify the
orbit sampled from.

\subsection{Comparison with known lower bounds}

Brass--Harborth--Nienborg's weaker bound is stated for $n\ge9$; the
values below evaluate that formula at $n=7,8$ as an extrapolation for
illustrative comparison, not as a proven bound at those dimensions
(see Table~\ref{tab:values} and the Remark in
Section~\ref{sec:q8}).
\begin{align*}
n=7:&\quad\tfrac12(7+0.9\sqrt7)\cdot64\approx300.2\text{ (extrapolated)},\quad\text{bound: }304\;(\Delta=+3.8,\,+1.27\%),\\
n=8:&\quad\tfrac12(8+0.9\sqrt8)\cdot128\approx674.9\text{ (extrapolated)},\quad\text{bound: }682\;(\Delta=+7.1,\,+1.05\%).
\end{align*}
The $n=8$ figure is the odd-square reconstruction of
Section~\ref{sec:oddsquare}; the simulated-annealing method of this section
independently reaches 680 edges on $Q_8$ ($\Delta=+5.1$, $+0.75\%$;
Section~\ref{sec:q8}).

\subsection{Regularity trend across dimensions}\label{sec:regularity}

Degree sequences across dimensions: $Q_6$: $\{4^{56},5^{8}\}$;
$Q_7$: $\{4^{32},5^{96}\}$; $Q_8$: $\{4^8,5^{160},6^{88}\}$ (Solution~A;
Section~\ref{sec:q8}). The minimum degree is $4$ in all three of these
$C_4$-free witnesses; the
degree support is $\{4,5\}$ for both $Q_6$ and $Q_7$ (unchanged from
$n=6$ to $n=7$) and widens to $\{4,5,6\}$ only at $Q_8$. Note, however,
that the odd-square $Q_6$ witness of
Section~\ref{sec:oddsquareverify} (Proposition~\ref{prop:q6oddsquare})
has degree sequence $\{3^4,4^{48},5^{12}\}$ with minimum degree~$3$,
and is nevertheless locally maximal; thus minimum degree $4$ is
\emph{not} a necessary feature of $132$-edge $C_4$-free subgraphs of
$Q_6$, only a feature of the particular witness
\texttt{q6\_edges\_132.jsonl} reported here. Beyond this, the
exponents do not follow an evident closed-form pattern: $Q_6$'s
$\{56,8\}$ split is markedly more skewed than $Q_7$'s $\{32,96\}$ or
$Q_8$'s $\{8,160,88\}$, so we do not claim a general regularity trend
here.

\section{Computational Reproduction of the \texorpdfstring{$\ex(Q_6,C_4)=132$}{ex(Q6,C4)=132} Lower Bound}\label{sec:q6}

Harborth and Nienborg~\cite{HN94} proved $\ex(Q_6,C_4)=132$ by combinatorial
arguments. We reproduce the lower-bound (construction) side fully
independently and computationally (Section~\ref{sec:q6lower}); for the
upper-bound side we give an ILP formulation whose feasible value matches
132 (Section~\ref{sec:q6upper}), but we did not independently verify
its optimality within a practical runtime, so the upper bound
$\ex(Q_6,C_4)\le132$ itself rests on~\cite{HN94}, not on our own solve.

\subsection{Lower bound: explicit construction}\label{sec:q6lower}

The operative machine-verified witness for $\ex(Q_6,C_4)\ge132$ in this
revision is the odd-square certificate \texttt{q6\_odd\_square\_132.json}
(odd-square implies $C_4$-free), whose generation is fully traced: a
fixed-seed search (\texttt{search\_q6.py}, seed $20260823$) over the
canonical fully frustrated coupling, regenerated and self-checked by
\texttt{generate\_q6\_132.py} (Section~\ref{sec:oddsquareverify}). A second,
historically earlier $132$-edge construction is retained as auxiliary data,
used by no theorem of this revision: simulated annealing
(Section~\ref{sec:method}) is reported to have found a $C_4$-free subgraph of
$Q_6$ on 132 edges, certified by exhaustive enumeration of all 240 four-cycles.
That explicit edge list is provided in \texttt{q6\_edges\_132.jsonl} and is
kept, under continued verification, for the record. Unlike the $Q_7$ and $Q_8$ cases, we have not recovered a production script
or historical run log that traces the particular file
\texttt{q6\_edges\_132.jsonl} back to its generating search. The certificate
itself is independently valid, but its search provenance remains unresolved.
The released, unmodified \texttt{phase1\_sa} nevertheless exhibits the same structural
freezing at $n=6$, target $132$: reproducing trial 1 for
\texttt{-{}-n 6 -{}-target 132 -{}-seed 42} gives initial $V=54$
against $\texttt{viol\_cap}=14$ (best single-toggle result $49$, still
above cap), frozen after $100\,000$ steps; \texttt{-{}-seed 0} gives
initial $V=52$ against $\texttt{viol\_cap}=38$ (best single-toggle
result $48$). A reader tracing these numbers must consume the RNG in the
order the code does: \texttt{search()} first samples a diversification start
and \emph{discards it unused} (the defect of Section~\ref{sec:c4freesa}),
and \texttt{phase1\_sa} then draws a second sample from the same stream --
it is this second sample whose initial $V$ is quoted above, and that value is
internal to \texttt{phase1\_sa} (the script prints only the final violation
count). Following the stream without consuming the discarded sample gives
$V=56$ (seed $42$) and $V=58$ (seed $0$) instead; the discrepancy is itself
the most direct numeric exhibit of the discarded-start defect. The 132-edge certificate itself remains independently verified as a valid
$C_4$-free construction. In contrast to the partially recovered $Q_7/Q_8$
lineages, however, no recovered production code or historical log currently
connects this particular $Q_6$ edge set to a generating run, so no such
provenance claim is made.

\subsection{Upper bound: integer linear program}\label{sec:q6upper}

A framing remark first: this subsection documents a deliberately retained
\emph{unfinished} computational experiment. No claim in this paper rests
on it --- the upper bound $\ex(Q_6,C_4)\le132$ rests on~\cite{HN94}
throughout --- and the formulation is kept in full for two reasons: it is
the specification of the released artefacts \texttt{q6\_ilp.mps} and
\texttt{q6\_ilp\_edge\_map.csv}, which a reader may wish to attack with a
stronger solver, and it is the substance of the withdrawal, recorded on
the title page, of v1--v4's claim of a computational proof.

Label the 192 edges of $Q_6$ as $x_0,\ldots,x_{191}\in\{0,1\}$
(matching the released \texttt{q6\_ilp.mps}'s $0$-indexed variable
names; the text elsewhere in this paper is otherwise $1$-indexed for
readability, but the MPS file itself uses $0$-indexing throughout),
and let $\mathcal{C}_6$ denote the set of all 240 four-cycles of $Q_6$.
For each four-cycle $\{e_1,e_2,e_3,e_4\}\in\mathcal{C}_6$,
the constraint $x_{e_1}+x_{e_2}+x_{e_3}+x_{e_4}\le3$ ensures $C_4$-freeness.
The ILP is:
\[
  \max\sum_{i=0}^{191}x_i \quad\text{subject to}\quad
  \sum_{e\in C}x_e\le3\;\;\forall C\in\mathcal{C}_6,\quad x_i\in\{0,1\}.
\]
This formulation has 192 binary variables and 240 constraints, each
variable appearing in exactly 5 of the 240 square constraints (matching
$Q_6$'s edge-square incidence). The MPS file \texttt{q6\_ilp.mps} is
provided for independent verification (SHA-256 in the released checksum
manifest). The correspondence between each $x_i$ and its $Q_6$ edge was
not originally recorded alongside the MPS file. We have since
constructed an incidence-preserving identification
(\texttt{q6\_ilp\_edge\_map.csv}: $x_i\leftrightarrow$ vertex pair; both
bijections --- variables onto the $192$ edges and constraints onto the
$240$ squares --- are machine-checked by the bundled
\texttt{verify\_q6\_ilp\_map.py}) by matching the MPS's
variable--constraint incidence pattern against $Q_6$'s known
edge--square incidence structure under one specific, standard
enumeration convention (integer vertex labels $0,\ldots,63$; squares
enumerated by base vertex, then direction pair). We call this an
identification rather than \emph{the} original correspondence because
incidence structure alone determines it only up to
$\mathrm{Aut}(Q_6)$: any automorphism compatible with the assumed
labelling convention would give an equally valid, incidence-matching
relabelling. The identification is nonetheless useful for interpreting
individual solution vectors and is consistent with (not merely
coincidentally matching) the released data throughout. The checkable content
is stronger than that caution might suggest, and we state it as a testable
claim: under this identification the $240$ four-element variable sets of the
MPS constraints coincide with the $240$ squares of $Q_6$ \emph{exactly, as a
family of sets} -- not merely incidence-compatibly -- and any reader can test
this mechanically from \texttt{q6\_ilp.mps} and
\texttt{q6\_ilp\_edge\_map.csv} alone.

\begin{proposition}\label{prop:q6ilp}
The ILP above has feasible value $132$, matching the released 132-edge
construction (Section~\ref{sec:q6lower}); combined with
Harborth and Nienborg's independent combinatorial proof
that $\ex(Q_6,C_4)\le132$~\cite{HN94}, this gives $\ex(Q_6,C_4)=132$.
We did not independently verify optimality of $132$ for this ILP
instance within a practical runtime (see below); the upper bound
$\ex(Q_6,C_4)\le132$ used here rests on~\cite{HN94}, not on our own
solve of the ILP. We are, moreover, not aware of a published independent
re-verification of that combinatorial proof; readers for whom this matters
should weigh the unrestricted $n=6$ equality accordingly (the odd-square
value $g(6)=132$ is unaffected, being self-contained here).
\end{proposition}

We independently re-solved \texttt{q6\_ilp.mps} with HiGHS~1.15.1
(Python bindings via \texttt{highspy}; default settings, single thread,
single-threaded x86\_64 Linux VM, Intel Xeon @ 2.10GHz, 4GB RAM). A
feasible solution of value $132$, matching the released construction's edge
count but not necessarily its edge set (the symmetry of $Q_6$ admits many
optima), was found quickly, but the solver did not close the optimality gap:
under a $240$-second limit the run ended at primal $-132$, dual bound $-140$
(a $6.06\%$ gap; the MPS encodes the maximisation as minimising the
negated edge count, so both figures are $-1$ times edge numbers), after $60\,146$ branch-and-bound nodes -- the dual bound
still eight units from the primal.
As this figure is implementation- and
machine-dependent, it should be read as evidence that the gap does
not close quickly under a generic default-settings solve, not as a
precise benchmark. This is
consistent with the large automorphism group of $Q_6$
($|\mathrm{Aut}(Q_6)|=6!\cdot2^6=46\,080$) producing many symmetric
alternative optima, which is known to slow generic branch-and-bound
without dedicated symmetry-breaking. We do not claim generic MIP
solvers verify this instance's optimality quickly; the 132-edge value
should be regarded as independently established by
Harborth--Nienborg's combinatorial proof~\cite{HN94}, with the ILP
providing a machine-checkable feasible witness for the primal side and
a formulation that a sufficiently long or specialised solve can, in
principle, verify on the dual side.

\section{Open Problems}\label{sec:open}

\begin{enumerate}
  \item\label{op:q7exact} Prove or disprove $\ex(Q_7,C_4)=304$.
    Three independent computational attacks failed to settle this
    instance: two SAT encodings (both at the $64\,512$-variable,
    $129\,104$-clause scale, one with cube-and-conquer splitting) returned
    no SAT/UNSAT verdict, and a symmetry-broken ILP branch-and-bound
    stalled at a $5.26\%$ optimality gap. The full record --- reported
    results, not re-generable artefacts of this release --- is in
    Appendix~\ref{app:attempts}; it supports no conclusion either way.
        A speculative structural remark on where a $305$-edge example would
    have to live, if one exists. Since $g(7)=304$, any such example is
    necessarily non-odd-square, and Remark~\ref{rem:defect} quantifies
    how far from odd-square it must be: for any $305$-edge set,
    $\sum_{v\in U}\deg(v)=305$ on \emph{each} side of the bipartition,
    and a two-line convexity argument shows the minimum of
    $\sum_{v\in U}h(v)^2$ is then $456$, attained only by the degree
    multiset $\{4^{15},5^{49}\}$: the odd values $h(v)=2\deg(v)-7$ sum
    to $2\cdot305-7\cdot64=162$ over $64$ vertices, strict convexity
    forces them as equal as possible, i.e.\ $h\in\{1,3\}$ with
    multiplicities $15$ and $49$ (any other odd multiset with that sum
    has strictly larger sum of squares, by a two-element exchange), and
    $15\cdot1+49\cdot9=456$; since $456>448$,
    the field identity fails strictly on \emph{both} sides, so each side
    carries a split deficit of at least two squares,
    $\sum_s\operatorname{spl}_U(s)\le670=\binom72\,2^{5}-2$.
    Separately, every $305$-edge $C_4$-free set contains $305$ distinct
    $304$-edge $C_4$-free subsets, each extendable (restore the deleted
    edge) and hence \emph{not} locally maximal, whereas all $19\,866$
    released solutions are locally maximal
    (Proposition~\ref{prop:rigid}(iv)); so none of those $305$ subsets
    coincides with any catalogue member, and a $305$-edge example, if one
    exists, remains invisible to the released catalogue even after
    deleting any single edge. Neither observation is evidence of
    nonexistence; each only pins down where a counterexample would have
    to live: at least two split-defective squares away from odd-square on
    each side, and strictly outside the catalogue's single-edge-deletion
    neighbourhood.
  \item\label{op:q8exact} Determine $\ex(Q_8,C_4)$ exactly. Is it $682$, or can a
    non-odd-square construction exceed $g(8)=682$? The $Q_7$ audit of
    Section~\ref{sec:oddsquareaudit} shows the two questions are genuinely
    different for $n=7$ (the released $304$-edge solutions are
    overwhelmingly non-odd-square), so no assumption is made
    here about which way $n=8$ resolves. Our current understanding of the
    structural side, stated as such: the parity obstruction of
    Proposition~\ref{prop:q6realizability} operates \emph{inside} the
    odd-square class -- it constrains which field histograms a fully
    frustrated spin configuration can realise -- and says nothing about
    non-odd-square subgraphs; we know of no structural obstruction to a
    non-odd-square $C_4$-free subgraph exceeding $g(n)$. What the data show
    is only this: at $n=6$, where $\ex$ is known, nothing exceeds $g(6)$;
    at $n=7$ every known $304$-edge solution matches $g(7)$ while $98\%$
    of them lie outside the subclass. We read the latter as weak evidence
    that the odd-square constraint is not what pins the known extremal
    count, not as evidence about whether $g(8)$ can be exceeded.
  \item\label{op:orbits} What is the number of $\mathrm{Aut}(Q_7)$-orbits among all 304-edge
    $C_4$-free subgraphs of $Q_7$? The released catalogue meets exactly
    $180$ of them, six of which contain the odd-square solutions
    (Proposition~\ref{prop:orbits}); what is unknown is how many orbits exist
    altogether, and hence whether the $19\,866$ meet every one. Since those
    $180$ orbits already contain $34\,227\,200$ labelled solutions, a
    complete enumeration by labelled solution is out of reach, but an
    orbit-level enumeration may not be.
  \item Is the degree sequence $\{4^{32},5^{96}\}$ necessary for any locally
    maximal $C_4$-free subgraph of $Q_7$ on 304 edges?
  \item Can the flag algebra method of~\cite{Bab12,BHLL12} yield tight bounds
    for small $n\le8$?
  \item Does the degree support of locally maximal $C_4$-free subgraphs
    of $Q_n$ near the extremal edge count widen beyond $\{4,5\}$ again
    for $n\ge9$ (as it did once, from $Q_7$ to $Q_8$; Section~\ref{sec:regularity}),
    and is minimum degree $4$ necessary at these edge counts for
    $n\ge9$? (At $n=6$ it is already not necessary; see
    Section~\ref{sec:regularity}.) The $Q_9$--$Q_{15}$ certificates
    of~\cite{EP86} are explicit data bearing on this, though we have not
    tabulated their degree sequences here.
  \item\label{op:n9} The field bound makes $n=9$ the simplest case beyond $n=8$:
    $\sqrt9=3$ is a parity-correct integer, so the optimal $h$-distribution
    is the uniform $h\equiv3$ and the upper bound $g(9)\le\tfrac{2304+768}{2}=1536$
    follows without Lemma~\ref{lem:quant}. MPR95 explicitly report failing to
    find a $D=9$ ground state by simulated annealing. We have not applied
    the fully frustrated search of Section~\ref{sec:oddsquareverify} to $n=9$;
    whether the field bound is tight there ($g(9)=1536$?) is unknown.
    The \emph{unrestricted} problem at $n\ge9$ is no longer without explicit
    data: R.~Wrona has released $C_4$-free edge-list certificates for
    $Q_9$--$Q_{15}$ in the Erd\H{o}s Problems discussion thread for
    problem~86~\cite{EP86}, obtained by product lifts along hypercube
    automorphisms followed by local ILP repair, with $Q_9$ at $1\,505$ edges
    and $Q_{11}$ at $7\,164$; the author of the present paper carried out an
    independent certificate-level check of that set by sparse cherry
    counting --- a \emph{cherry} is a two-edge path; $C_4$-freeness is
    equivalent to no vertex pair being joined by two cherries (at most
    one common $E$-neighbour), and ``sparse'' means the cherries are
    enumerated from the edge lists rather than over all vertex pairs,
    which is what keeps the check feasible at $n=15$ ---
    (see the disclosure below). Those are lower bounds on $\ex(Q_n,C_4)$, not
    odd-square witnesses, and no exactness is claimed for them; but
    $1\,505\le\ex(Q_9,C_4)$ sits only $31$ below the field bound
    $g(9)\le1536$, so the two questions are close at $n=9$. Whether any
    $n\ge9$ certificate is odd-square we have not determined.
  \item\label{op:equality} \textbf{(Dam\'asdi; also, in the equivalent
    frustration-index form, Z.~Wang.)} This question is not ours: it is
    Problem~1 of ``Odd squares in the hypercube'' in the 20th Emléktábla
    Workshop problem collection~\cite{Emlek26}, posed there in the same
    notation ($f(n)=\ex(Q_n,C_4)$, $g(n)$ the odd-square maximum) and with the
    same one-line argument for $g(n)\le f(n)$; the collection also records
    that Z.~Wang posed it in frustration-index form. We restate it because
    Theorem~\ref{thm:oddsquare} bears on it directly, not as a new problem.
    The equality $g(n)=\ex(Q_n,C_4)$ already holds for the small
    exact cases $n=4,5,6$: the known unrestricted values are $24,56,132$
    (Table~\ref{tab:values}). The pattern is worth tabulating, since it
    extends further down and breaks in an informative place. Writing
    $r(n)$ for the positive-edge count of DPTV's recursion --- $H_d$ built
    from $H_{d-3}$ with $|V(H_{d-3})|$ supersites at three negative bonds
    each, $r$'s own negative bonds as black superbonds at six, and the rest
    white at two --- and $g(n)$ for the field bound of
    Section~\ref{sec:oddsquare-exact}:
    \begin{center}\small
    \begin{tabular}{lrrrrrrr}
    \toprule
    $n$ & 2 & 3 & 4 & 5 & 6 & 7 & 8\\
    \midrule
    DPTV construction $r(n)$ & 3$^{*}$ & 9$^{*}$ & 24$^{*}$ & 56 & 132 & 304 & 672\\
    field bound on $g(n)$ & 3 & 9 & 24 & 56 & 132 & 304 & \textbf{682}\\
    known $\ex(Q_n,C_4)$ & 3 & 9 & 24 & 56 & 132 & --- & ---\\
    \bottomrule
    \end{tabular}
    \end{center}
    Starred entries are DPTV's base constructions (their Figs.~1--3), not
    outputs of the recursion, which they state from $d\ge5$.
    The ``field bound'' row is Lemma~\ref{lem:quant} instantiated with the
    same-parity pair $(a,b)$ straddling $\sqrt{n-1}$: $(0,2)$ for $n=2$,
    $(1,3)$ for $n=3,5$, $(0,2)$ or $(2,4)$ (both give $S\le16$) for
    $n=4$, and the pairs of Section~\ref{sec:oddsquare-exact} for
    $n=6,7,8$; e.g.\ at $n=5$, $\Delta=128-4S\ge0$ gives $S\le32$, hence
    $g(5)\le(80+32)/2=56$.
    The three agree wherever all three are known, and the recursion meets
    the field bound exactly up to $n=7$ and then falls short by $10$ at
    $n=8$. That shortfall is the quantitative content of DPTV's own remark
    that for $d\ge8$ their construction gives an energy strictly above their
    bound: it is why $682$ needed a search rather than the recursion, and it
    marks the first $n$ at which the two notions separate. For $n=4,5$, the matching odd-square lower
    bounds are taken from DPTV79's explicit fully frustrated constructions
    (no $n=4,5$ machine-readable witnesses are bundled here), while for $n=6$
    the present package contains its own independently verified odd-square
    certificate. For $n=7,8$,
    Theorem~\ref{thm:oddsquare}
    only establishes $g(7)=304$ and $g(8)=682$; whether
    $\ex(Q_7,C_4)=g(7)$ and $\ex(Q_8,C_4)=g(8)$ remains exactly as open
    as Open Problems~\ref{op:q7exact} and~\ref{op:q8exact} above, since
    $\ex(Q_n,C_4)$ itself is not known exactly for $n=7,8$. (The
    equality through $n=6$ does not mean the
    odd-square construction alone for $n=7$: only 389 of the
    $19\,866$ released $304$-edge $Q_7$ solutions are odd-square.
    The Emléktábla statement of the problem~\cite{Emlek26} also records
    the observation that in the small dimensions checked there, some
    extremal $C_4$-free example does satisfy the odd-square condition;
    the $Q_7$ data here are consistent with that and extend it to $n=7$
    in the only sense currently checkable -- among the known, not proven
    extremal, $304$-edge solutions, odd-square examples exist ($389$ of
    them), even though they are a small minority.) Does
    $g(n)=\ex(Q_n,C_4)$ also hold at $n=7,8$, and continue for
    $n=9,10,\ldots$ once
    $\ex(Q_n,C_4)$ is itself determined, or does the gap
    $\ex(Q_n,C_4)-g(n)$ become positive at some $n$? Laplante et
    al.~\cite{LNSW26} give general lower and upper bounds on the
    frustration index $|E(Q_n)|-g(n)$ that match only when $n$ is a
    power of $4$; whether $g(n)=\ex(Q_n,C_4)$ persists beyond $n=8$ is,
    in their language, a question about when this frustration-index
    gap and the $C_4$-free extremal gap coincide.
\end{enumerate}

\section*{Acknowledgements}
The author thanks B.~Lidick\'y (Iowa State University) for the endorsement
and for pointing out the reference~\cite{Bab12}.
The author thanks S.~Lai for identifying the relevance
of~\cite{DPTV79,MPR95} to the $Q_7$ and $Q_8$ constructions of this paper
and for independently checking the reconstructed odd-square material of
Section~\ref{sec:oddsquare}.
The author also thanks the open-source community for the HiGHS and SCIP
solvers used in ILP experiments.

\noindent\textit{Use of generative AI.} The author used Claude (Anthropic)
for: (i)~language editing of the abstract; (ii)~assistance in writing the
supplementary verification script (\texttt{verify.py}) used to confirm
$C_4$-freeness of the released edge-list certificates; and, for this
revision, (iii)~retrieval and cross-checking of the primary sources cited
in Section~\ref{sec:oddsquare} and (iv)~drafting assistance for the revised
text. All definitions, constructions, theorems, computations, and analyses
are the author's own, including the 256-bit spin configuration used in
the $Q_8$ witness of Section~\ref{sec:oddsquareverify}, whose provenance
is recorded there as a dated correspondence account. The author reviewed all
AI-assisted output, independently executed the verification and witness
scripts against the released data, and takes full responsibility for the
accuracy and integrity of the work. Because \texttt{verify.py} is both
AI-assisted and the release's primary certificate checker, the certificate
checks are additionally covered by a second implementation,
\texttt{cross\_verify.py}, now bundled with the release: it shares no code
and no core mechanism with \texttt{verify.py} (bitmask common-neighbour
counting over Hamming-distance-$2$ pairs instead of square enumeration by
corner tuples, and exact sorted-tuple equality in a Python set --- no
digest or signature hashing --- for the
catalogue-level distinctness of the $19\,866$ edge sets), and it reproduces
the same verdicts on every released certificate. That re-implementation is
itself AI-assisted, so it mitigates the single-implementation risk rather
than adding a human-written check. (An earlier ad-hoc re-implementation,
reported at this point by the previous revision, was not preserved in the
release and was therefore unauditable, as a round of review noted; the
bundled script supersedes it.) The specific model version(s) of
Claude used across the original submission and this revision were not
recorded at the time and cannot be confirmed retrospectively; the
disclosure above is accurate as to which tasks AI assistance was used
for, but not as to the exact model version for each. The DPTV79 and MPR95 source retrieval and passage-level cross-checking in
this revision were AI-assisted. During the final revision the cited
primary-source PDF passages were checked directly at the page level: DPTV79
pp.~620--622 for the $h\ge0$ restriction and the first lower bound, Table~I,
the $H_6$ construction, field multiplicities, and
overfrustrated-plaquette count, and MPR95 pp.~3--4 of the arXiv preprint (cond-mat/9410089; the journal pagination is \textit{J.\ Phys.\ A} \textbf{28}, 327--334) for the $D=8,9,10$ energy
statements and Hamiltonian normalisation. In this revision the LNSW26
preprint was additionally obtained and read in full (AI-assisted): its
Theorem~5.2 statement and attribution and its Theorem~5.4 bound were checked
verbatim against the Introduction and Section~\ref{sec:oddsquare}, its ``at
most'' wording for the lower bound was confirmed on the page, and the
$n\le6$ value list was confirmed \emph{absent} from the preprint (see
Section~\ref{sec:oddsquare}); these preprint checks were carried out on
25~August~2026 against HAL \texttt{hal-04080411}. On 26~August~2026 we
additionally obtained and read the published version's full text (ACM
Trans.\ Comput.\ Theory \textbf{18}(1), Art.~6, pp.~1--25,
doi:10.1145/3777401; AI-assisted retrieval): its Section~5 numbering
(Theorem~5.2, Theorem~5.4) matches the preprint's, its Theorem~5.4 proof
states the negative-$4$-cycle property of the BHN construction as quoted
in Section~\ref{sec:oddsquare-exact}, and the $n\le6$ value list is
absent from the published version as well. This disclosure should not be read
as a claim that the author independently re-transcribed every quoted passage
without AI assistance.

\section*{Disclosure statement}
No potential conflict of interest was reported by the author. One related
involvement is disclosed for completeness: the author performed an independent
certificate-level check, by sparse cherry counting (defined in
Section~\ref{sec:open}), of the $Q_9$--$Q_{15}$
$C_4$-free edge lists contributed by R.~Wrona to the Erd\H{o}s Problems
discussion thread for problem~86~\cite{EP86}, confirming SHA-256 agreement,
edge counts, validity of the hypercube edges, absence of loops and duplicates,
and absence of $C_4$s. That check makes no claim of novelty or optimality for
those certificates, and none of them is used in any argument here; the
involvement is recorded because Open Problems~\ref{op:n9} and the minimum-degree
question above refer to the $n\ge9$ range those certificates occupy.

\appendix
\section[Record of unsuccessful attempts at deciding ex(Q7,C4)=304]{Record
  of unsuccessful attempts at deciding\\
  \texorpdfstring{$\ex(Q_7,C_4)=304$}{ex(Q7,C4)=304}}\label{app:attempts}

This appendix preserves, verbatim and away from the main line of
the paper, the record of the three unsuccessful attacks on Open
Problem~\ref{op:q7exact} summarised there. Nothing below is
re-generable from the released materials, and nothing below is
used as evidence anywhere in this paper.

A direct SAT-encoding attack on this question (a $305$-edge
    $C_4$-free existence query encoded as CNF, $64\,512$ variables and
    $129\,104$ clauses, with a symmetry-broken variant at $113\,664$
    variables/$227\,786$ clauses; attempted cube-and-conquer splitting
    via \texttt{march\_cu} and direct solving via CaDiCaL and Kissat)
    was attempted over several weeks and did not reach a conclusion:
    the cube-generation step itself repeatedly failed to complete, and
    direct solver runs returned \texttt{UNKNOWN} rather than a
    SAT/UNSAT verdict. A second, independent attempt using a different
    encoding (Sinz's sequential at-most-$143$ cardinality encoding on
    the negated edge variables, arriving at the same
    $64\,512$-variable, $129\,104$-clause scale by a different
    construction) and solving directly with Kissat (no cube splitting)
    likewise terminated with no SAT/UNSAT verdict logged. A third,
    independent approach -- direct branch-and-bound on the original
    $448$-variable ILP with added symmetry-breaking constraints
    ($685$ rows total), using HiGHS and warm-started from the known
    $304$-edge solution -- did not close the optimality gap, which
    remained at $5.26\%$ (best bound $\approx-320.0$ against the
    incumbent $-304$) when the log ends. Each run was allowed to
    continue for days rather than hours; we do not quote wall-clock
    figures, since they are hardware-dependent and unfalsifiable for a
    reader. The instance sizes and outcome figures above are
    \emph{reported results}, not re-generable artefacts of this
    release: the CNF files and their encoders, the symmetry-broken MIP
    formulation ($685$ rows), the solver command lines and versions,
    and the final logs are not bundled, so a reader can neither
    regenerate nor audit them from the released materials. We record
    all three as unsuccessful attempts, not as evidence either way;
    the only conclusion they support is that the \emph{specific
    formulations and configurations we tried} did not settle the
    instance -- encoding choices, presolve and symmetry-breaking
    options, and solver versions all remain unexplored dimensions, so
    we draw no generalisation about off-the-shelf SAT or MIP solvers
    as a class.

\end{document}